\documentclass[11pt]{article}

\usepackage{amssymb,amsmath,amscd,amsthm,xy,enumitem,mathabx,mathrsfs,stmaryrd}
\usepackage{pstricks,pst-node}
\xyoption{all}

\newtheorem{thm}{Theorem}[section]
\newtheorem{prp}[thm]{Proposition}
\newtheorem{lmm}[thm]{Lemma}
\newtheorem{crl}[thm]{Corollary}

\theoremstyle{definition}
\newtheorem{dfn}[thm]{Definition}

\theoremstyle{remark}
\newtheorem{rmk}[thm]{Remark}

\numberwithin{equation}{section}

\def\lra{\longrightarrow}

\def\BE#1{\begin{equation}\label{#1}}
\def\EE{\end{equation}}
\def\eref#1{(\ref{#1})}
\def\lr#1{\langle#1\rangle}
\def\flr#1{\left\lfloor{#1}\right\rfloor}
\def\blr#1{\big\langle#1\big\rangle}
\def\wt#1{\widetilde{#1}}
\def\ov#1{\overline{#1}}
\def\tn#1{\textnormal{#1}}
\def\sf#1{\textsf{#1}}
\def\wh#1{\widehat{#1}}
\def\sm#1{\begin{small}#1\end{small}}
\def\coeff#1{\llbracket{#1}\rrbracket}
\def\ch#1{\check{#1}}

\def\De{\Delta}
\def\Ga{\Gamma}
\def\La{\Lambda}
\def\Si{\Sigma}

\def\al{\alpha}
\def\de{\delta}
\def\ga{\gamma}
\def\io{\iota}
\def\ka{\kappa}
\def\om{\omega}
\def\si{\sigma}

\def\ve{\varepsilon}
\def\vph{\varphi}
\def\vp{\varpi}
\def\vr{\varrho}
\def\vt{\vartheta}

\def\cA{\mathcal A}
\def\fB{\mathfrak B}
\def\C{\mathbb C}
\def\cC{\mathcal C}  
\def\cD{\mathcal{D}}
\def\sD{\mathscr D}
\def\E{\tn{E}}
\def\sE{\mathscr E}
\def\fF{\mathfrak F}
\def\sF{\mathscr F}
\def\bF{\mathbb F}
\def\sG{\mathscr{G}}
\def\cH{\mathcal H}
\def\bI{\mathbb I}
\def\sI{\mathscr I}

\def\cJ{\mathcal J}
\def\fJ{\mathfrak j}
\def\cM{\mathcal M}
\def\fM{\mathfrak M}
\def\cN{\mathcal N}
\def\P{\mathbb P}
\def\sP{\mathscr P}
\def\Q{\mathbb Q}
\def\R{\mathbb{R}}

\def\cT{\mathcal T}
\def\sT{\mathscr{T}}
\def\cU{\mathcal U}
\def\fU{\mathfrak U}
\def\nV{\tn{V}}
\def\sV{\mathscr V}
\def\fX{\mathfrak X}
\def\Z{\mathbb{Z}}
\def\cZ{\mathcal Z}

\def\a{\mathbf a}
\def\fc{\mathfrak c}
\def\fd{\mathfrak d}
\def\ff{\mathfrak f}
\def\fg{\mathfrak g}

\def\u{\mathbf u}

\def\ale{\aleph}

\def\Aut{\tn{Aut}}

\def\codim{\tn{codim}}

\def\tnd{\tn{d}}
\def\Div{\tn{Div}}

\def\ev{\tn{\sf{ev}}}
\def\Fl{\tn{Fl}}

\def\id{\tn{id}}
\def\Im{\tn{Im}}
\def\ind{\tn{ind}}

\def\Inv{\tn{Inv}}
\def\mc{\tn{mc}}

\def\st{\tn{\sf{st}}}
\def\supp{\tn{supp}}
\def\top{\textnormal{top}}

\def\Ver{\tn{Ver}}

\def\dbar{\bar\partial}
\def\prt{\partial}
\def\eset{\emptyset}
\def\i{\infty}

\def\w{\wedge}
\def\bu{\bullet}
\def\dag{\dagger}

\begin{document}

\title{Real Ruan-Tian Perturbations}
\author{Aleksey Zinger\thanks{Partially supported by NSF grant 1500875}}
\date{\today}
\maketitle

\begin{abstract}
\noindent
Ruan-Tian deformations of the Cauchy-Riemann operator  enable 
a geometric definition of (standard) Gromov-Witten invariants of
semi-positive symplectic manifolds in arbitrary genera.
We describe an analogue of these deformations compatible with our recent 
construction of real Gromov-Witten invariants in arbitrary genera.
Our approach avoids the need for an embedding of the universal curve
into a smooth manifold and systematizes the deformation-obstruction setup
behind constructions of Gromov-Witten invariants.
\end{abstract}

\tableofcontents

\section{Introduction}
\label{intro_sec}

\noindent
The introduction of $J$-holomorphic curves techniques into symplectic topology in~\cite{Gr} 
led to definitions of (complex) \sf{Gromov-Witten} (or \sf{GW-}) 
\sf{invariants} of semi-positive symplectic manifolds in genus~0 in~\cite{MS94}
and in arbitrary genera in \cite{RT,RT2} as actual counts of 
simple \hbox{$J$-holomorphic} maps.
Local versions of the inhomogeneous deformations of the $\dbar_J$-equation pioneered
in \cite{RT,RT2} were later used to endow the moduli space of (complex)
\hbox{$J$-holomorphic} maps with a so-called \sf{virtual fundamental class}
(or \sf{VFC}) in \cite{LT,FO}
and thus to define GW-invariants for arbitrary symplectic manifolds.\\

\noindent
A \sf{real symplectic manifold} $(X,\om,\phi)$ is a symplectic manifold $(X,\om)$
with a smooth involution \hbox{$\phi\!:X\!\lra\!X$} such that \hbox{$\phi^*\om\!=\!-\om$}.
Invariant signed counts of genus~0 real curves, i.e.~those preserved by~$\phi$,
were defined for \sf{semi-positive} real symplectic 4- and 6-manifolds
in~\cite{Wel4,Wel6} in the general spirit of \cite{MS94}.
The interpretation of these counts in~\cite{Sol} in the general spirit of~\cite{LT}
removed the need for the semi-positive restriction and made
them amendable to the standard computational techniques of GW-theory;
see~\cite{PSW}, for example.
Building on the perspectives in \cite{Melissa,Sol},
genus~0 real GW-invariants for many other real symplectic manifolds were later defined
in \cite{Ge2,Teh}.
The recent work \cite{RealGWsI,RealGWsII,RealGWsIII} sets up the theory of 
real GW-invariants in arbitrary genera with conjugate pairs of insertions
and in genus~1 with arbitrary point insertions in the general spirit of~\cite{LT}.
Following a referee's suggestion, we now describe these invariants in 
the spirit of~\cite{RT,RT2};
this description is more geometric and should lead more readily to a tropical perspective 
on these invariants that has proved very powerful in studying the genus~0 real
GW-invariants of~\cite{Wel4,Wel6}.\\

\noindent
A \sf{conjugation} on a complex vector bundle $V\!\lra\!X$ 
\sf{lifting} an involution~$\phi$ on~$X$ is a vector bundle involution 
\hbox{$\vph\!:V\!\lra\!V$} covering~$\phi$ 
such that the restriction of~$\vph$ to each fiber is anti-complex linear.
A \sf{real bundle pair} \hbox{$(V,\vph)\!\lra\!(X,\phi)$}   
consists of a complex vector bundle $V\!\lra\!X$ and 
a conjugation~$\vph$ on $V$ lifting~$\phi$.
For example, 
$$(TX,\tnd\phi)\lra(X,\phi) \qquad\hbox{and}\qquad
(X\!\times\!\C^n,\phi\!\times\!\fc)\lra(X,\phi),$$
where $\fc\!:\C^n\!\lra\!\C^n$ is the standard conjugation on~$\C^n$,
are real bundle pairs.
For any real bundle pair $(V,\vph)\!\lra\!(X,\phi)$, 
we denote~by
$$\La_{\C}^{\top}(V,\vph)=(\La_{\C}^{\top}V,\La_{\C}^{\top}\vph)$$
the top exterior power of $V$ over $\C$ with the induced conjugation.
A real symplectic manifold $(X,\om,\phi)$ is \sf{real-orientable} if
there exists a rank~1 real bundle pair $(L,\wt\phi)$ over $(X,\phi)$ such~that 
\BE{realorient_e}w_2(TX^{\phi})=w_1(L^{\wt\phi})^2
\qquad\hbox{and}\qquad
\La_{\C}^{\top}(TX,\tnd\phi)\approx(L,\wt\phi)^{\otimes 2}\,.\EE

\begin{dfn}\label{realorient_dfn2}
A \sf{real orientation} on a real-orientable symplectic manifold $(X,\om,\phi)$ consists~of 
\begin{enumerate}[label=(RO\arabic*),leftmargin=*]

\item\label{LBP_it} a rank~1 real bundle pair $(L,\wt\phi)$ over $(X,\phi)$
satisfying~\eref{realorient_e},

\item\label{isom_it} a homotopy class of isomorphisms of real bundle pairs 
in~\eref{realorient_e}, and

\item\label{spin_it} a spin structure~on the real vector bundle
$TX^{\phi}\!\oplus\!2(L^*)^{\wt\phi^*}$ over~$X^{\phi}$
compatible with the orientation induced by~\ref{isom_it}. 

\end{enumerate}
\end{dfn}

\vspace{.1in}

\noindent
By \cite[Theorem~1.3]{RealGWsI}, a real orientation on $(X,\om,\phi)$
orients the moduli space $\ov\fM_{g,l}(X,B;J)^{\phi}$ of
genus~$g$ degree~$B$ real $J$-holomorphic maps with $l$~conjugate pairs of  marked points
whenever the ``complex" dimension of~$X$ is odd. 
By the proof of \cite[Theorem~1.5]{RealGWsI}, it also orients 
the moduli space $\ov\fM_{1,l;k}(X,B;J)^{\phi}$ of
genus~1 maps with $k$ real marked points
outside of certain codimension~1 strata.
In general, these moduli spaces are not smooth and the above orientability statements
should be viewed in the usual moduli-theoretic (or virtual) sense.\\

\noindent
The description in~\cite{Melissa} of versal families of deformations of 
symmetric Riemann surfaces provides the necessary ingredient for 
adapting the interpretation of Gromov's topology in~\cite{LT} 
from the complex to the real setting and eliminates 
the (virtual) boundary of $\ov\fM_{g,l;k}(X,B;J)^{\phi}$.
A Kuranishi atlas for this moduli space is then obtained by carrying out
the constructions of \cite{LT,FO} in a \hbox{$\phi$-invariant} manner;
see \cite[Section~7]{Sol} and \cite[Appendix]{FOOO9}.
If oriented, this atlas determines a VFC for $\ov\fM_{g,l;k}(X,B;J)^{\phi}$
and thus gives rise to  genus~$g$ real GW-invariants of $(X,\om,\phi)$;
see \cite[Theorem~1.4]{RealGWsI}.
If this atlas is oriented only on the complement of some codimension~1 strata,
real GW-invariants can still be obtained in some special cases 
by adapting the principle of~\cite{Cho,Sol} to show that the problematic strata
are avoided by a generic path; \cite[Theorem~1.5]{RealGWsI}.
In some important situations, the real genus~$g$ GW-invariants arising from 
\cite[Theorem~1.3]{RealGWsI}  can be described as actual counts 
of curves in the spirit of \cite{RT,RT2}.\\

\noindent
For a manifold $X$, denote by 
$$H_2^S(X;\Z)\equiv\big\{u_*[S^2]\!:u\!\in\!C(S^2;X)\big\}\subset H_2(X;\Z)$$
the subset of \sf{spherical classes}. 
There are two topological types of anti-holomorphic involutions on~$\P^1$; 
they are represented~by 
$$\tau,\eta\!:\P^1\lra\P^1, \qquad z\lra\frac{1}{\bar{z}},-\frac{1}{\bar{z}}\,.$$
For a manifold $X$ with an involution~$\phi$, denote~by
\begin{gather*}
H_2^{\si}(X;\Z)^{\phi}\equiv
\big\{u_*[S^2]\!:u\!\in\!C(S^2;X),\,u\!\circ\!\si\!=\!\phi\!\circ\!u\big\}
\qquad\hbox{for}~\si\!=\!\tau,\eta, \\
H_2^{\R S}(X;\Z)^{\phi}\equiv H_2^{\tau}(X;\Z)^{\phi}\!\cup\!H_2^{\eta}(X;\Z)^{\phi}
\subset\{B\!\in\!H_2(X;\Z)\!:\phi_*B\!=\!-B\big\}
\end{gather*}
the subsets of \sf{$\si$-spherical classes} and \sf{real spherical classes}.

\begin{dfn}\label{SSP_dfn}
A symplectic  $2n$-manifold $(X,\om)$ is \sf{semi-positive}~if
$$\lr{c_1(X),B}\ge0  \qquad\forall~B\!\in\!H_2^S(X;\Z)
~~\hbox{s.t.}~\lr{\om,B}\!>\!0,~\lr{c_1(X),B}\!\ge\!3\!-\!n.$$
A real symplectic $2n$-manifold $(X,\om,\phi)$ is \sf{semi-positive}
if $(X,\om)$ is semi-positive and
\begin{alignat*}{3}
\lr{c_1(X),B}&\ge\de_{n2} &\qquad &\forall~B\!\in\!H_2^{\R S}(X;\Z)^{\phi}&~&\hbox{s.t.}~
~\lr{\om,B}\!>\!0,~\lr{c_1(X),B}\!\ge\!2\!-\!n,\\
\lr{c_1(X),B}&\ge1 &\qquad &\forall~B\!\in\!H_2^{\tau}(X;\Z)^{\phi}&~&\hbox{s.t.}~
~\lr{\om,B}\!>\!0,~\lr{c_1(X),B}\!\ge\!2\!-\!n,
\end{alignat*}
where $\de_{n2}\!=\!1$ if $n\!=\!2$ and 0 otherwise.
\end{dfn}

\noindent
The stronger middle bound in the $n\!=\!2$ case above rules out the appearance of
real degree~$B$ \hbox{$J$-holomorphic} spheres with 
$\lr{c_1(X),B}\!=\!0$ for a generic one-parameter
family of real almost complex structures on a real symplectic manifold~$(X,\om,\phi)$
and provides for the second bound in~\eref{g0emptyC_e}.
The latter in turn ensures that the expected dimension of the moduli space
of complex  degree~$B$ \hbox{$J$-holomorphic} spheres in such a family of 
almost complex structures is not smaller than
the expected dimension of the moduli space of real  degree~$B$ \hbox{$J$-holomorphic} spheres.\\

\noindent
Monotone symplectic manifolds, including all projective spaces and 
Fano hypersurfaces, are semi-positive.
The~maps
\begin{alignat*}{2}
\tau_n\!:\P^{n-1}&\lra\P^{n-1}, &\qquad [Z_1,\ldots,Z_n]&\lra[\ov{Z}_1,\ldots,\ov{Z}_n],\\
\eta_{2m}\!: \P^{2m-1}& \lra\P^{2m-1},&\qquad
[Z_1,Z_2,\ldots,Z_{2m-1},Z_{2m}]&\lra 
\big[-\ov{Z}_2,\ov{Z}_1,\ldots,-\ov{Z}_{2m},\ov{Z}_{2m-1}\big],
\end{alignat*}
are anti-symplectic involutions with respect to the standard Fubini-Study symplectic
forms~$\om_n$ on~$\P^{n-1}$ and~$\om_{2m}$ on $\P^{2m-1}$, respectively.
If 
$$k\!\ge\!0, \qquad \a\equiv(a_1,\ldots,a_k)\in(\Z^+)^k\,,$$
and $X_{n;\a}\!\subset\!\P^{n-1}$ is a complete intersection of multi-degree~$\a$
preserved by~$\tau_n$,  then $\tau_{n;\a}\!\equiv\!\tau_n|_{X_{n;\a}}$
is an anti-symplectic involution on $X_{n;\a}$ with respect to the symplectic form
$\om_{n;\a}\!=\!\om_n|_{X_{n;\a}}$. 
Similarly, if $X_{2m;\a}\!\subset\!\P^{2m-1}$ is preserved by~$\eta_{2m}$, then
$\eta_{2m;\a}\!\equiv\!\eta_{2m}|_{X_{2m;\a}}$
is an anti-symplectic involution on $X_{2m;\a}$ with respect to the symplectic form
$\om_{2m;\a}\!=\!\om_{2m}|_{X_{2m;\a}}$.
The projective spaces $(\P^{2m-1},\tau_{2m-1})$ and $(\P^{4m-1},\eta_{4m-1})$,
as well as many real complete intersections in these spaces, are real orientable;
see \cite[Proposition~2.1]{RealGWsI}.\\

\noindent
We show in this paper that the semi-positive property of Definition~\ref{SSP_dfn} 
for $(X,\om,\phi)$ plays
the same role in the real GW-theory as the 
semi-positive property for~$(X,\om)$ plays in ``classical" GW-theory.
For each element~$(J,\nu)$ of the space~\eref{cHdfn_e},
the moduli space $\ov\fM_{g,l;k}(X,B;J,\nu)^{\phi}$ 
of stable degree~$B$ genus~$g$ real 
$(J,\nu)$-maps with~$l$ conjugate pairs of marked points and $k$~real points
is coarsely stratified by the subspaces~$\fM_{\ga}(J,\nu)^{\phi}$
of maps of the same combinatorial type; see~\eref{fMgadfn_e}.
By Proposition~\ref{RTreg_prp}, the open subspace
$$\fM_{\ga}^*(J,\nu)^{\phi}\subset\fM_{\ga}(J,\nu)^{\phi}$$
consisting of simple maps in the sense of Definition~\ref{Jnumap_dfn2}
is cut out transversely by the $\{\dbar_J\!-\!\nu\}$-operator 
for a generic pair~$(J,\nu)$;
thus, it is smooth and of the expected dimension.
The image~of
$$\fM_{\ga}^{\mc}(J,\nu)^{\phi}\equiv  \fM_{\ga}(J,\nu)^{\phi}-\fM_{\ga}^*(J,\nu)^{\phi}$$
under the product of the stabilization~$\st$ and the evaluation map~$\ev$ in~\eref{evst_e} 
is covered by smooth maps from finitely many spaces $\fM_{\ga'}^*(J,\nu')^{\phi}$
of simple degree~$B'$ genus~$g'$  real maps with \hbox{$\om(B')\!<\!\om(B)$} and $g'\!\ge\!g$.
By the proof of Proposition~\ref{RTdim_prp}, the dimensions of the latter spaces 
are at least~2 less than 
the virtual dimension of $\ov\fM_{g,l;k}(X,B;J,\nu)^{\phi}$ if
 $(J,\nu)$ is generic and $(X,\om,\phi)$ is semi-positive.\\

\noindent
By Theorem~\ref{RTreal_thm}\ref{RTexist_it}, 
the restriction~\eref{RTreal_e} of~\eref{evst_e} is a pseudocycle 
for a generic pair~$(J,\nu)$ in the space~\eref{cHdfn_e} 
whenever $(X,\om,\phi)$ is a  semi-positive real symplectic manifold 
of odd ``complex" dimension with a real orientation.
Intersecting this pseudocycle with constraints 
in the Deligne-Mumford moduli space $\R\ov\cM_{g,l}$ of real curves
and in~$X$,
we obtain an interpretation of the genus~$g$  real GW-invariants 
provided by \cite[Theorem~1.4]{RealGWsI}
as  counts of real $(J,\nu)$-curves in $(X,\om,\phi)$ which depend only
on the homology classes of the constraints.
A similar conclusion applies to the genus~1 real GW-invariants 
with real marked points provided by \cite[Theorem~1.5]{RealGWsI};
see Remark~\ref{RealRT3_rmk}.\\

\noindent
For the purposes of Theorem~\ref{RTreal_thm}\ref{RTexist_it},
the $2\!-\!n$ inequalities in Definition~\ref{SSP_dfn} could be replaced by $3\!-\!n$
(which would weaken~it).
This would make its restrictions vacuous if $n\!=\!2$, i.e.~\hbox{$\dim_{\R}X\!=\!4$}.
The $2\!-\!n$ condition ensures that the conclusion of Proposition~\ref{RTdim_prp}
remains valid for a generic one-parameter family of elements~$(J,\nu)$ in the space~\eref{cHdfn_e}
and that the homology class determined by the pseudocycle~\eref{RTreal_e} is independent
of the choice of~$(J,\nu)$; see Proposition~\ref{RTdim2_prp} and 
the first statement of Theorem~\ref{RTreal_thm}\ref{RTindep_it}.
For $n\!=\!2$, this condition excludes the appearance of stable maps represented
by the two diagrams of Figure~\ref{c1B_fig} for a generic one-parameter family
of~$(J,\nu)$.
The maps of the first type are not regular solutions of the $(\dbar_J\!-\!\nu)$-equation
if $\vr\!\ge\!2$.
The maps of the second type are not regular solutions of the $(\dbar_J\!-\!\nu)$-equation
if the images of the top and bottom irreducible components are the same 
(i.e.~each of them is preserved by~$\phi$) and $\vr\!\in\!\Z^+$.
If maps of either type exist, their images under $\st\!\times\!\ev$ form a subspace 
of real codimension~1 in the image of~\eref{evst_e}.

\begin{figure}
\begin{pspicture}(-3,-1.5)(10,2)
\psset{unit=.3cm}
\psellipse[linewidth=.08](5,0)(4,2)
\psarc[linewidth=.08](3.2,-1.2){1.5}{45}{135}
\psarc[linewidth=.08](3.2,1.2){1.5}{-135}{-45}
\psarc[linewidth=.08](6.8,-1.2){1.5}{45}{135}
\psarc[linewidth=.08](6.8,1.2){1.5}{-135}{-45}
\psline[linewidth=.05]{<->}(-3.5,-3)(-3.5,3)\rput{90}(-4,0){\sm{$\si$}}
\pscircle[linewidth=.08](-1,0){2}\pscircle*(1,0){.3}
\rput(-1,0){\sm{$\vr B_0'$}}\rput(10,0){\sm{$B_0$}}
\rput(3,-4){\sm{$\lr{c_1(X),B_0'}\!=\!0,~\vr B_0'\!+\!B_0\!=\!B,~\vr\!\ge\!2$}}
\psellipse[linewidth=.08](20,0)(4,2)
\psarc[linewidth=.08](18.2,-1.2){1.5}{45}{135}
\psarc[linewidth=.08](18.2,1.2){1.5}{-135}{-45}
\psarc[linewidth=.08](21.8,-1.2){1.5}{45}{135}
\psarc[linewidth=.08](21.8,1.2){1.5}{-135}{-45}
\pscircle[linewidth=.08](20,4){2}\pscircle*(20,2){.3}
\pscircle[linewidth=.08](20,-4){2}\pscircle*(20,-2){.3}
\psline[linewidth=.05]{<->}(15.5,-3)(15.5,3)\rput{90}(15,0){\sm{$\si$}}
\rput(25,0){\sm{$B_0$}}\rput(23.3,4){\sm{$\vr B^+$}}\rput(23.3,-4){\sm{$\vr B^-$}}
\psline[linewidth=.05]{<->}(19,5)(21,3)\psline[linewidth=.05]{<->}(19,-5)(21,-3)
\rput(33,0){\begin{tabular}{l}$\phi_*B^+\!=\!-B^-$\\ $\lr{c_1(X),B^{\pm}}\!=\!0$\\
$\vr B^+\!+\!B_0\!+\!\vr B^-\!=\!B$\\ $\vr\!\in\!\Z^+$\end{tabular}}
\end{pspicture}  
\caption{Typical elements of subspaces of $\fM_{\ga}^{\mc}(\al)^{\phi}$
with codimension-one images under $\st\!\times\!\ev$
for a generic one-parameter family~$\al$ of real Ruan-Tian deformations~$(J,\nu)$
on a real symplectic 4-manifold~$(X,\om,\phi)$.
The degrees of the maps on the irreducible components of the domains are
shown next to the corresponding components.
The double-headed arrows labeled by~$\si$ indicate the involutions 
on the entire domains of the maps.
The smaller double-headed arrows indicate the involutions 
on the real images of the corresponding irreducible components of the~domain.} 
\label{c1B_fig}
\end{figure}
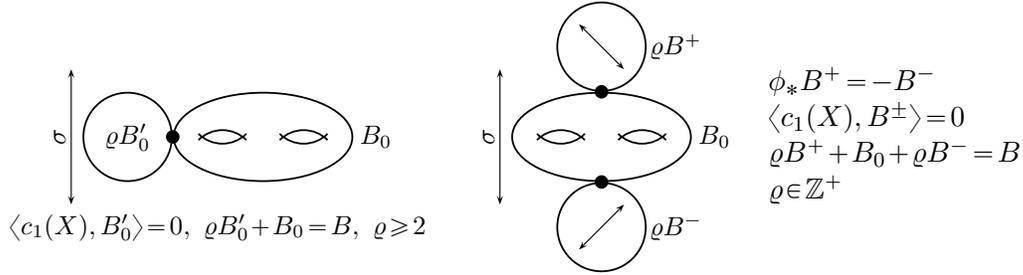

\begin{rmk}\label{SSP_rmk}
Real symplectic 4-manifolds $(X,\om,\phi)$ with classes 
$B\!\in\!H_2^{\R S}(X;\Z)^{\phi}$ such that \hbox{$\lr{\om,B}\!>\!0$} and
$\lr{c_1(X),B}\!=\!0$ are not excluded from the  constructions
of genus~0 real GW-invariants in \cite{Wel4,Sol}.
However, the geometric proofs of the invariance of the curve counts defined in these papers
neglect to consider stable maps as in Figure~\ref{c1B_fig}.
The second Hirzebruch surface \hbox{$\bF_2\!\lra\!\P^1$} contains two natural section
classes, $C_2$ and~$E$, with normal bundles of degrees~$2$ and~$-2$, respectively.
Along with the fiber class~$F$, either of them generates $H_2(\bF_2;\Z)$.
There are algebraic families \hbox{$p\!:\cC\!\lra\!S$} and \hbox{$\pi\!:\fX\!\lra\!S$},
where $S$ is a neighborhood of $0\!\in\!\C$, such~that
$$p^{-1}(0)=\P^1\!\vee\!\P^1, \qquad \pi^{-1}(0)=\bF_2, \qquad
p^{-1}(z)=\P^1, ~~\pi^{-1}(z)=\P^1\!\times\!\P^1 \quad\forall~z\!\in\!S\!-\!\{0\}.$$
The projection~$\pi$ can be viewed as an algebraic family of algebraic structures on~$\bF_2$.
By \cite[Proposition~3.2.1]{AB}, a morphism~$f$ from~$p^{-1}(0)$
of degree \hbox{$D\!=\!aC_2\!+\!bF$}, with $a\!\in\!\Z^+$ and $b\!\in\!\Z^{\ge0}$,
to~$\pi^{-1}(0)$ that passes through $4a\!+\!2b\!-\!1$ general points in~$\bF_2$ and extends to
a morphism \hbox{$\wh{f}\!:\cC\!\lra\!\fX$} restricts to an isomorphism 
from a component of~$p^{-1}(0)$ to~$E$. 
The end of the proof of \cite[Proposition~2.9]{Wel15} cites~\cite{AB} 
as establishing this conclusion for 
a generic one-parameter family of real almost complex structures~$J$ on
blowups of~$(\bF_2,E)$ away from~$E$.
This is used to claim that multiply covered disk bubbles of Maslov index~0
do not appear in a one-parameter family of almost complex structures in the proof
of \cite[Theorem~0.1]{Wel15} and that maps as in the first diagram of Figure~\ref{c1B_fig}
do not appear in the proof of \cite[Theorem~0.1]{Wel4}; see \cite[Remark~2.12]{Wel15}.
The potential appearance of  maps as in the second diagram of Figure~\ref{c1B_fig}
is not even discussed in any geometric argument we are aware~of.
On the other hand, these maps create no difficulties in the virtual class approach
of \cite[Section~7]{Sol}.
\end{rmk}

\noindent
Section~\ref{TermNot_sec} sets up the relevant notation for the moduli spaces
of complex and real curves and for their covers.
Section~\ref{realRTdfn_subs} introduces a real version of the perturbations of~\cite{RT2} 
and concludes with the main theorem.
The strata $\fM_{\ga}(J,\nu)^{\phi}$ splitting the moduli space 
$\ov\fM_{g,l;k}(X,B;J,\nu)^{\phi}$ based on the combinatorial type of the map
are described in Section~\ref{RTstr_subs}.
As summarized in Section~\ref{RTstr_subs2}, the subspaces $\fM_{\ga}^*(J,\nu)^{\phi}$
of 
these strata consisting of simple maps are smooth manifolds.
We use the regularity statements of this section, 
Propositions~\ref{RTreg_prp} and~\ref{RTreg2_prp}, to establish the main
theorem in Section~\ref{RTrealpf_subs}.
The two propositions are proved
in Sections~\ref{DefObs_subs} and~\ref{RTregpf_subs}. 
The first of these sections introduces suitable deformation-obstruction settings and 
then shows that the deformations of real Ruan-Tian pairs~$(J,\nu)$ suffice to
cover the obstruction space in all relevant cases; see Lemmas~\ref{RTtransC_lmm}
and~\ref{RTtransR_lmm}.
By Section~\ref{RTregpf_subs}, Lemmas~\ref{RTtransC_lmm} and~\ref{RTtransR_lmm} 
ensure the smoothness 
of the universal moduli space of simple $(J,\nu)$-maps from a domain of each topological type;
see Theorem~\ref{UfM_prp}.
As is well-known, the latter implies the smoothness of the corresponding stratum of
the moduli space of simple $(J,\nu)$-maps for a generic pair~$(J,\nu)$ and thus concludes
the proof of Proposition~\ref{RTreg_prp}.
In the process of establishing Theorem~\ref{RTreal_thm}, 
we systematize and streamline the constructions of GW-pseudocycles 
in~\cite{MS,RT2}.\\

\noindent
The author would like to thank P.~Georgieva and J.~Starr for enlightening discussions on 
the Deligne-Mumford moduli of curves.

\section{Terminology and notation}
\label{TermNot_sec}

\noindent
Ruan-Tian's deformations~$\nu$ are obtained by passing to a \sf{regular} 
cover of the Deligne-Mumford space $\ov\cM_{g,l}$ of stable genus~$g$ complex curves
with $l$~marked points.
After recalling such covers in Section~\ref{CDM_subs},  
we describe their analogues suitable for real GW-theory.

\subsection{Moduli spaces of complex curves}
\label{CDM_subs}

\noindent
For $l\!\in\!\Z^{\ge0}$, let
$$[l]\equiv\big\{i\!\in\!\Z^+\!:i\le\!l\big\}.$$
For $g\!\in\!\Z^{\ge0}$, we denote by $\cD_g$ the group of diffeomorphisms 
of a smooth compact connected orientable genus~$g$ surface~$\Si$
and by $\cJ_g$ the space of complex structures on~$\Si$.
If in addition $l\!\in\!\Z^{\ge0}$ and $2g\!+\!l\!\ge\!3$, let 
$$\cM_{g,l}\subset \ov\cM_{g,l}$$
be the open subspace of smooth curves in the Deligne-Mumford moduli space of 
genus~$g$ complex curves with $l$~marked points and define
$$\cJ_{g,l}=\big\{(\fJ,z_1,\ldots,z_l)\!\in\!\cJ_g\!\times\!\Si^l\!:
z_i\!\neq\!z_j~\forall\,i\!\neq\!j\big\}.$$
The group~$\cD_g$ acts on~$\cJ_{g,l}$ by 
$$h\cdot\big(\fJ,z_1,\ldots,z_l\big)=\big(h^*\fJ,h^{-1}(z_1),\ldots,h^{-1}(z_l)\big)\,.$$

\vspace{.2in}

\noindent
Denote by $\sT_{g,l}$ the Teichm\"uller space of~$\Si$ with $l$ punctures and
by $\sG_{g,l}$ the corresponding mapping class group.
Thus,
\BE{cMglquot_e}\cM_{g,l}=\cJ_{g,l}\big/\cD_g=\sT_{g,l}\big/\sG_{g,l}\,.\EE
Let 
\BE{ffgldfn_e}\ff_{g,l}\!: \ov\cU_{g,l}\!=\!\ov\cM_{g,l+1}\lra \ov\cM_{g,l}\EE
be the forgetful morphism dropping the last marked point;
it determines the \sf{universal family} over~$\ov\cM_{g,l}$.\\

\noindent
For a tuple $\sD\!\equiv\!(g_1,S_1;g_2,S_2)$ consisting of $g_1,g_2\!\in\!\Z^{\ge0}$ 
with \hbox{$g\!=\!g_1\!+\!g_2$} and  $S_1,S_2\!\subset\![l]$ with 
\hbox{$[l]\!=\!S_1\!\sqcup\!S_2$}, denote~by
$$\ov\cM_{\sD}\subset\ov\cM_{g,l}$$
the closure of the subspace of marked curves with two irreducible components~$\Si_1$
and~$\Si_2$ of genera~$g_1$ and~$g_2$, respectively, and carrying the marked points indexed
by~$S_1$ and~$S_2$, respectively.
Let 
$$\io_{\sD}\!: \ov\cM_{g_1,|S_1|+1}\!\times\ov\cM_{g_2,|S_2|+1}
\lra \ov\cM_{g,l}$$
be the natural node-identifying immersion with image  $\ov\cM_{\sD}$
(it sends the first $|S_i|$ marked points
of the $i$-th factor to the marked points indexed by~$S_i$ in the order-preserving fashion).
We denote by $\Div_{g,l}$ the set of tuples~$\sD$ above.\\

\noindent
For each involution~$\si$ on the~set $[l]$,
define 
\begin{alignat*}{2}
\wt\si\!:\big[l\!+\!1\big]&\lra\big[l\!+\!1\big], &\qquad
\wt\si(i)&=\begin{cases}\si(i),&\hbox{if}~i\!\in\![l];\\
i,&\hbox{if}~i\!=\!l\!+\!1; \end{cases}\\
\si_g\!:\cJ_{g,l}&\lra\cJ_{g,l}, &\qquad
\si_g\big(\fJ,z_1,\ldots,z_l\big)&=\big(-\fJ,z_{\si(1)},\ldots,z_{\si(l)}\big).
\end{alignat*}
Since the last involution commutes with the action of~$\cD_g$, it descends to 
an involution on the quotient~\eref{cMglquot_e}.
The latter extends to an involution 
\BE{sigldfn_e}\si_g\!:\ov\cM_{g,l}\lra \ov\cM_{g,l} \qquad\hbox{s.t.}\quad
\si_g\!\circ\!\ff_{g,l}=\ff_{g,l}\!\circ\!\wt\si_g\,.\EE
A genus~$g$ complex curve~$\cC$ is cut out by polynomial equations in some~$\P^{N-1}$
($N$ can be taken to be the same for all elements of~$\ov\cM_{g,l}$).
The standard involution~$\tau_N$ on~$\P^{N-1}$ sends~$\cC$ to another genus~$g$ curve~$\ov\cC$.
If~$\cC$ is smooth, $\tau_N$ identifies~$\cC$ and~$\ov\cC$ as smooth surfaces reversing 
the complex structure.
The conjugation~$\tau_N$ thus induces the involution~\eref{sigldfn_e}.\\

\noindent
Since~$\sT_{g,l}$ is simply connected, the involution~$\si_g$ on~$\ov\cM_{g,l}$ 
lifts to a $\sG_{g,l}$-equivariant involution
\BE{RUClmm_e3}\si_g\!:\sT_{g,l}\lra\sT_{g,l}\,.\EE
Such a lift can be described as follows. 
Let $\Si_{g,l}$ be a smooth compact connected oriented genus~$g$ surface~$\Si$ 
with $l$~distinct marked points $z_1,\ldots,z_l$ and 
$\cD_{g,l}\!\subset\!\cD_g$ be the subgroup of diffeomorphisms of~$\Si_{g,l}$
isotopic to the identity (and preserving the marked points).
Choose an orientation-reversing involution~$\si_{g,l}$ on~$\Si_{g,l}$ 
that restricts to~$\si$ on the marked points.
An element of~$\sT_{g,l}$ is the \hbox{$\cD_{g,l}$-orbit~$[\fJ]$} of an element $\fJ\!\in\!\cJ_g$ 
compatible with the orientation of~$\Si_{g,l}$.
A lift as in~\eref{RUClmm_e3} can be obtained by defining
$$\si_g\!:\sT_{g,l}\lra\sT_{g,l}, \qquad [\fJ]\lra \big[-\si_{g,l}^{\,*}\fJ\big]\,.$$
This description is standard in the analytic perspective on the moduli spaces of curves;
see \cite[Section~2]{SepSil}, for example.

\begin{dfn}\label{RUC_dfn1}
Let $g,l\!\in\!\Z^{\ge0}$ with $2g\!+\!l\!\ge\!3$ and 
\BE{RUCdfn_e}p\!:\wt\cM_{g,l}\lra\ov\cM_{g,l}\EE
be a finite branched cover in the orbifold category.
A \sf{universal curve over $\wt\cM_{g,l}$} is a tuple
$$ \big(\pi\!:\wt\cU_{g,l}\!\lra\!\wt\cM_{g,l},s_1,\ldots,s_l\big),  $$
where $\wt\cU_{g,l}$ is a projective variety and $\pi$ is a projective morphism
with disjoint sections $s_1,\ldots,s_l$, such~that for each $\wt\cC\!\in\!\wt\cM_{g,l}$
the~tuple
$(\pi^{-1}(\wt\cC),s_1(\wt\cC),\ldots,s_l(\wt\cC))$
is a stable genus~$g$ curve with $l$~marked points whose equivalence class is~$p(\wt\cC)$.
\end{dfn} 

\begin{dfn}\label{RUC_dfn2}
Let $g,l\!\in\!\Z^{\ge0}$ with $2g\!+\!l\!\ge\!3$.
A cover~\eref{RUCdfn_e} is \sf{regular} if 
\begin{enumerate}[label=$\bu$,leftmargin=*]

\item it admits a universal curve,

\item  each topological component of $p^{-1}(\cM_{g,l})$
is the quotient of~$\sT_{g,l}$ by a 
 subgroup of~$\sG_{g,l}$, and  

\item for every element $\sD\!\equiv\!(g_1,S_1;g_2,S_2)$ of $\Div_{g,l}$,
$$\big(\ov\cM_{g_1,|S_1|+1}\!\times\ov\cM_{g_2,|S_2|+1}\big)
\times_{(\io_{\sD},p)}\wt\cM_{g,l}\approx 
\wt\cM_{g_1,|S_1|+1}\!\times\!\wt\cM_{g_2,|S_2|+1}$$
for some covers $\wt\cM_{g_i,|S_i|+1}$ of $\ov\cM_{g_i,|S_i|+1}$.

\end{enumerate}
\end{dfn} 

\vspace{.2in}

\noindent
The moduli space $\ov\cM_{0,l}$ is smooth and the universal family over it 
satisfies the requirement of Definition~\ref{RUC_dfn1}.
For $g\!\ge\!2$, \cite[Theorems~2.2,3.9]{BP} provide
covers~\eref{RUCdfn_e} satisfying 
the last two requirements of Definition~\ref{RUC_dfn2}
so that the orbifold fiber product
\BE{CwtcU_e} \pi\!:\wt\cU_{g,l}\!
\equiv\!\wt\cM_{g,l}\!\otimes_{\ov\cM_{g,l}}\!\ov\cU_{g,l}\lra\wt\cM_{g,l}\EE
satisfies the requirement of Definition~\ref{RUC_dfn1};
see also \cite[Section~2.2]{Ruan}.
The same reasoning applies in the $g\!=\!1$ case if \hbox{$l\!\ge\!1$}.

\begin{lmm}\label{RUC_lmm}
If \eref{RUCdfn_e} satisfies the second condition in Definition~\ref{RUC_dfn2}
and $\si$ is an involution on~$[l]$,
then the involutions~$\si_g$ on~$\ov\cM_{g,l}$ and~$\wt\si_g$ on~$\ov\cU_{g,l}$ 
lift to involutions
\BE{RUClmm_e}\si_g\!: \wt\cM_{g,l}\lra\wt\cM_{g,l}, ~~\wt\si_g\!:\wt\cU_{g,l}\lra \wt\cU_{g,l}
\qquad\hbox{s.t.}~~\si_g\!\circ\!\pi=\pi\!\circ\!\wt\si_g\,.\EE
\end{lmm} 

\begin{proof}
If~\eref{RUCdfn_e} satisfies the second condition in Definition~\ref{RUC_dfn2},
then the involution~\eref{RUClmm_e3} descends to an involution on~$p^{-1}(\cM_{g,l})$.
Since every point $[\cC]\!\in\!\ov\cM_{g,l}$ has an arbitrary small neighborhood~$U_{\cC}$
such that $U_{\cC}\!\cap\!\cM_{g,l}$ is connected and dense in~$U_{\cC}$,
the last involution extends to an involution~$\si_g$ as in~\eref{RUClmm_e}. 
By the identity in~\eref{sigldfn_e}, the involution~$\wt\si_g$ on~$\ov\cU_{g,l}$ 
lifts to an involution~$\wt\si_g$ as in~\eref{RUClmm_e} over the projection 
$\wt\cU_{g,l}\!\lra\!\ov\cU_{g,l}$
so that the identity in~\eref{RUClmm_e} holds.
\end{proof}

\subsection{Moduli spaces of real curves}
\label{RDM_subs}
 
\noindent
A \sf{symmetric surface} $(\Si,\si)$ is a nodal compact connected orientable surface~$\Si$ 
(manifold of real dimension~2 with distinct pairs of points identified) 
with an orientation-reversing involution~$\si$.
If $\Si$ is smooth, then the fixed locus~$\Si^{\si}$ of~$\si$ is a disjoint union of circles.
There are $\flr{\frac{3g+4}{2}}$ different topological types of orientation-reversing 
involutions~$\si$ on a smooth surface~$\Si$; see \cite[Corollary~1.1]{Nat}.
We denote the set of these types by~$\sI_g^-$.\\

\noindent
For an orientation-reversing involution~$\si$ on a smooth compact connected 
orientable genus~$g$ surface~$\Si$, let
$$\cD_g^{\si}=\big\{h\!\in\!\cD_g\!:h\!\circ\!\si\!=\!\si\!\circ\!h\big\}, \qquad
\cJ_g^{\si}=\big\{\fJ\!\in\!\cJ_g\!:\si^*\fJ\!=\!-\fJ\big\}.$$
If in addition $l,k\!\in\!\Z^{\ge0}$, define
$$\cJ_{g,l;k}^{\si}=\big\{\big(\fJ,
(z_i^+,z_i^-)_{i\in[l]},(z_i)_{i\in[k]}\big)\!\in\!\cJ_{g,2l+k}\!:
\fJ\!\in\!\cJ_g^{\si},\,\si(z_i^{\pm})\!=\!z_i^{\mp}~\forall i\!\in\![l],\,
\si(z_i)\!=\!z_i~\forall i\!\in\![k]\big\}.$$
An element of $\cJ_{g,l;k}^{\si}$ is a smooth \sf{real curve} 
of genus~$g$ with  $l$~conjugate pairs of marked points and 
$k$~real marked points.\\

\noindent
The action of $\cD_g$ on $\cJ_{g,2l+k}$ restricts to an action of $\cD_g^{\si}$ on~$\cJ_{g,l;k}^{\si}$.
Let
$$\cM_{g,l;k}^{\si}\equiv\cJ_{g,l;k}^{\si}\big/\cD_g^{\si}\,.$$
If $2(g\!+\!l)\!+\!k\!\ge\!3$, the Deligne-Mumford moduli space $\R\ov\cM_{g,l;k}$ 
of real genus~$g$ curves with $l$~conjugate pairs of marked points and 
$k$~real marked points is a compactification~of 
$$\R\cM_{g,l;k}\equiv\bigsqcup_{\si\in\sI_g^-}\!\!\!\cM_{g,l;k}^{\si}$$
with strata of equivalence classes of stable nodal real curves
of genus~$g$ with  $l$~conjugate pairs of marked points and $k$~real marked points.
This moduli space is topologized via versal deformations of real curves
as described in \cite[Section~3.2]{Melissa}.\\

\noindent
Fix $g,l\!\in\!\Z^{\ge0}$ and define
$$\si\!:\big[2l\!+\!k\big]\lra \big[2l\!+\!k\big], \qquad
\si(i)=\begin{cases}i\!+\!1,&\hbox{if}~i\!\le\!2l,\,i\!\not\in\!2\Z;\\
i\!-\!1,&\hbox{if}~i\!\le\!2l,\,i\!\in\!2\Z;\\
i,&\hbox{if}~i\!>\!2l.
\end{cases}$$
There is a natural morphism
\BE{RtoCmap_e} \R\ov\cM_{g,l;k} \lra \ov\cM_{g,2l+k}\,.\EE
A genus~$g$ symmetric surface $(\Si,\si_{\Si})$ is cut out by real polynomial equations 
in some~$\P^{N-1}$ so that $\si_{\Si}\!=\!\tau_N|_{\Si}$.
The morphism~\eref{RtoCmap_e} sends the equivalence class of $(\Si,\si_{\Si})$
to the equivalence class of~$\Si$.
The image of~\eref{RtoCmap_e} is contained in the fixed locus $\ov\cM_{g,2l+k}^{\,\si_g}$
of the involution~$\si_g$.
For $g\!=\!0$, \eref{RtoCmap_e} is an isomorphism onto~$\ov\cM_{g,2l+k}^{\,\si_g}$.
In general, \eref{RtoCmap_e} is neither injective nor surjective onto~$\ov\cM_{g,2l+k}^{\,\si_g}$;
see \cite[Section~6.2]{Sep98}.\\

\noindent
Let $p$ be as in~\eref{RUCdfn_e} with $l$ replaced by $2l\!+\!k$.
Define
\begin{alignat}{1}
\label{RRUCdfn_e}
p_{\R}\!: \R\wt\cM_{g,l;k}\!\equiv\!
\R\ov\cM_{g,l;k}\!\times_{\ov\cM_{g,2l+k}}\!\wt\cM_{g,2l+k}
\lra\R\ov\cM_{g,l;k},\\
\label{RwtcU_e} 
\pi_{\R}\!:\R\wt\cU_{g,l;k}\!\equiv\!
\R\wt\cM_{g,l;k}\!\times_{\wt\cM_{g,2l+k}}\!\wt\cU_{g,2l+k}
\lra\R\wt\cM_{g,l;k}  
\end{alignat}
be the orbifold fiber products of the morphisms~\eref{RtoCmap_e} and~$p$
and of the projection to the second component in~\eref{RRUCdfn_e} and~\eref{CwtcU_e}, 
respectively.
Suppose in addition that \eref{RUCdfn_e}
satisfies the second condition in Definition~\ref{RUC_dfn2}.
Since the image of~\eref{RtoCmap_e} is  contained in $\ov\cM_{g,2l+k}^{\,\si_g}$,
an involution~$\wt\si_g$ on $\wt\cU_{g,2l+k}$ provided by Lemma~\ref{RUC_lmm}
then lifts to an involution 
\BE{wtsiRdfn_e}\wt\si_{\R}\!: \R\wt\cU_{g,l;k}\lra \R\wt\cU_{g,l;k}\EE
which preserves the fibers of~$\pi_{\R}$.

\section{Real Ruan-Tian pseudocycles}
\label{RTpseudo_sec}

\noindent
Building on the approach in \cite[Section~2.1]{Ge2} from the $g\!=\!0$ real setting case,
we introduce a real analogue of the geometric perturbations of~\cite{RT2} 
in Section~\ref{realRTdfn_subs}.
Theorem~\ref{RTreal_thm} provides an interpretation of the  arbitrary-genus real GW-invariants of
\cite[Theorem~1.4]{RealGWsI} for semi-positive targets in the style of~\cite{RT2}.
A similar interpretation of the  genus~1 real GW-invariants of \cite[Theorem~1.5]{RealGWsI}
is obtained by combining its proof with the portions of the proof of Theorem~\ref{RTreal_thm}
not specific to the $k\!=\!0$ case; see Remark~\ref{RealRT3_rmk}.\\

\noindent
The covers~\eref{RUCdfn_e} of the Deligne-Mumford moduli spaces of curves
provided by~\cite{BP} are branched over the boundaries of the moduli spaces.
The total spaces of the universal curves~\eref{CwtcU_e} over these covers
thus have singularities around the nodal points of the fibers of the~from
$$\big\{(t,x,y)\!\in\C^3\!:xy\!=\!t^m\big\}\lra \C, \qquad (t,x,y)\lra t;$$
see the proof of \cite[Proposition~1.4]{BP}.
The approach of \cite[Section~2]{RT2} to deal with these singularities is to
embed the universal curve~\eref{CwtcU_e} into some~$\P^N$.
Standard, though delicate, algebro-geometric arguments provide
an embedding of the real universal curve~\eref{RwtcU_e} into~$\P^N$
suitable for carrying out the approach of~\cite{RT2} in the relevant real settings.
Following~\cite{LT}, we bypass such an embedding by using perturbations supported
away from the nodes.\\

\noindent
For a symplectic manifold $(X,\om)$, we denote~by $\cJ_{\om}$
the space of $\om$-compatible almost complex structures on~$X$.
If $(X,\om,\phi)$ is a real symplectic manifold, let 
\BE{cJomdfn_e}\cJ_{\om}^{\phi}=\big\{J\!\in\!\cJ_{\om}\!:\,\phi^*J\!=\!-J\big\}.\EE
For an almost complex structure~$J$ on a smooth manifold~$X$,
a complex structure~$\fJ$ on a nodal surface~$\Si$,
and a smooth map \hbox{$u\!:\Si\!\lra\!X$},  let 
$$\dbar_{J,\fJ}u=\frac{1}{2}\big(\tnd u+J\circ\tnd u\!\circ\!\fJ\big)\!:
(T\Si,-\fJ)\lra u^*(TX,J)\,.$$
Such a map is called \sf{$J$-holomorphic} if $\dbar_{J,\fJ}u\!=\!0$.
If $\Si$ is a smooth connected orientable surface, a \hbox{$C^1$-map}
\hbox{$u\!:\Si\!\lra\!X$}~is 
\begin{enumerate}[label=$\bu$,leftmargin=*]

\item \sf{somewhere injective} if there exists $z\!\in\!\Si$
such that $u^{-1}(u(z))\!=\!\{z\}$ and $\tnd_uz\!\neq\!0$,

\item \sf{multiply covered} if $u\!=\!u'\!\circ\!h$ 
for some smooth connected orientable surface~$\Si'$,
branched cover \hbox{$h\!:\!\Si\!\lra\!\Si'$} of degree different from $\pm1$,
and a smooth map $u'\!:\Si\!\lra\!\Si'$,

\item \sf{simple} if it is not multiply covered. 

\end{enumerate}
By \cite[Proposition~2.5.1]{MS}, a somewhere injective $J$-holomorphic map is simple. 
For an involution~$\phi$ on~$X$ and an involution~$\si$ on~$\Si$,
a \sf{real~map} 
$$u\!:(\Si,\si)\lra(X,\phi)$$ 
is a map $u\!:\Si\!\lra\!X$ such that $u\!\circ\!\si=\phi\!\circ\!u$.\\

\noindent
The pseudocycle constructions in \cite{MS,RT2} are based on showing that 
\begin{enumerate}[label=(\arabic*),leftmargin=*]

\item\label{regul_it} the open subspace $\fM_{\ga}^*(J)$ of each stratum $\fM_{\ga}(J)$
of the moduli space of $J$-holomorphic maps to~$X$ consisting of \sf{simple}
maps in the sense of Definition~\ref{Jnumap_dfn2} is smooth for 
a generic choice of \hbox{$J\!\in\!\cJ_{\om}$}
(of a compatible pair $(J,\nu)$ in~\cite{RT2}),

\item\label{semipos_it} 
the image of $\fM_{\ga}(J)\!-\!\fM_{\ga}^*(J)$ under the natural evaluation map~$\ev$
(and the stabilization map~$\st$ in~\cite{RT2}) is covered by smooth maps~$\ev$
(and $\st_{\ga',\vp}$)  
from some other smooth spaces  $\cZ_{\ga',\vp}^*(J)$ of 
dimension at least~2 less than the dimension of the top stratum of the moduli space.

\end{enumerate}
Our proof of Theorem~\ref{RTreal_thm} provides a systematic perspective 
on the reasoning in \cite{MS,RT2}.
We specify all spaces and maps relevant to~\ref{semipos_it} above.
The regularity of these spaces, which include the spaces in~\ref{regul_it}
as special cases, is the subject of Propositions~\ref{RTreg_prp} and~\ref{RTreg2_prp};
they are proved in Section~\ref{trans_sec}.

\subsection{Main statement}
\label{realRTdfn_subs}

\noindent
Let $g,l,k\!\in\!\Z^{\ge0}$ with $2(g\!+\!l)\!+\!k\!\ge\!3$ and \eref{RUCdfn_e} 
be a regular cover.
This implies that the fibers of~\eref{RwtcU_e} are stable genus~$g$ 
real curves with $l$~conjugate pairs of  marked points and
$k$~real marked points.
We denote~by
$$\R\wt\cU_{g,l;k}^{\,*}\subset\R\wt\cU_{g,l;k}$$
the complement of the nodes of the fibers of~$\pi_{\R}$ and~by
$$\cT_{g,l;k}=\ker\tnd\big(\pi_{\R}|_{\R\wt\cU_{g,l;k}^{\,*}}\big)\lra \R\wt\cU_{g,l;k}^{\,*}$$
the vertical tangent bundle.
The latter is a complex line bundle; let $\fJ_{\cU}$ denote its complex structure.
The action of the differential of~\eref{wtsiRdfn_e} reverses~$\fJ_{\cU}$.\\

\noindent
Let $(X,J,\phi)$ be an almost complex manifold with an involution~$\phi$
reversing the almost complex structure~$J$.
Denote~by
$$\pi_1,\pi_2\!: \R\wt\cU_{g,l;k}^{\,*}\!\times\!X\lra \R\wt\cU_{g,l;k}^{\,*},X$$
the projection maps.
Let
$$\Ga^{0,1}_{g,l;k}(X;J)=
\big\{\nu\!\in\!\Ga\big(\R\wt\cU_{g,l;k}^{\,*}\!\times\!X;
\pi_1^*(\cT_{g,l;k},-\fJ_{\cU})^*\!\otimes_{\C}\!\pi_2^*(TX,J)\big)\!\!:
\supp(\nu)\subset\R\wt\cU_{g,l;k}^*\!\times\!X\big\},$$
where $\supp(\nu)$ is the closure of the set
$$\big\{(z,x)\!\in\!\R\wt\cU_{g,l;k}^{\,*}\!\times\!X\!:\nu(z,x)\!\neq\!0\big\}$$
in~$\R\wt\cU_{g,l;k}\!\times\!X$.
Define
\begin{alignat}{1}
\label{Ga01dfn_e}
\Ga^{0,1}_{g,l;k}(X;J)^{\phi}&=
\big\{\nu\!\in\!\Ga^{0,1}_{g,l;k}(X;J)\!:
\tnd\phi\!\circ\!\nu\!=\!\nu\!\circ\!\tnd\wt\si_{\R}\big\},\\
\label{cHdfn_e}
\cH_{g,l;k}^{\om,\phi}(X)&=\big\{(J,\nu)\!:J\!\in\cJ_{\om}^{\phi},\,
\nu\!\in\!\Ga^{0,1}_{g,l;k}(X;J)^{\phi}\big\}.
\end{alignat}

\begin{dfn}\label{Jnumap_dfn}
Suppose $g,l,k\!\in\!\Z^{\ge0}$ with $2(g\!+\!l)\!+\!k\!\ge\!3$,
 $(X,J,\phi)$ is an almost complex manifold with an involution~$\phi$
reversing~$J$, 
$\nu\!\in\!\Ga^{0,1}_{g,l;k}(X;J)^{\phi}$, and $B\!\in\!H_2(X;\Z)$.
A \sf{degree~$B$ genus~$g$ real $(l,k)$-marked $(J,\nu)$-map} is a tuple
\BE{Jnumap_e} \u\equiv \big(u_{\cM}\!:\Si\!\lra\!\R\wt\cU_{g,l;k},
u\!:\Si\!\lra\!X,(z_i^+,z_i^-)_{i\in[l]},(x_i)_{i\in[k]},\si,\fJ\big),\EE
where $\big(\Si,(z_i^+,z_i^-)_{i\in[l]},(x_i)_{i\in[k]},\si,\fJ\big)$
is a nodal real genus~$g$ curve with $l$~conjugate pairs of  points and 
$k$~real points,
$u_{\cM}$ is a $(\wt\si_{\R},\si)$-real $(\fJ_{\cU},\fJ)$-holomorphic map 
onto a fiber of~$\pi_{\R}$ preserving the marked points,
and $u$ is a $(\phi,\si)$-real map such~that  
$$\dbar_{J,\fJ}u\big|_z=\nu\big(u_{\cM}(z),u(z)\big)\!\circ\!\tnd_zu_{\cM}\in 
\big(T_z\Si,-\fJ\big)^*\!\otimes_{\C}\!(T_{u(z)}X,J) ~~\forall\,z\!\in\!\Si,
\quad u_*[\Si]=B\in H_2(X;\Z).$$
\end{dfn}

\begin{dfn}\label{Jnumap_dfn2}
Suppose $g,l,k$, $(X,J,\phi)$, $\nu$, and $B$ are as in 
Definition~\ref{Jnumap_dfn}.
A \hbox{$(J,\nu)$-map~$\u$} as in~\eref{Jnumap_e} is \sf{simple} if 
the restriction of~$u$ to each irreducible component~$\Si'$ of~$\Si$ contracted by~$u_{\cM}$ 
is simple whenever $u|_{\Si'}$ is not constant and
the images of any two such components~$\Si'$ under~$u$ are~distinct.
\end{dfn}

\noindent
Following the standard terminology, we call  
the irreducible components contracted by~$u_{\cM}$ 
the \sf{contracted components of}~\eref{Jnumap_e}. 
Every such component~$\Si'$  is smooth and of genus~0;
the restriction of~$u$ to~$\Si'$ is $J$-holomorphic.
The last condition in Definition~\ref{Jnumap_dfn2} implies that the image curve 
$u(\Si')\!\subset\!X$ is not real if $\u$ is a simple map, 
$u|_{\Si'}$ is not constant, and $\Si'\!\subset\!\Si$ is not a real component.
In particular, the maps represented by the second diagram in Figure~\ref{c1B_fig}
are not simple.\\

\noindent
A $(J,\nu)$-map $\u$ as in~\eref{Jnumap_e} is \sf{equivalent} to another $(J,\nu)$-map
$$ \u\equiv \big(u_{\cM}'\!:\Si'\!\lra\!\R\wt\cU_{g,l;k},
u'\!:\Si'\!\lra\!X,(z_i'^+,z_i'^-)_{i\in[l]},(x_i')_{i\in[k]},\si',\fJ'\big)$$
if there exists a biholomorphic map $h\!:(\Si,\fJ)\!\lra\!(\Si',\fJ')$ such that 
$$h\!\circ\!\si=\si'\!\circ\!h, \quad h(z_i^+)=z_i'^+~~\forall i\!\in\![l], \quad
h(x_i)=x_i'~~\forall i\!\in\![k], \quad
(u_{\cM},u)=\big(u_{\cM}'\!\circ\!h,u'\!\circ\!h\big).$$
A $(J,\nu)$-map $\u$ is \sf{stable} if its group of automorphisms
is finite.
This is the case if and only if the degree of the restriction of~$u$ to 
every contracted component~$\Si'$ of~$\Si$ containing only one or two 
special (nodal or marked) points is not zero.\\ 

\noindent
For $(J,\nu)\!\in\!\cH_{g,l;k}^{\om,\phi}(X)$,
we denote by $\ov\fM_{g,l;k}(X,B;J,\nu)^{\phi}$ the moduli space of equivalences classes 
of stable degree~$B$ genus~$g$ real $(l,k)$-marked $(J,\nu)$-maps.
It is topologized as in \cite[Section~3]{LT} using maps from 
families of real curves described in \cite[Section~3]{Melissa}.
The~map
\BE{evst_e}\begin{split}
\st\!\times\!\ev\!: \ov\fM_{g,l;k}(X,B;J,\nu)^{\phi}&\lra 
\R\ov\cM_{g,l;k}\times \big(X^l\!\times\!(X^{\phi})^k\big),\\
\big[u_{\cM},u,(z_i^+,z_i^-)_{i\in[l]},(x_i)_{i\in[k]},\si,\fJ\big]&\lra
\big(p_{\R}(\pi_{\R}(\Im\,u_{\cM})),(u(z_i^+))_{i\in[l]},(u(x_i))_{i\in[k]}\big),
\end{split}\EE
is continuous with respect to this topology.
Let 
\BE{fMstar_e}\fM_{g,l;k}^{\star}(X,B;J,\nu)^{\phi}\subset\ov\fM_{g,l;k}(X,B;J,\nu)^{\phi}\EE
be the subspace of simple maps from domains with at most one node.

\begin{thm}\label{RTreal_thm}
Suppose $n\!\not\in\!2\Z$,
$(X,\om,\phi)$ is a compact semi-positive real symplectic \hbox{$2n$-manifold}
endowed with a real orientation, 
$g,l\!\in\!\Z^{\ge0}$ with $g\!+\!l\!\ge\!2$, and $B\!\in\!H_2(X;\Z)$.
\begin{enumerate}[label=(\arabic*),leftmargin=*]

\item\label{RTexist_it} There exists a Baire subset 
$\wh\cH_{g,l}^{\om,\phi}(X)\!\subset\!\cH_{g,l;0}^{\om,\phi}(X)$
of second category such~that  the restriction
\BE{RTreal_e} \st\!\times\!\ev\!:\fM_{g,l;0}^{\star}(X,B;J,\nu)^{\phi} \lra 
\R\ov\cM_{g,l}\!\times\!X^l\EE
is a pseudocycle of dimension 
\BE{RTreal_e2}\dim_{\R}\fM_{g,l;0}^{\star}(X,B;J,\nu)^{\phi}=
\blr{c_1(TX),B}\!+\!(n\!-\!3)(1\!-\!g)\!+\!2l\EE
for every $(J,\nu)\!\in\!\wh\cH_{g,l}^{\om,\phi}(X)$.

\item\label{RTindep_it} The homology class on $\R\ov\cM_{g,l}\!\times\!X^l$
determined by this pseudocycle is independent of the choice of 
$(J,\nu)$ in~$\wh\cH_{g,l}^{\om,\phi}(X)$.
The~class
\BE{RTclass_e}
\frac{1}{\tn{deg}\,p}\Big[\st\!\times\!\ev\!:\fM_{g,l;0}^{\star}(X,B;J,\nu)^{\phi} \lra 
\R\ov\cM_{g,l}\!\times\!X^l\Big]\in 
H_*\Big(\R\ov\cM_{g,l}\!\times\!X^l;\Q\Big)\EE
is also independent of the choice of a regular cover~\eref{RUCdfn_e}.
\end{enumerate}
\end{thm}


\noindent
The same conclusions apply with the $\R\ov\cM_{g,l}$ factor dropped everywhere.
Identical notions of pseudocycle with target in a manifold~$M$  
appear in \cite[Section~1.1]{pseudo} and in \cite[Definition~6.5.1]{MS}.
They readily extend to orbifold targets~$M$ and include 
the more elaborate and less convenient notion of pseudo-manifold
of \cite[Definition~4.1]{RT2}.
By \cite[Theorem~1.1]{pseudo}, the group of pseudocycles into a manifold~$M$ 
modulo equivalence is naturally isomorphic to~$H_*(M;\Z)$.
The same reasoning applies to orbifold targets.\\

\noindent
The proof of Theorem~\ref{RTreal_thm} in the rest of the paper follows the same general principles 
as the proofs of \cite[Theorem~6.6.1]{MS} and \cite[Propositions~2.3,2.5]{RT2}.
However, their implementation in the real case requires more care.
For example, the proof of the crucial transversality statements
Propositions~\ref{RTreg_prp} and~\ref{RTreg2_prp}
requires classifying the irreducible components 
of the domains of the elements of each stratum of the moduli space into five types, 
instead of~one in~\cite{MS} and~two in~\cite{RT2}, 
based on whether they are contracted or not and whether they are real or~not.

\begin{rmk}\label{RealRT3_rmk}
The only steps in the proof of Theorem~\ref{RTreal_thm} dependent on the $k\!=\!0$ assumption
are Corollaries~\ref{RTorient_crl} and~\ref{RTorient2_crl}.
A geometric interpretation of the genus~1 real GW-invariants of 
\cite[Theorem~1.5]{RealGWsI} is obtained by combining its proof with the remaining steps
in the proof of Theorem~\ref{RTreal_thm}.
\end{rmk}

\subsection{Strata of stable real maps}
\label{RTstr_subs}

\noindent
The moduli spaces of real curves and maps are stratified based on
the topological type of the domain and the distribution of
the map degree between its irreducible components in the latter case.
These data correspond to certain decorated graphs.
Because of the contraction operations on these graphs
that are central to our proof of Theorem~\ref{RTreal_thm},
we define such graphs based on the perspective in \cite[Section~2.1.1]{Ceyhan}.\\

\noindent
For $l,k\!\in\!\Z^{\ge0}$, define
$$\si_{l;k}\!:S_{l;k}\!\equiv\!\{i^+\!:i\!\in\![l]\big\}\!\sqcup\!
\{i^-\!:i\!\in\![l]\big\}\!\sqcup\!\{i\!:i\!\in\![k]\big\}\lra S_{l;k},
\quad \si_{l;k}(f)=\begin{cases}i^{\mp},&\hbox{if}~f\!=\!i^{\pm},i\!\in\![l];\\
f,&\hbox{if}~f\!\in\![k].
\end{cases}$$
An \sf{$(l,k)$-marked graph} is a tuple
\BE{stgadfn_e} \ov\ga\equiv 
\big(\ov\fg\!:\!\ov\Ver\!\lra\!\Z^{\ge0},
\ov\ve\!:S_{l;k}\!\sqcup\!\ov\Fl\!\lra\!\ov\Ver,\ov\vt\!:\ov\Fl\!\lra\!\ov\Fl,
\ov\si\!:\ov\Ver\!\sqcup\!\ov\Fl\!\lra\!\ov\Ver\!\sqcup\!\ov\Fl\big),\EE
where $\ov\Ver$ and $\ov\Fl$ are finite sets (of \sf{vertices} and \sf{flags}, respectively)
and $\ov\vt$ and $\ov\si$ are involutions such~that 
$$\ov\vt(v)\neq v~~\forall\,v\!\in\!\ov\Fl, \quad
\ov\fg\!\circ\!\ov\si|_{\ov\Ver}=\ov\fg, \quad
\ov\ve\!\circ\!\big\{\si_{l;k}\!\sqcup\!\ov\si|_{\ov\Fl}\big\}
=\ov\si|_{\ov\Ver}\!\circ\!\ov\ve, \quad 
\ov\vt\!\circ\!\ov\si|_{\ov\Fl}=\ov\si|_{\ov\Fl}\!\circ\!\ov\vt\,.$$
For $f\!\in\!S_{l;k}$, let $\ov\si(f)\!=\!\si_{l;k}(f)$.
We denote by $\Aut(\ov\ga)$ the group of automorphisms of~$\ov\ga$,
i.e.~pairs of automorphisms of the sets~$\ov\Ver$ and~$\ov\Fl$ commuting
with the maps $\ov\fg$, $\ov\ve$, and~$\ov\si$.
Define
\BE{nVRCdfn_e}\nV_{\R}(\ov\ga)=\big\{v\!\in\!\ov\Ver\!:\ov\si(v)\!=\!v\big\}, \qquad
\nV_{\C}(\ov\ga)=\big\{v\!\in\!\Ver\!:\ov\si(v)\!\neq\!v\big\}\,.\EE

\vspace{.2in}

\noindent
The set of \sf{edges} of $\ov\ga$ as in~\eref{stgadfn_e} is
$$\E(\ov\ga)\equiv\big\{e\!=\!\{f,\ov\vt(f)\}\!:f\!\in\!\ov\Fl\big\}.$$
The involution~$\ov\si|_{\ov\Fl}$ induces an involution on $\E(\ov\ga)$,
which we also denote by~$\ov\si$.
The graph~$\ov\ga$ is \sf{connected} if for all $v,v'\!\in\!\ov\Ver$ distinct 
there exist
\begin{gather*}
m\!\in\!\Z^+,~f_1^-,f_1^+,\ldots,f_m^-,f_m^+\!\in\!\ov\Fl \qquad\hbox{s.t.}\\
\ov\ve\big(f_1^-\big)\!=\!v,~\ov\ve\big(f_m^+\big)\!=\!v',~
\ov\ve\big(f_i^+\big)\!=\!\ov\ve\big(f_{i+1}^-\big)~\forall\,i\!\in\![m\!-\!1],~
\big\{f_i^-,f_i^+\big\}\!\in\!\E(\ov\ga)~\forall\,i\!\in\![m].
\end{gather*}
Define
\BE{Etypdfn_e}\begin{split}
\E_{\C}(\ov\ga)&\equiv\big\{e\!\in\!\E(\ov\ga)\!:\ov\si(e)\!\neq\!e\big\},
\quad
\E_{\R\C}(\ov\ga)\equiv\big\{
e\!=\!\{f,\ov\vt(f)\}\!:f\!\in\!\ov\Fl,\,\ov\si(f)\!=\!\ov\vt(f)\big\},\\
&\hbox{and}\quad \E_{\R\R}(\ov\ga)\equiv\big\{
e\!=\!\{f,\ov\vt(f)\}\!:f\!\in\!\ov\Fl,\,\ov\si(f)\!=\!f,
\,\ov\si(\ov\vt(f))\!=\!\ov\vt(f)\big\}.
\end{split}\EE

\vspace{.2in}

\noindent
For each $v\!\in\!\ov\Ver$, let 
\BE{Svdfn_e}
S_{v;\R}(\ov\ga)=\big\{f\!\in\!\ov\ve^{-1}(v)\!:\ov\si(f)\!=\!f\big\},\qquad
S_{v;\C}(\ov\ga)=\big\{f\!\in\!\ov\ve^{-1}(v)\!:\ov\si(f)\!\neq\!f\big\}.\EE
If $v\!\in\!\nV_{\R}(\ov\ga)$, the involution~$\ov\si$ restricts to 
an involution~$\ov\si_v$ on $S_{v;\R}(\ov\ga)$ and $S_{v;\C}(\ov\ga)$.
If $v\!\in\!\nV_{\C}(\ov\ga)$, $S_{v;\R}(\ov\ga)\!=\!\eset$ and
the involution~$\ov\si$ restricts to 
an involution $\ov\si_v\!=\!\ov\si_{\ov\si(v)}$ on 
$S_{v;\C}(\ov\ga)\!\cup\!S_{\ov\si(v);\C}(\ov\ga)$.\\

\noindent
Let $\ov\ga$ be as in~\eref{stgadfn_e}.
A vertex $v\!\in\!\ov\Ver$ of~$\ov\ga$ is \sf{trivalent}~if 
\BE{trivver_e}2\ov\fg(v)\!+\!\big|\ov\ve^{-1}(v)\big|\ge 3.\EE
The graph~$\ov\ga$ is \sf{trivalent} if all its vertices are trivalent.
For $g\!\in\!\Z^{\ge0}$,
we denote the (finite) set of (equivalence classes of) connected trivalent graphs~$\ov\ga$ 
as in~\eref{stgadfn_e} such~that 
\BE{stgacnd_e}
1\!+\!|\E(\ov\ga)|\!-\!|\ov\Ver|\!+\!\sum_{v\in\ov\Ver}\!\!\!\ov\fg(v)=g\EE
by~$\cA_{g,l;k}$.\\

\noindent
Suppose \hbox{$g,l,k\!\in\!\Z^{\ge0}$} with \hbox{$2(g\!+\!l)\!+\!k\!\ge\!3$},
$\ov\ga$ is a connected graph as in~\eref{stgadfn_e} satisfying~\eref{stgacnd_e},
and $v\!\in\!\ov\Ver$ is a vertex not satisfying~\eref{trivver_e}.
The vertices~$v$ and $\ov\si(v)$ can then be \sf{contracted} to obtain another
$(l,k)$-marked graph
$$ \ov\ga'\equiv 
\big(\ov\fg'\!:\!\ov\Ver'\!\lra\!\Z^{\ge0},
\ov\ve'\!:S_{l;k}\!\sqcup\!\ov\Fl'\!\lra\!\ov\Ver',\ov\vt'\!:\ov\Fl'\!\lra\!\ov\Fl',
\ov\si'\!:\ov\Ver'\!\sqcup\!\ov\Fl'\!\lra\!\ov\Ver'\!\sqcup\!\ov\Fl'\big)$$
satisfying~\eref{stgacnd_e} and
\begin{gather*}
\ov\Ver'=\ov\Ver\!-\!\big\{v,\ov\si(v)\big\}, \quad
\ov\fg'=\ov\fg|_{\ov\Ver'}, \quad 
\ov\Fl'\subset\ov\Fl\!\cap\!\ve^{-1}\big(\ov\Ver'\big),
~~\ov\si\big(\ov\Fl'\big)\!=\!\ov\Fl', \quad
\ov\si'=\ov\si|_{\ov\Ver'\sqcup \ov\Fl'},\\
\ov\ve'=\ov\ve~~\hbox{on}~\big(S_{l;k}\!\cap\!\ov\ve^{-1}(\ov\Ver')\!\big)\!\sqcup\!\ov\Fl',
\quad
\ov\vt'=\ov\vt~~\hbox{on}~\big\{f\in\ov\Fl'\!:\ov\ve(\ov\vt(f))\!\neq\!v,\ov\si(v)\big\}
\end{gather*}
as follows. 
We take 
$$\ov\Fl'=\begin{cases}
\ov\Fl\!\cap\!\ov\ve^{-1}\big(\ov\Ver'\big),
&\hbox{if}~|\ov\Fl\!\cap\!\ov\ve^{-1}(v)|\!=\!2;\\
\big\{f\!\in\!\ov\Fl\!\cap\!\ov\ve^{-1}\big(\ov\Ver'\big)\!:
\ov\ve\big(\ov\vt(f)\big)\!\neq\!v,\ov\si(v)\big\},
&\hbox{if}~|\ov\Fl\!\cap\!\ov\ve^{-1}(v)|\!=\!1.
\end{cases}$$
In the case $|\ov\Fl\!\cap\!\ov\ve^{-1}(v)|\!=\!2$, we extend $\ov\vt'$ 
from its specification above~by
$$\ov\vt'(f_1)=f_2\quad\hbox{if}~f_1,f_2\!\in\!\ov\Fl',~
f_1\!\neq\!f_2,\,
\ov\ve\big(\ov\vt(f_1)\big)\!=\!\ov\ve\big(\ov\vt(f_2)\big)\in\big\{v,\ov\si(v)\big\}.$$
In the case $|S_{l;k}\!\cap\!\ov\ve^{-1}(v)|\!=\!1$, we extend $\ov\ve'$ 
from its specification above~by
$$\ov\ve'(f_1)=\ov\ve\big(\ov\vt(f_2)\big) \quad\hbox{if}~f_1\!\in\!S_{l;k},\,f_2\!\in\!\ov\Fl',~
\ov\ve(f_1)\!=\!\ov\ve(f_2)\in\big\{v,\ov\si(v)\big\}.$$
By the assumption that $v$ does not satisfy~\eref{trivver_e}, 
these extensions are well-defined.\\

\noindent
Let \hbox{$g,l,k\!\in\!\Z^{\ge0}$} with \hbox{$2(g\!+\!l)\!+\!k\!\ge\!3$}.
For each $\ov\ga\!\in\!\cA_{g,l;k}$, denote by
$$\R\cM_{\ov\ga}\subset\R\ov\cM_{g,l;k}$$
the subspace parametrizing marked real curves 
$$\cC\equiv\big(\Si,(z_i^+,z_i^-)_{i\in[l]},(z_i)_{i\in[k]},\si,\fJ\big)$$
with dual graph~$\ov\ga$.
Thus, the irreducible components~$\Si_v$ and the nodes~$z_e$
of~$\Si$ are indexed by the elements 
of~$\ov\Ver$ and $\E(\ov\ga)$, respectively. 
The node~$z_e$ corresponding to \hbox{$e\!=\!\{f,\ov\vt(f)\}$} is obtained 
by identifying a point $z_f\!\in\!\Si_{\ov\ve(f)}$ with a point   
\hbox{$z_{\ov\vt(f)}\!\in\!\Si_{\ov\ve(\ov\vt(f))}$}.
The marked point~$z_f$ corresponding to $f\!\in\!S_{l;k}$ is carried
by the irreducible component~$\Si_{\ov\ve(f)}$.
The involution~$\si$ sends~$\Si_v$ to~$\Si_{\ov\si(v)}$ and~$z_e$ 
to~$z_{\ov\si(e)}$.
The two sets in~\eref{nVRCdfn_e} correspond to the \sf{real components} of~$\Si$,
i.e.~those fixed by the involution, 
and the \sf{conjugate components} of~$\Si$, i.e.~those interchanged by the involution.
The three sets in~\eref{Etypdfn_e}
correspond to the C, E, and H-nodes of~$\Si$; 
see \cite[Section~3]{RBP} for the terminology.
Let
\BE{cMga_e1}\R\wt\cM_{\ov\ga}\!=\!p_{\R}^{-1}\big(\R\cM_{\ov\ga}\big)
\subset\R\wt\cM_{g,l;k}.\EE
The stratification of~$\R\wt\cM_{g,l;k}$ by the subspaces in~\eref{cMga_e1} 
is analogous to the stratification \cite[(3.2)]{RT2} in the complex case.\\

\noindent
For $v\!\in\!\nV_{\R}(\ov\ga)$, let $\R\cM_{\ov\ga;v}$ 
denote the moduli space of smooth real genus~$\ov\fg(v)$ connected curves with 
the real and conjugate marked points indexed by $S_{v;\R}(\ov\ga)$
and $S_{v;\C}(\ov\ga)$, respectively.
For $v\!\in\!\nV_{\C}(\ov\ga)$, let $\cM_{\ov\ga;v}^{\bu}$ 
denote the moduli space of smooth real curves 
with two genus~$\ov\fg(v)$ topological components, $\Si_v$ and~$\Si_{\ov\si(v))}$,
interchanged by the involution and carrying the marked points indexed by~$S_{v;\C}(\ov\ga)$,
and~$S_{\ov\si(v);\C}(\ov\ga)$, respectively;
we call such curves \sf{real doublets}.
The image of the immersion
\BE{iogaimm_e}\prod_{v\in\nV_{\R}(\ov\ga)}\!\!\!\!\!\R\cM_{\ov\ga;v}\times
\prod_{\{v,\ov\si(v)\}\subset\nV_{\C}(\ov\ga)}\hspace{-.33in}\cM_{\ov\ga;v}^{\bu}
\lra \R\ov\cM_{g,l;k}\EE
identifying the marked point~$z_f$ with $z_{\ov\vt(f)}$ for each $f\!\in\!\ov\Fl$
is~$\R\cM_{\ov\ga}$. 
This immersion descends to an isomorphism from the quotient of its domain
by the natural $\Aut(\ov\ga)$ action to~$\R\cM_{\ov\ga}$.\\

\noindent
By the last requirement in Definition~\ref{RUC_dfn2}, 
there exist covers
$$\R\wt\cM_{\ov\ga;v}\lra\R\cM_{\ov\ga;v},~~v\!\in\!\nV_{\R}(\ov\ga),
\quad\hbox{and}\quad
\wt\cM_{\ov\ga;v}^{\bu}\lra\cM_{\ov\ga;v}^{\bu},
~~\big\{v,\ov\si(v)\big\}\subset\nV_{\C}(\ov\ga),$$
with universal curves 
$$\R\wt\cU_{\ov\ga;v}\lra \R\wt\cM_{\ov\ga;v},~~v\!\in\!\nV_{\R}(\ov\ga),
\quad\hbox{and}\quad
\wt\cU_{\ov\ga;v}^{\bu}\lra \wt\cM_{\ov\ga;v}^{\bu},~~
\big\{v,\ov\si(v)\big\}\subset\nV_{\C}(\ov\ga),$$
and an immersion
\BE{wtiocRM_e}\wt\io_{\ov\ga}\!:\R\wt\cM_{\ov\ga}\!\equiv\!
\prod_{v\in\nV_{\R}(\ov\ga)}\!\!\!\!\!\R\wt\cM_{\ov\ga;v}\times
\prod_{\{v,\ov\si(v)\}\subset\nV_{\C}(\ov\ga)}\hspace{-.33in}\wt\cM_{\ov\ga;v}^{\bu}
\lra \R\wt\cM_{g,l;k}\EE
lifting~\eref{iogaimm_e}.\\

\noindent
Let $B\!\in\!H_2(X;\Z)$.
An \sf{$(l,k)$-marked degree~$B$ graph} is a tuple
\BE{gadfn_e}\begin{split} 
\ga\equiv 
\Big((\fg,\fd)\!:\!\Ver\!\lra\!\Z^{\ge0}\!\oplus\!H_2(X;\Z),
\ve\!:S_{l;k}\!\sqcup\!\Fl\!\lra\!\Ver,\vt\!:\Fl\!\lra\!\Fl,&\\
\si\!:\Ver\!\sqcup\!\Fl\!\lra\!\Ver\!\sqcup\!\Fl&\Big)
\end{split}\EE
such that the tuple
$$\ga_{\cM}\equiv \big(\fg\!:\!\Ver\!\lra\!\Z^{\ge0},
\ve\!:S_{l;k}\!\sqcup\!\Fl\!\lra\!\Ver,\vt\!:\Fl\!\lra\!\Fl,
\si\!:\Ver\!\sqcup\!\Fl\!\lra\!\Ver\!\sqcup\!\Fl\big)$$
is an $(l,k)$-marked graph and 
\BE{gafdcond_e}\fd\!\circ\!\si=-\phi_*\!\circ\!\fd, \qquad
\sum_{v\in\Ver}\!\!\!\fd(v)=B,\qquad
\blr{\om,\fd(v)}\ge0~~\forall~v\!\in\!\Ver\,.\EE
Denote by $\Aut(\ga)$ the group of automorphisms of~$\ga$,
i.e.~the subgroup of automorphisms of~$\ga_{\cM}$ preserving~$\fd$.
Let
\begin{gather*}
\nV_{\R}(\ga)=\nV_{\R}(\ga_{\cM}), \quad 
\nV_{\C}(\ga)=\nV_{\C}(\ga_{\cM}), \quad 
S_{v;\R}(\ga)=S_{v;\R}(\ga_{\cM}),~
S_{v;\C}(\ga)=S_{v;\C}(\ga_{\cM})~~\forall\,v\!\in\!\Ver,\\
\E(\ga)=\E\big(\ga_{\cM}\big), ~~
\E_{\C}(\ga)=\E_{\C}(\ga_{\cM}), ~~
\E_{\R\C}(\ga)=\E_{\R\C}(\ga_{\cM}), ~~
\E_{\R\R}(\ga)=\E_{\R\R}(\ga_{\cM}),~~
|\ga|=\big|\E(\ga)\big|.
\end{gather*} 
We call $\ga$ \sf{connected} if~$\ga_{\cM}$ is connected
and a vertex $v\!\in\!\Ver$ \sf{trivalent} if 
it satisfies~\eref{trivver_e} with the overlines dropped.\\ 

\noindent
Let \hbox{$g,l,k\!\in\!\Z^{\ge0}$} with \hbox{$2(g\!+\!l)\!+\!k\!\ge\!3$}
and $B\!\in\!H_2(X;\Z)$.
We denote the set of (equivalence classes of) connected graphs~$\ga$ 
as in~\eref{gadfn_e} such~that \eref{stgacnd_e} with the overlines dropped
holds and 
$$\fd(v)\!=\!0,~2\fg(v)\!+\!\big|\ve^{-1}(v)\big|\ge3 
\qquad\forall~v\!\in\!\Ver~\hbox{s.t.}~\blr{\om,\fd(v)}\!=\!0$$
by~$\cA_{g,l;k}^{\phi}(B)$.
As described below, the moduli space on the left-hand side of~\eref{evst_e} is stratified 
by the subspaces $\fM_{\ga}(J,\nu)$
that are indexed by $\ga\!\in\!\cA_{g,l;k}^{\phi}(B)$ and consist
of maps from domains of the same topological type~$\ga_{\cM}$.\\

\noindent
Let $\ga\!\in\!\cA_{g,l;k}^{\phi}(B)$ be as in~\eref{gadfn_e}.
The \sf{stabilization} $\ov\ga\!\in\!\cA_{g,l;k}$ 
of~$\ga$ is the trivalent graph as in~\eref{stgadfn_e}
obtained by contracting the non-trivalent vertices of~$\ga$ until all vertices become trivalent.
The set~$\ov\Ver$ thus consists of trivalent vertices of~$\ga$.
It contains all vertices $v\!\in\!\Ver$ with $\fg(v)\!>\!0$, but may be missing some vertices~$v$
with $\fg(v)\!=\!0$ and \hbox{$|\ve^{-1}(v)|\!\ge\!3$}.
Let
$$\ale(\ga)\equiv\Ver\!-\!\ov\Ver, \quad
\ale_{\R}(\ga)=\big\{v\!\in\!\ale(\ga)\!:\si(v)\!=\!v\big\}, \quad\hbox{and}\quad
\ale_{\C}(\ga)=\big\{v\!\in\!\ale(\ga)\!:\si(v)\!=\!v\big\}$$
denote the set of vertices contracted by the stabilization, 
its subset fixed by the involution on the graph, and the complement of this~subset.\\

\noindent
Define
\begin{alignat*}{2}
\R\wt\cU_{\ga;v}\!\equiv\!\R\wt\cU_{\ov\ga;v}&\lra
\R\wt\cM_{\ga;v}\!\equiv\!\R\wt\cM_{\ov\ga;v}
&\quad&\hbox{if}~v\!\in\!\nV_{\R}(\ov\ga), \\
\wt\cU_{\ga;v}^{\bu}\!\equiv\!\wt\cU_{\ov\ga;v}^{\bu}
&\lra \wt\cM_{\ga;v}^{\bu}\!\equiv\wt\cM_{\ov\ga;v}^{\bu}
&\quad&\hbox{if}~v\!\in\!\nV_{\C}(\ov\ga).
\end{alignat*}
If $v\!\in\!\ale_{\R}(\ga)$ with $|\ve^{-1}(v)|\!\ge\!3$, 
we take $\R\wt\cM_{\ga;v}\!=\!\R\cM_{\ga;v}$
with $\R\cM_{\ga;v}$ defined as before and 
$$\R\wt\cU_{\ga;v}\lra\R\wt\cM_{\ga;v}$$
to be the universal curve.
If $v\!\in\!\ale_{\C}(\ga)$ with $|\ve^{-1}(v)|\!\ge\!3$, 
we take $\wt\cM_{\ga;v}^{\bu}\!=\!\cM_{\ga;v}^{\bu}$
with $\cM_{\ga;v}^{\bu}$ defined as before and 
$$\wt\cU_{\ga;v}^{\bu}\lra\wt\cM_{\ga;v}^{\bu}$$
to be the universal curve.
For $v\!\in\!\ale_{\R}(\ga)$ with $|\ve^{-1}(v)|\!\le\!2$ and 
$v\!\in\!\ale_{\C}(\ga)$ with $|\ve^{-1}(v)|\!\le\!2$,
denote by $\R\wt\cM_{\ga;v}$ and $\wt\cM_{\ga;v}^{\bu}$ the one-point spaces.
Let
$$\R\wt\cM_{\ga}= \prod_{v\in\nV_{\R}(\ga)}\!\!\!\!\!\!\R\wt\cM_{\ga;v}\times
\prod_{\{v,\si(v)\}\subset\nV_{\C}(\ga)}\hspace{-.33in}\wt\cM_{\ga;v}^{\bu}\,.$$
Denote by
$$p_{\ga;v}\!:\R\wt\cM_{\ga}\lra\R\wt\cM_{\ga;v}, ~~v\!\in\!\nV_{\R}(\ga), \qquad
p_{\ga;v}^{\bu}\!:\R\wt\cM_{\ga}\lra \wt\cM_{\ga;v}^{\bu}~~v\!\in\!\nV_{\C}(\ga),$$
the projection maps.\\

\noindent
Let  $\Aut(\P^1)$ be the group of holomorphic automorphisms of~$\P^1$.
For $v\!\in\!\ale_{\C}(\ga)$ with $|\ve^{-1}(v)|\!\le\!2$, let
$$\wt\cU_{\ga;v}^{\bu}\equiv \Si_v\!\sqcup\!\Si_{\si(v)}\equiv \P^1\!\sqcup\!\P^1$$
be the genus~0 real doublet with the involution $\tau_1\!:\Si_v\!\lra\!\Si_{\si(v)}$
and the marked points indexed by \hbox{$\ve^{-1}(v)\!\cup\!\ve^{-1}(\si(v))$} so that 
the marked points on~$\Si_v$ are given~by
\BE{nodecondC_e}\big\{z_f\!:f\!\in\!\ve^{-1}(v)\big\}=
\begin{cases}
\{\i\},&\hbox{if}~\,|\ve^{-1}(v)|\!=\!1;\\
\{\i,0\},&\hbox{if}~\,|\ve^{-1}(v)|\!=\!2.
\end{cases}\EE
Define 
$$G_v=\begin{cases}
\{h\!\in\!\Aut(\P^1)\!:h(\i)\!=\!\i\},&\hbox{if}~|\ve^{-1}(v)|\!=\!1;\\
\{h\!\in\!\Aut(\P^1)\!:h(\i)\!=\!\i,h(0)\!=\!0\},&\hbox{if}~|\ve^{-1}(v)|\!=\!2.
\end{cases}$$

\vspace{.2in}

\noindent
For $\rho\!=\!\tau,\eta$, let $\Aut_{\rho}(\P^1)$ be the subgroup of $\Aut(\P^1)$
consisting of the automorphisms that commute with~$\rho$. 
Denote by $\Inv(\ga)$ the set of~maps 
\begin{gather*}
\rho\!:\big\{v\!\in\!\ale_{\R}(\ga)\!:|\ve^{-1}(v)|\!\le\!2\big\}
\lra\big\{\tau,\eta\} \qquad\hbox{s.t.}\quad
\rho(v)=\tau~~\hbox{if}~~S_{v;\R}(\ga)\neq\eset.
\end{gather*}
For $\rho\!\in\!\Inv(\ga)$ and $v\!\in\!\ale_{\R}(\ga)$ such that $|\ve^{-1}(v)|\!\le\!2$, 
let 
$$\R\wt\cU_{\ga;\rho;v}\equiv \Si_v\equiv \P^1$$
be the genus~0 real curve with the involution~$\rho(v)$
and the marked points
\BE{nodecondR_e}\big\{z_f\!:f\!\in\!\ve^{-1}(v)\big\}=
\begin{cases}
\{1\},&\hbox{if}~|\ve^{-1}(v)|\!=\!1;\\
\{1,-1\},&\hbox{if}~|S_{v;\R}(\ga)|\!=\!2;\\
\{\i,0\},&\hbox{if}~|S_{v;\C}(\ga)|\!=\!2.
\end{cases}\EE
Define
$$G_{\rho;v}=\begin{cases}
\{h\!\in\!\Aut_{\tau}(\P^1)\!:h(1)\!=\!1\},&\hbox{if}~|\ve^{-1}(v)|\!=\!1;\\
\{h\!\in\!\Aut_{\tau}(\P^1)\!:h(1)\!=\!1,h(-1)\!=\!-1\},&\hbox{if}~|S_{v;\R}(\ga)|\!=\!2;\\
\{h\!\in\!\Aut_{\rho(v)}(\P^1)\!:h(\i)\!=\!\i,h(0)\!=\!0\},&\hbox{if}~|S_{v;\C}(\ga)|\!=\!2.
\end{cases}$$

\vspace{.2in}

\noindent
For each $v\!\in\!\nV_{\C}(\ga)$, we have thus constructed a fibration 
\BE{pigarhovC_e}
\pi_{\ga;v}^{\bu}\!\equiv\!\pi_{\ga;v}\!\sqcup\!\pi_{\ga;\si(v)}\!:
\wt\cU_{\ga;v}^{\bu}\!\equiv\!\wt\cU_{\ga;v}\!\sqcup\!\wt\cU_{\ga;\si(v)}
\lra \wt\cM_{\ga;v}^{\bu}\EE
by smooth genus~$\fg(v)$ real marked doublets.
For $\rho\!\in\!\Inv(\ga)$ and \hbox{$v\!\in\!\nV_{\R}(\ga)$}, 
we have constructed a fibration
\BE{pigarhovR_e}\pi_{\ga;\rho;v}\!:\R\wt\cU_{\ga;\rho;v}\lra \R\wt\cM_{\ga;v}\EE
by smooth genus~$\fg(v)$ real marked curves 
(described above without the $\rho$~subscript except for 
$v\!\in\!\ale_{\R}(\ga)$ with $|\ve^{-1}(v)|\!\le\!2$).
Define
\begin{gather*}
\R\wt\cU_{\ga;\rho} = \bigg(
\bigsqcup_{v\in\nV_{\R}(\ga)}\!\!\!\!\!p_{\ga;v}^{\,*}\R\wt\cU_{\ga;\rho;v}\times
\bigsqcup_{\{v,\si(v)\}\subset\nV_{\C}(\ga)}
\hspace{-.32in} p_{\ga;v}^{\bu\,*}\wt\cU_{\ga;v}^{\bu}
\bigg)\!\!\bigg/\!\!\!\!\sim \qquad\hbox{with}\\
\big((\cC_v)_{v\in\nV_{\R}(\ga)},(\cC_v^{\bu})_{\{v,\si(v)\}\subset\nV_{\C}(\ga)};
z_f\big)\sim
\big((\cC_v)_{v\in\nV_{\R}(\ga)},(\cC_v^{\bu})_{\{v,\si(v)\}\subset\nV_{\C}(\ga)};
z_{\vt(f)}\big)\quad\forall\,f\!\in\!\Fl,
\end{gather*}
i.e.~we identify marked points in a fiber of 
$$\bigsqcup_{v\in\nV_{\R}(\ga)}\!\!\!\!\!p_{\ga;v}^{\,*}\R\wt\cU_{\ga;\rho;v}\times
\bigsqcup_{\{v,\si(v)\}\subset\nV_{\C}(\ga)}
\hspace{-.3in} p_{\ga;v}^{\bu\,*}\wt\cU_{\ga;v}^{\bu}\lra\R\wt\cM_{\ga}$$
if they correspond to flags interchanged by the involution~$\vt$ on~$\Fl$.
The natural projection
\BE{pigarho_e}\pi_{\ga;\rho}\!:\R\wt\cU_{\ga;\rho}\lra\R\wt\cM_{\ga}\EE
is a fibration whose fibers are  nodal marked real curves with dual graph~$\ga_{\cM}$;
the irreducible components~$\Si_v$ of these fibers are indexed by the set~$\Ver$.
Let $\R\wt\cU_{\ga;\rho}^*\!\subset\!\R\wt\cU_{\ga;\rho}$
be the complement of the nodes of the fibers of~$\pi_{\ga;\rho}$.\\

\noindent
The involutions~$\si$ and~$\rho$ induce an involution~$\wt\si_{\ga;\rho}$ 
on~$\R\wt\cU_{\ga;\rho}$.
Let
$$\cT_{\ga;\rho}\equiv\ker\tnd\big(\pi_{\ga;\rho}|_{\R\wt\cU_{\ga;\rho}^*}\big)
\lra \R\wt\cU_{\ga;\rho}^*$$
be the vertical tangent bundle. 
Denote by~$\fJ_{\ga;\rho}$ the complex structure on this complex line bundle. 
The~group 
$$G_{\ga;\rho}^{\circ}\equiv 
\prod_{\begin{subarray}{c}v\in\ale_{\R}(\ga)\\ |\ve^{-1}(v)|\le2\end{subarray}}
\!\!\!\!\!\!\!\!G_{\rho;v}
~\times
\prod_{\begin{subarray}{c}\{v,\si(v)\}\subset\ale_{\C}(\ga)\\ |\ve^{-1}(v)|\le2\end{subarray}}
\!\!\!\!\!\!\!\!\!\!\!\!\!G_v$$
acts on $\R\wt\cU_{\ga;\rho}$ by reparametrizing the irreducible components~$\Si_v$ of the fibers 
with \hbox{$|\ve^{-1}(v)|\!\le\!2$}.
Let 
\BE{qgarhodfn_e}
q_{\ga;\rho}\!:  \R\wt\cU_{\ga;\rho}\lra\R\wt\cU_{g,l;k}\big|_{\R\wt\cM_{\ov\ga}}\EE
be the surjection covering the composition of the projection 
$\R\wt\cM_{\ga}\!\lra\!\R\wt\cM_{\ov\ga}$ with~\eref{wtiocRM_e} and contracting 
the irreducible components of the fibers of~$\pi_{\ga;\rho}$ indexed by~$\ale(\ga)$.
Denote  by $G_{\ga;\rho}$ the group of holomorphic automorphisms of~$q_{\ga;\rho}$
that commute with~$\wt\si_{\ga;\rho}$ and preserve the marked points.
The identity component of~$G_{\ga;\rho}$ is~$G_{\ga;\rho}^{\circ}$;
the group $G_{\ga;\rho}/G_{\ga;\rho}^{\circ}$ is naturally isomorphic to~$\Aut(\ga)$.\\

\noindent
For $J\!\in\!\cJ_{\om}^{\phi}$, define 
\begin{equation*}\begin{split}
\Ga^{0,1}_{\ga;\rho}(X;J)^{\phi}
=\Big\{\nu\!\in\!\Ga\big(\R\wt\cU_{\ga;\rho}^*\!\times\!X;
\pi_1^*(\cT_{\ga;\rho},-\fJ_{\ga;\rho})^*\!\otimes_{\C}\!\pi_2^*(TX,J)\big)\!:
\supp(\nu)\!\subset\!\R\wt\cU_{\ga;\rho}^*\!\times\!X,&\\
\tnd\phi\!\circ\!\nu\!=\!\nu\!\circ\!\tnd\wt\si_{\ga;\rho}&\Big\}.
\end{split}\end{equation*}
For $\nu\!\in\!\Ga^{0,1}_{\ga;\rho}(X;J)^{\phi}$, let
$\wt\fM_{\ga;\rho}(J,\nu)$ be the space of~tuples
\BE{Jnumap_e5}\u\equiv\big(u\!:\Si\!\lra\!X,(z_f)_{f\in S_{l;k}},\si,\fJ\big),\EE
where $(\Si,(z_f)_{f\in S_{l;k}},\si,\fJ)$
is a fiber of $\pi_{\ga;\rho}$ and $u$ is a $(\phi,\si)$-real map such~that  
\begin{alignat*}{2}
\dbar_{J,\fJ}u\big|_z&=\nu\big(z,u(z)\big)\in 
\big(T_z\Si,-\fJ\big)^*\!\otimes_{\C}\!(T_{u(z)}X,J) &\qquad&\forall\,z\!\in\!\Si,\\
\quad u_*[\Si_v]&=\fd(v)\in H_2(X;\Z) &\qquad&\forall\,v\!\in\!\Ver\,.
\end{alignat*}

\vspace{.2in}

\noindent
For $J\!\in\!\cJ_{\om}^{\phi}$ and $\nu\!\in\!\Ga^{0,1}_{g,l;k}(X;J)^{\phi}$, let
$$\nu_{\ga;\rho}=\big\{q_{\ga;\rho}\!\times\!\id_X\big\}^*\nu \in 
\Ga^{0,1}_{\ga;\rho}(X;J)^{\phi}\,, \qquad
\wt\fM_{\ga;\rho}(J,\nu)=\wt\fM_{\ga;\rho}(J,\nu_{\ga;\rho})\,.$$
The group $G_{\ga;\rho}$ acts on $\wt\fM_{\ga;\rho}(J,\nu)$ 
by reparametrizing the domains of maps as usual.
Define
\BE{fMgadfn_e}\fM_{\ga}(J,\nu)=\bigsqcup_{\rho\in\Inv(\ga)}\!\!\!\!\!
\wt\fM_{\ga;\rho}(J,\nu)\big/G_{\ga;\rho}
\subset\ov\fM_{g,l;k}(X,B;J,\nu)^{\phi}\,.\EE
The stratification by the subspaces~\eref{fMgadfn_e} is analogous to the stratification 
by the subspaces in 
\cite[(5.1.5)]{MS} and on the right-hand side of \cite[(3.25)]{RT2}.
The number of nodes of the domains in the stratum~\eref{fMgadfn_e} is~$|\ga|$.
The dual graph of the element $\st(\u)$ of~$\R\ov\cM_{g,l;k}$
for any \hbox{$\u\!\in\!\fM_{\ga}(J,\nu)$} is~$\ov\ga$.
By \cite[Proposition~4.1.5]{MS}, the~set 
$$\cA_{g,l;k}^{\phi}(B;J,\nu)\equiv
\big\{\ga\!\in\!\cA_{g,l;k}^{\phi}(B)\!:\fM_{\ga}(J,\nu)\!\neq\!\eset\big\}$$
is finite for each pair~$(J,\nu)$.

\subsection{Strata of simple real maps: definitions}
\label{RTstr_subs2}

\noindent
Let~$\ga$ be as in~\eref{gadfn_e}. 
For each subset $\sV\!\subset\!\Ver$, let 
\BE{sVsEsF_e}\begin{aligned}
\sV_0(\ga)&=\big\{v\!\in\!\sV\!:\fd(v)\!=\!0\big\}, &
\sE_{\ga}(\sV)&=\big\{\{f_1,f_2\}\!\in\!\E(\ga)\!:\ve(f_1),\ve(f_2)\!\in\!\sV\big\},\\
\sF_{\ga}(\sV)&=\big\{f\!\in\!\Fl\!:\ve(f),\ve(\vt(f))\!\in\!\sV\big\}, &
\sF_{\ga}^*(\sV)&=\big\{f\!\in\!\Fl\!:\ve(f)\!\in\!\sV,\,\ve(\vt(f))\!\not\in\!\sV\big\}.
\end{aligned}\EE
The~tuple
$$\ga_{\sV}\equiv \big(\fg\!:\!\sV\!\lra\!\Z^{\ge0},
\ve\!:(S_{l;k}\!\cap\!\ve^{-1}(\sV))\!\sqcup\!\sF_{\ga}(\sV)\!\lra\!\sV,
\vt\!:\sF_{\ga}(\sV)\!\lra\!\sF_{\ga}(\sV)\big)$$
is then an $S_{l;k}\!\cap\!\ve^{-1}(\sV)$-marked graph 
(without an involution~$\ov\si$ as in~\eref{stgadfn_e}).
Let $\pi_0(\ga,\sV)$ be the set of connected components of~$\ga_{\sV}$ and
$$\ell(\ga,\sV)\equiv \big|\sE_{\ga}(\sV)\big|\!-\!|\sV|\!+\!\big|\pi_0(\ga,\sV)\big|$$
be the number of loops in~$\ga_{\sV}$, i.e.~the sum of the genera 
of its connected components.\\

\noindent
For $\sV\!\subset\!\Ver$ as above, let 
$$\sF_{\ga}^{\circ}\!(\sV)\subset\Fl\!\cap\!\ve^{-1}\big(\Ver\!-\!\sV\big)$$
be the collection of the flags~$f$ such that the edge $e\!\equiv\!\{f,\vt(f)\}$ 
disconnects a connected component~$\ga_{\sV'}$ of~$\ga_{\sV}$ from the rest of~$\ga$.
Denote by
$$\sF_{\ga}^{\dag}\!(\sV)\subset \sF_{\ga}^{\circ}\!(\sV)$$
the subcollection of the flags~$f$ such that \hbox{$S_{l;k}\!\cap\!\ve^{-1}(\sV')\!=\!\eset$}
for the subset $\sV'\!\subset\!\sV$ of the vertices separated from the remainder of~$\ga$
by the edge $\{f,\vt(f)\}$ .
In particular,
\BE{sEprp_e}\begin{split}
&\big|\E(\ga)\!-\!\sE_{\ga}(\sV)\!-\!\sE_{\ga}(\Ver\!-\!\sV)\big|
+\big|\sF_{\ga}^{\dag}\!(\sV)\big|+\big|S_{l;k}\!\cap\!\ve^{-1}(\sV)\big| \\
&\hspace{1in} \ge\big|\E(\ga)\!-\!\sE_{\ga}(\sV)\!-\!\sE_{\ga}(\Ver\!-\!\sV)\big|
+\big|\sF_{\ga}^{\circ}\!(\sV)\big| \ge 2\big|\pi_0(\ga,\sV)\big|\,.
\end{split}\EE
For $\rho\!\in\!\Inv(\ga)$, let
\begin{equation*}\begin{split}
\Ga^{0,1}_{\ga;\rho;\sV}(X;J)
=\Big\{\nu\!\in\!\Ga^{0,1}_{\ga;\rho}(X;J)^{\phi}\!:
\nu\big|_{(p_{\ga;v}^{\,*}\R\wt\cU_{\ga;\rho;v}\cap\R\wt\cU_{\ga;\rho}^*)\times X}
\!=\!0~\forall\,v\!\in\!\nV_{\R}(\ga)\!\cap\!\sV,&\\
\nu\big|_{(p_{\ga;v}^{\bu\,*}\R\wt\cU_{\ga;v}^{\bu}\cap\R\wt\cU_{\ga;\rho}^*)\times X}
\!=\!0~\forall\,v\!\in\!\nV_{\C}(\ga)\!\cap\!\sV&\Big\}.
\end{split}\end{equation*}

\vspace{.2in}

\noindent
Let $S(\ga)$ denote the collections of subsets $\sV\!\subset\!\Ver$ such~that 
$$\ale(\ga)\subset\sV,\qquad \si(\sV)=\sV, \qquad
\fg(v)\!=\!0~~\forall\,v\!\in\!\sV\,.$$
For each $\sV\!\in\!S(\ga)$, define
$$\R\cM_{\ga;\sV}=
 \prod_{v\in \sV_0\!\cap\nV_{\R}(\ga)}\hspace{-.25in}\R\cM_{\ga;\rho;v}
~\times
\prod_{\{v,\si(v)\}\subset\sV_0\!\cap\nV_{\C}(\ga)}\hspace{-.4in}\cM_{\ga;v}^{\bu}
\,.$$
By~\eref{sEprp_e} with $\sV$ replaced by~$\sV_0$,
\BE{cMgasV_e} 
\big(\!\dim_{\R}\!\R\cM_{\ga;\sV}\!+\!\big|\sF_{\ga}^{\dag}\!(\sV_0)\big|\big)
+\big|\sE_{\ga}(\sV_0)\big|+\big|\pi_0(\ga,\sV_0)\big| 
\ge  3\ell(\ga,\sV_0).\EE

\vspace{.2in}

\noindent
Let $\ov\ga\!\in\!\cA_{g,l;k}$ be as in~\eref{stgadfn_e}. 
Denote by $\cA(\ov\ga)$ the collection of pairs
$(\ga,\vp)$, where \hbox{$\ga\!\in\!\cA_{g',l;k}^{\phi}(B)$} is as in~\eref{gadfn_e} such~that 
\begin{gather}
\label{gared_e1}
\ov\Ver\subset\Ver, \quad \ale(\ga,\ov\ga)\!\equiv\!\Ver\!-\!\ov\Ver\in S(\ga),  \quad
\sF_{\ga}^{\circ}\!\big(\ale(\ga,\ov\ga)_0\big)\subset\sF_{\ga}\big(\ale(\ga,\ov\ga)\big),\\
\notag
\sF_{\ga}\big(\ov\Ver\big)\subset\ov\Fl\subset
\sF_{\ga}\big(\ov\Ver\big)\!\cup\!\sF_{\ga}^*\big(\ov\Ver\big), 
\quad \si(\ov\Fl)=\ov\Fl,\\
\notag
\ov\fg=\fg|_{\ov\Ver}, \quad \ov\ve\big|_{(S_{l;k}\cap\ve^{-1}(\ov\Ver))\sqcup\ov\Fl}
=\ve\big|_{(S_{l;k}\cap\ve^{-1}(\ov\Ver))\sqcup\ov\Fl}, \quad
\ov\vt|_{\sF_{\ga}(\ov\Ver)}=\vt|_{\sF_{\ga}(\ov\Ver)}, \quad
\ov\si=\si|_{\ov\Ver\sqcup\ov\Fl},
\end{gather}
and $\vp\!:S_{l;k}\!\cap\!\ve^{-1}(\ale(\ga,\ov\ga))\!\lra\!\sF_{\ga}^*(\ov\Ver)\!-\!\ov\Fl$
is a $(\si,\si_{l;k})$-equivariant injective map such~that
$$\ov\ve|_{S_{l;k}\cap\ve^{-1}(\ale(\ga,\ov\ga))}=\ve\!\circ\!\vp\!:
S_{l;k}\!\cap\!\ve^{-1}\big(\ale(\ga,\ov\ga)\big)\lra \ov\Ver.$$
Thus, $\ov\ga$ is obtained from~$\ga$ by
\begin{enumerate}[label=$\bu$,leftmargin=*]

\item dropping every vertex $v\!\in\!\ale(\ga,\ov\ga)$,

\item attaching each marked point~$f\!\in\!S_{l;k}$ carried by a vertex in $\ale(\ga,\ov\ga)$
to the vertex $\ve(\vp(f))\!\in\!\ov\Ver$,

\item identifying some pairs of the flags in $\sF_{\ga}^*(\ov\Ver)$
 into the edges of~$\ov\ga$ that are not contained in~$\sE_{\ga}(\ov\Ver)$.

\end{enumerate}
Define
\begin{alignat*}{3}
\ale(\ga)_0&=\ale(\ga)_0(\ga), &\quad 
\ale(\ga)_{\bu}&=\ale(\ga)\!-\!\ale(\ga)_0, &\quad
\ale(\ga)_0^c&=\Ver\!-\!\ale(\ga)_0,\\
\ale(\ga,\ov\ga)_0&=\ale(\ga,\ov\ga)_0(\ga), &\quad
\ale(\ga,\ov\ga)_{\bu}&=\ale(\ga,\ov\ga)\!-\!\ale(\ga,\ov\ga)_0, &\quad
\ale(\ga,\ov\ga)_0^c&=\Ver\!-\!\ale(\ga,\ov\ga)_0\,.
\end{alignat*}

\vspace{.2in}

\noindent
If $\ov\ga$ is the stabilization of~$\ga$, then $g'\!=\!g$, $\ale(\ga,\ov\ga)\!=\!\ale(\ga)$, and
\hbox{$(\ga,\vp_{\ga})\!\in\!\cA(\ov\ga)$} for a unique injective~map
\BE{gavpovga_e}\vp_{\ga}\!:
S_{l;k}\!\cap\!\ve^{-1}\big(\ale(\ga)\big)
\lra \sF_{\ga}^*(\ov\Ver)\!-\!\ov\Fl\,.\EE
A more elaborate example is depicted in Figure~\ref{gaovga_fig}.\\

\begin{figure}
\begin{pspicture}(-3,-1)(10,2.5)
\psset{unit=.3cm}
\psline[linewidth=.08](-4,0)(0,0)\pscircle*(0,0){.3}\pscircle*(-4,0){.3}
\psline[linewidth=.05](-5.5,1.5)(-4,0)\rput(-5.5,2.2){\sm{$1$}}\rput(-4,-1){\sm{$1$}}
\psline[linewidth=.05](1.5,1.5)(0,0)\rput(1.5,2.2){\sm{$2$}}\rput(0,-1){\sm{$1$}}
\rput(-2,-3){$\ov\ga$}
\psline[linewidth=.08](7,0)(19,0)\pscircle*(11,0){.3}\pscircle*(7,0){.3}
\pscircle*(15,0){.3}\pscircle*(19,0){.3}
\psline[linewidth=.05](4.5,2.5)(6,4)\rput(6.5,4){\sm{$1$}}\rput(7,-1){\sm{$1$}}
\psline[linewidth=.05](19,0)(17.5,1.5)\rput(17.2,2){\sm{$2$}}\rput(19,-1){\sm{$1$}}
\psline[linewidth=.08](7,0)(4.5,2.5)\pscircle*(4.5,2.5){.3}
\psline[linewidth=.08](19,0)(21.5,2.5)\pscircle*(21.5,2.5){.3}
\rput(13,-3){$\ga$}
\psline[linewidth=.08](25,0)(28.7,0)\psline[linewidth=.08](29.3,0)(36.7,0)
\psline[linewidth=.08](37.3,0)(41,0)
\pscircle(29,0){.3}\pscircle*(25,0){.3}
\pscircle*(33,0){.3}\pscircle(37,0){.3}\pscircle*(41,0){.3}
\psarc[linewidth=.08](33,0){4}{5}{175}
\psarc[linewidth=.08](33,0){8}{0}{180}\pscircle*(33,8){.3}
\psline[linewidth=.05](33,8)(33,5.2)\rput(33.5,5.2){\sm{$1$}}\rput(25,-1){\sm{$1$}}
\psline[linewidth=.05](41,0)(39.5,1.5)\rput(39.2,2){\sm{$2$}}\rput(41,-1){\sm{$1$}}
\rput(33,-3){$\ga'$}
\end{pspicture}  
\caption{A graph $\ov\ga$ as in~\eref{stgadfn_e} and graphs $\ga,\ga'$ as  in~\eref{gadfn_e}
such that $(\ga,\vp_{\ga})$ and $(\ga',\vp')$ are elements of $\cA(\ov\ga)$
for some~$\vp'$.
All vertices, edges, and marked points are taken to be real (i.e.~$\si$ acts trivially).
The value of~$\fg$ on the vertices with the number~1 next to them is~1;
its value on the remaining vertices is~0.
The value of~$\fd'$ on each of the unshaded vertices is~0;
the values of~$\fd$ on all vertices in the middle diagram and 
of~$\fd'$ on the shaded vertices in the last diagram are not~0.   
The graph~$\ov\ga$ is the stabilization of~$\ga$, but not of~$\ga'$.}
\label{gaovga_fig}
\end{figure}
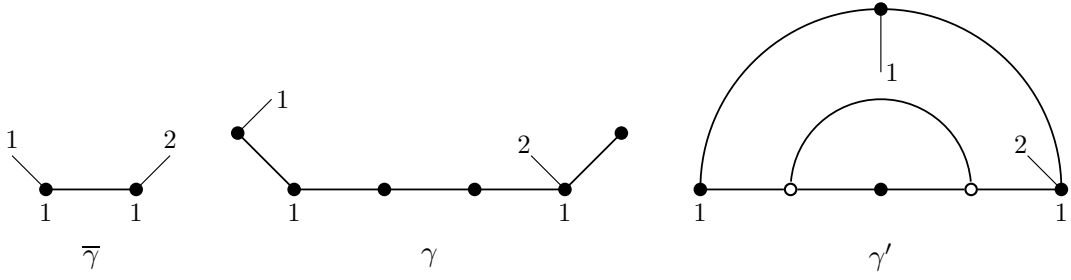

\noindent
Let $(\ga,\vp)\!\in\!\cA(\ov\ga)$ and $\rho\!\in\!\Inv(\ga)$. 
The reduction of~$\ga$ to~$\ov\ga$ described above determines a smooth~map
\BE{stgathdfn_e0} \st_{\ga,\vp}\!: \R\wt\cM_{\ga} \lra \R\wt\cM_{\ov\ga} 
\subset \R\ov\cM_{g,l;k}\,.\EE
This map lifts to smooth~maps
\begin{alignat*}{2}
q_{\ga,\vp;v}\!: p_{\ga;v}^{\,*}\R\wt\cU_{\ga;\rho;v}&\lra  \R\wt\cU_{g,l;k},
&\qquad &v\!\in\!\nV_{\R}(\ov\ga), \\
q_{\ga,\vp;v}^{\bu}\!: p_{\ga;v}^{\bu\,*}
\R\wt\cU_{\ga;v}^{\bu}&\lra \R\wt\cU_{g,l;k},
&\qquad &\big\{v,\si(v)\big\}\!\subset\!\nV_{\C}(\ov\ga).
\end{alignat*}
Each map $q_{\ga,\vp;v}$ restricts to a degree~one map from  a fiber~of
$$p_{\ga;v}^{\,*}\R\wt\cU_{\ga;\rho;v}\big|_{\st_{\ga,\vp}^{-1}(\cC)}\lra
\st_{\ga,\vp}^{-1}(\cC)\subset \R\wt\cM_{\ga}$$
onto a real irreducible component of a fiber of~\eref{RwtcU_e} over~$\R\wt\cM_{\ov\ga}$.
Each map $q_{\ga,\vp;v}^{\bu}$ restricts to a degree~one map from 
a connected component of a fiber 
$$p_{\ga;v}^{\bu\,*}\R\wt\cU_{\ga;v}^{\bu}\big|_{\st_{\ga,\vp}^{-1}(\cC)}\lra
\st_{\ga,\vp}^{-1}(\cC)\subset \R\wt\cM_{\ga}$$
onto a conjugate irreducible component of a fiber of~\eref{RwtcU_e} over~$\R\wt\cM_{\ov\ga}$.
In both cases, the marked and nodal points of the domain are preserved. 
However, in general these maps do not induce a continuous map 
even over a fiber of~\eref{pigarho_e}.\\

\noindent
For $J\!\in\!\cJ_{\om}^{\phi}$ and $\nu\!\in\!\Ga^{0,1}_{g,l;k}(X;J)^{\phi}$, define 
$\nu_{\ga,\vp;\rho}\!\in\!\Ga^{0,1}_{\ga;\rho;\ale(\ga,\ov\ga)}(X;J)^{\phi}$ by
\begin{equation*}\begin{split}
\nu_{\ga,\vp;\rho}\big|_{(p_{\ga;v}^{\,*}\R\wt\cU_{\ga;\rho;v}\cap\R\wt\cU_{\ga;\rho}^*)\times X}
&=\begin{cases}
\{q_{\ga,\vp;v}\!\times\!\id_X\big\}^*\nu,&\hbox{if}~v\!\in\!\nV_{\R}(\ov\ga);\\
0,&\hbox{if}~v\!\in\!\nV_{\R}(\ga)\!\cap\!\ale(\ga,\ov\ga);
\end{cases}\\
\nu_{\ga,\vp;\rho}\big|_{p_{\ga;v}^{\bu\,*}\R\wt\cU_{\ga;v}^{\bu}\cap\R\wt\cU_{\ga;\rho}^*)\times X}
&=\begin{cases}
\{q_{\ga,\vp;v}^{\bu}\!\times\!\id_X\big\}^*\nu,&\hbox{if}~v\!\in\!\nV_{\C}(\ov\ga);\\
0,&\hbox{if}~v\!\in\!\nV_{\C}(\ga)\!\cap\!\ale(\ga,\ov\ga).
\end{cases}
\end{split}\end{equation*}
Let
$$\wt\fM_{\ga,\vp;\rho}(J,\nu)=\wt\fM_{\ga;\rho}\big(J,\nu_{\ga,\vp;\rho}\big),
\qquad
\fM_{\ga,\vp}(J,\nu)=\bigsqcup_{\rho\in\Inv(\ga)}\!\!\!\!\!
\wt\fM_{\ga;\vp;\rho}(J,\nu)\big/G_{\ga;\rho}^{\circ}\,.$$
The maps~\eref{stgathdfn_e0} with $\rho\!\in\!\Inv(\ga)$
and the evaluation morphism~$\ev$ determine a continuous map
\BE{stgathdfn_e} 
\st_{\ga,\vp}\!\times\!\ev\!: \fM_{\ga,\vp}(J,\nu)
\lra  \R\ov\cM_{g,l;k}\times \big(X^l\!\times\!(X^{\phi})^k\big)\,.\EE

\vspace{.2in}

\noindent
For $\u\!\in\!\wt\fM_{\ga,\vp;\rho}(J,\nu)$ as in~\eref{Jnumap_e5} 
and $v\!\in\!\ale(\ga,\ov\ga)_0$,
the restriction~$u_v$ of~$u$ to the irreducible component $\Si_v\!\subset\!\Si$
corresponding to~$v$ is constant. Thus, 
\BE{fMcZ_e}\wt\fM_{\ga,\vp;\rho}(J,\nu)\approx  
\R\cM_{\ga;\ale(\ga,\ov\ga)}\times\wt\cZ_{\ga,\vp;\rho}'(J,\nu)\EE
for some space $\wt\cZ_{\ga,\vp;\rho}'(J,\nu)$ of tuples $(u_v)_{v\in\ale(\ga,\ov\ga)_0^c}$ 
of $\ve^{-1}(v)$-marked maps with matching conditions at 
the points indexed~by $\Fl\!-\!\sF_{\ga}^{\circ}(\ale(\ga,\ov\ga)_0)$.
By the last assumption in~\eref{gared_e1}, 
$$\ve(f)\not\in\ov\Ver, \quad \fd\big(\ve(v)\big)\neq0 
\qquad\forall~f\!\in\!\sF_{\ga}^{\circ}\!\big(\ale(\ga,\ov\ga)_0\big).$$
Dropping the marked points~$z_f$ with $f\!\in\!\sF_{\ga}^{\dag}\!(\ale(\ga,\ov\ga)_0)$,
we thus obtain a fiber bundle
\BE{cZfib_e} \wt\cZ_{\ga,\vp;\rho}'(J,\nu) \lra \wt\cZ_{\ga,\vp;\rho}(J,\nu)\EE
with $|\sF_{\ga}^{\dag}\!(\ale(\ga,\ov\ga)_0)|$-dimensional fibers for 
some space $\wt\cZ_{\ga,\vp;\rho}(J,\nu)$ 
of tuples $(u_v)_{v\in\ale(\ga,\ov\ga)_0^c}$ of maps with marked points 
indexed~by $\ve^{-1}(v)\!-\!\sF_{\ga}^{\dag}\!(\ale(\ga,\ov\ga)_0)$
and with the same matching conditions as~before.\\

\noindent
The $G_{\ga;\rho}^{\circ}$-action
on the left-hand side of~\eref{fMcZ_e} corresponds to an action 
on the last factor on the  right-hand side.
The latter in turn descends to an action on the right-hand side of~\eref{cZfib_e}.
Let
$$\cZ_{\ga,\vp}'(J,\nu) =\bigsqcup_{\rho\in\Inv(\ga)}\!\!\!\!\!
\wt\cZ_{\ga,\vp;\rho}'(J,\nu) \big/G_{\ga;\rho}^{\circ}
\qquad\hbox{and}\qquad
\cZ_{\ga,\vp}(J,\nu) =\bigsqcup_{\rho\in\Inv(\ga)}\!\!\!\!\!
\wt\cZ_{\ga,\vp;\rho}(J,\nu) \big/G_{\ga;\rho}^{\circ}\,.$$
Thus,
\BE{fMcZ_e2}\fM_{\ga,\vp}(J,\nu)\approx  
\R\cM_{\ga;\ale(\ga,\ov\ga)}\times\cZ_{\ga,\vp}'(J,\nu)\EE
and the fibers of the projection
\BE{cZfib_e2} \cZ_{\ga,\vp}'(J,\nu) \lra \cZ_{\ga,\vp}(J,\nu)\EE
are $|\sF_{\ga}^{\dag}\!(\ale(\ga,\ov\ga)_0)|$-dimensional.
The map~\eref{stgathdfn_e} factors through the projection from
the left-hand side of~\eref{fMcZ_e2}
to the right-hand side in~\eref{cZfib_e2} and a continuous~map
\BE{redstev_e} \st_{\ga,\vp}\!\times\!\ev\!:\cZ_{\ga,\vp}(J,\nu)
\lra  \R\ov\cM_{g,l;k}\times \big(X^l\!\times\!(X^{\phi})^k\big)\,.\EE

\vspace{.2in}
 
\noindent
Denote~by 
\BE{fMga_e2} \fM_{\ga,\vp}^*(J,\nu)\subset\fM_{\ga,\vp}(J,\nu)
\qquad\hbox{and}\qquad
\wt\fM_{\ga,\vp;\rho}^*(J,\nu)\subset\wt\fM_{\ga,\vp;\rho}(J,\nu)\EE
the subspaces consisting of maps as in~\eref{Jnumap_e5} such~that 
$u_v$ is simple for every \hbox{$v\!\in\!\ale(\ga,\ov\ga)_{\bu}$} 
and 
$$u(\Si_{v_1})\neq u(\Si_{v_2}) \qquad\forall~v_1,v_2\in\ale(\ga,\ov\ga)_{\bu},
~v_1\!\neq\!v_2.$$
Denote~by 
$$\cZ_{\ga,\vp}'^*(J,\nu)\subset\cZ_{\ga,\vp}'(J,\nu) \qquad\hbox{and}\qquad
\cZ_{\ga,\vp}^*(J,\nu)\subset\cZ_{\ga,\vp}(J,\nu) $$
the image of $\fM_{\ga,\vp}^*(J,\nu)$ under the projection to the last component
in~\eref{fMcZ_e2} and the image of $\cZ_{\ga,\vp}'^*(J,\nu)$ under~\eref{cZfib_e2}.
The splitting~\eref{fMcZ_e2} and the fibration~\eref{redstev_e} restrict to
a splitting 
$$ \fM_{\ga,\vp}^*(J,\nu)\approx  
\R\cM_{\ga;\ale(\ga,\ov\ga)}\times\cZ_{\ga,\vp}'^*(J,\nu)$$
and fibration
$$ \cZ_{\ga,\vp}'^*(J,\nu)\!=\!\cZ_{\ga,\vp}'(J,\nu)
\big|_{\cZ_{\ga,\vp}^*(J,\nu)} \lra \cZ_{\ga,\vp}^*(J,\nu)$$
with $|\sF_{\ga}^{\dag}\!(\ale(\ga,\ov\ga)_0)|$-dimensional fibers.
The elements of the base of this fibration are real analogues 
of \sf{reduced GU-maps} of \cite[Definition~3.10]{RT2}.

\subsection{Strata of simple real maps: properties}
\label{RTstr_subs3}

\noindent
Let $\ga\!\in\!\cA_{g,l;k}^{\phi}(B)$ be as in~\eref{gadfn_e} and 
$\ov\ga$ be the stabilization of~$\ga$.
Denote by $\cA(\ga)\!\subset\!\cA(\ov\ga)$ the subset of pairs $(\ga',\vp)$
with
\BE{gaprdfn_e}\begin{split} 
\ga'\equiv 
\Big((\fg',\fd')\!:\!\Ver'\!\lra\!\Z^{\ge0}\!\oplus\!H_2(X;\Z),
\ve'\!:S_{l;k}\!\sqcup\!\Fl'\!\lra\!\Ver',\vt'\!:\Fl'\!\lra\!\Fl',&\\
\si'\!:\Ver'\!\sqcup\!\Fl'\!\lra\!\Ver'\!\sqcup\!\Fl'&\Big)
\in\cA_{g',l;k}^{\phi}(B')
\end{split}\EE
such that
\BE{redgcond_e}
\fd'\big|_{\ov\Ver}=\fd\big|_{\ov\Ver}, \qquad
g'\!+\!\big|\ale(\ga',\ov\ga)_{\bu}\big|=g\!+\!\big|\ale(\ga)_{\bu}\big|,\EE
and there are~maps 
\BE{redmapcond_e}\ka\!:\ale(\ga)_{\bu}\lra\ale(\ga',\ov\ga)_{\bu}
\qquad\hbox{and}\qquad \vr\!:\ale(\ga)_{\bu}\lra\Z^+\EE
so that $\ka$ is surjective and $(\si',\si)$-equivariant, 
$\vr$ is $\si$-invariant, and 
\BE{degprod_e}
\vr(v)\fd'\big(\!\ka(v)\!\big)=\fd(v) \qquad\forall\,v\!\in\!\ale(\ga)_{\bu}\,.\EE
In particular, 
\BE{degdiff_e}
\ell\big(\ga',\ale(\ga',\ov\ga)_0\big)\le g'\!-\!g, 
\quad B'=B-\sum_{v'\in\ale(\ga',\ov\ga)_{\bu}}\!\!\!
\bigg(\sum_{v\in\ka^{-1}(v')}\!\!\!\!\!\!\!\vr(v)~-\!1\!\!\bigg)\fd'(v')\,.\EE

\vspace{.2in}

\noindent
For example, $(\ga,\vp_{\ga})\!\in\!\cA(\ga)$. 
The map~$\ka$ is the identity in this case, while $\vr$ is the constant function
with value~1.
For suitable choices of the values of $\fd$ and $\fd'$ on
the shaded vertices in the middle and last diagrams in Figure~\ref{gaovga_fig},
$(\ga',\vp')\!\in\!\cA(\ga)$.

\begin{lmm}\label{mapred_lmm}
Suppose  $(X,\om,\phi)$ is a real symplectic manifold,
\hbox{$g,l,k\!\in\!\Z^{\ge0}$} with \hbox{$2(g\!+\!l)\!+\!k\!\ge\!3$},
\hbox{$B\!\in\!H_2(X;\Z)$}, and $\ga\!\in\!\cA_{g,l;k}^{\phi}(B)$.
Let $\ov\ga\!\in\!\cA_{g,l;k}$ be the stabilization of~$\ga$ and
\hbox{$(J,\nu)\!\in\!\cH_{g,l;k}^{\om,\phi}(X)$}.
For every element~$\u$ of $\ov\fM_{g,l;k}(X,B;J,\nu)^{\phi}$, 
there exist 
\BE{mapred_e0}(\ga',\vp)\!\in\!\cA(\ga),~\u'\!\in\!\cZ_{\ga',\vp}^*(J,\nu)
\qquad\hbox{s.t.}\quad
\big\{\ev\!\times\!\st\big\}(\u)=\big\{\ev\!\times\!\st_{\ga',\vp}\big\}(\u').\EE
\end{lmm}

\vspace{.2in}

\noindent
Let $\ga$, $\ov\ga$, and $\u$ be as in~\eref{gadfn_e}, 
\eref{stgadfn_e}, and~\eref{Jnumap_e5}, respectively.
We construct $(\ga',\vp)$, with $\ga'$ as in~\eref{gaprdfn_e},
$$\u'\equiv \big(u'\!:\Si'\!\lra\!X,(z_f')_{f\in S_{l;k}},\si',\fJ'\big)\\
\equiv \big(u_v'\!:\Si_v'\!\lra\!X,(z_f')_{f\in\ve'^{-1}(v)},\si_v',\fJ'\big)_{v\in\Ver'}
\in\fM_{\ga',\vp}^*(J,\nu),$$
and associated maps~\eref{redmapcond_e} explicitly below.
The image of this element~$\u'$ of $\fM_{\ga',\vp}^*(J,\nu)$ 
under the projection to the first component in~\eref{fMcZ_e2} 
and~\eref{cZfib_e2} is a desired element of~$\cZ_{\ga',\vp}^*(J,\nu)$.
The map~$\u'$ keeps the irreducible components~$u_v$ of~$\u$ with their marked
and nodal points that either correspond to the domains~$\Si_v$ preserved
by the stabilization~\eref{qgarhodfn_e} or are constant maps to~$X$. 
These components are indexed by the vertices~$v$ in $\ov\Ver\!\sqcup\!\ale(\ga)_0$;
the remaining irreducible components of~$\u$
are indexed by the vertices~$v$ in $\ale(\ga)_{\bu}$.\\

\noindent
We replace each $u_v$ with $v\!\in\!\ale(\ga)_{\bu}$ by simple map $u_v'$
with the same image and every set of such maps that have the same image in~$X$
by a single simple map~$u_{v'}'$.
The collection of maps~$u_{v'}'$ obtained in this way is indexed by the vertices~$v'$
in the set $\ale(\ga',\ov\ga)_{\bu}$ below, which is thus a quotient of $\ale(\ga)_{\bu}$.
This two-step replacement may send distinct marked and nodal points~$z_f$ 
of possibly different components~$u_v$ of~$\u$ into the same point of
the domain $\Si_{v'}\!\approx\!\P^1$ of~$u_{v'}'$.
We resolve such an accumulation by adding an extra contracted bubble $\Si_{v''}\!\approx\!\P^1$.
The collection of the extra bubbles is indexed by the set~$\sV_0$ below.\\

\noindent 
The description above identifies the flags $\Fl$ of~$\ga$ with 
the subset of the flags~$\Fl'$ not forming the edges $e\!=\!\{v',v''\}$
as in the previous paragraph.
We identify these flags into nodes of~$\u'$ in the same way as for~$\u$.
This procedure preserves the evaluation maps at the marked points.
Since the irreducible components~$u_v$ and~$u_v'$ with their marked points 
are the same whenever $v\!\in\!\ov\Ver$, $\u'$ remembers~$\st(\u)$.\\

\noindent
An analogue of this procedure in the complex case without the resolution step
is sketched after \cite[Definition~3.10]{RT2}.
A similar procedure with $g\!=\!0$ and $\nu\!=\!0$ is described in the proof of
\cite[Proposition~6.1.2]{MS}.
The tuple of maps produced by our two-step replacement procedure is in fact 
a real analogue of reduced GU-map of \cite[Definition~3.10]{RT2} and
is a desired element of $\cZ_{\ga',\vp}^*(J,\nu)$.
The resolution step demonstrates that this tuple is indeed an element of 
$\cZ_{\ga',\vp}^*(J,\nu)$ by producing an associated stable map 
in~$\fM_{\ga',\vp}^*(J,\nu)$.

\begin{proof}[{\bf{\emph{Proof of Lemma~\ref{mapred_lmm}}}}]
By definition,
\begin{gather*}
\Ver=\ov\Ver\sqcup\ale(\ga)_{\bu}\sqcup\ale(\ga)_0, \qquad
\Ver'=\ov\Ver\sqcup\ale(\ga',\ov\ga)_{\bu}\sqcup\ale(\ga',\ov\ga)_0, \\
(\fg',\fd')\big|_{\ov\Ver}=(\fg,\fd)|_{\ov\Ver}, \quad
\fg'|_{\ale(\ga',\ov\ga)_{\bu}}=0, \quad
(\fg',\fd')\big|_{\ale(\ga',\ov\ga)_0}=(0,0), \quad
\si'\big|_{\ov\Ver}=\si\big|_{\ov\Ver}\,.
\end{gather*}
We produce $\ga'$ so that 
\begin{gather*}
\ale(\ga',\ov\ga)_0=\ale(\ga)_0\!\sqcup\!\sV_0, \quad
\Fl'=\Fl\!\sqcup\!\sF_0, \quad |\sF_0|=2|\sV_0|, \quad
\vt'|_{\Fl}=\vt, \quad \si'|_{\ale(\ga)_0\sqcup\Fl}=\si|_{\ale(\ga)_0\sqcup\Fl},\\
\ve'\big|_{(S_{l;k}\sqcup\Fl)-\ve^{-1}(\ale(\ga)_{\bu})}
=\ve\big|_{(S_{l;k}\sqcup\Fl)-\ve^{-1}(\ale(\ga)_{\bu})}, \quad 
\ve'\big(\ve^{-1}(\ale(\ga)_{\bu})\!\sqcup\!\sF_0\big)\subset 
\ale(\ga',\ov\ga)_{\bu}\!\sqcup\!\sV_0,
\end{gather*}
for some finite (possibly empty) sets $\sV_0$ of additional degree~0 vertices
and $\sF_0$ of additional flags.
Along with the surjectivity of~$\ka$, the condition $|\sF_0|\!=\!2|\sV_0|$ implies
the second property in~\eref{redgcond_e}.\\ 

\noindent
For each $v\!\in\!\ov\Ver\!\sqcup\!\ale(\ga)_0$, define
\BE{mapred_e3} u_v'\!=\!u_v\!:\Si_v\lra X, \qquad z_f'=z_f~\forall\,
f\!\in\!\ve'^{-1}(v)\!=\!\ve^{-1}(v).\EE
Thus, the components $u_v$ and $u_v'$ of~$\u$ and~$\u'$, respectively, corresponding
to each element~$v$ of $\ov\Ver\!\sqcup\!\ale(\ga)_0$ are the same and 
carry the same marked and nodal points.\\

\noindent
For each $v\!\in\!\ale(\ga)_{\bu}$, there exist a branched cover
$h_v\!:\P^1\!\lra\!\P^1$ and a simple $J$-holomorphic map $u_v'\!:\P^1\!\lra\!X$ 
such that $u_v\!=\!u_v'\!\circ\!h_v$;  see \cite[Proposition~2.5.1]{MS}.
If $u_v''$ is another simple $J$-holomorphic map such that 
$u_v''(\P^1)\!=\!u_v'(\P^1)$, then \hbox{$u_v''\!=\!u_v'\!\circ\!h$}
for some $h\!\in\!\Aut(\P^1)$.
Let
$$\ka\!:\ale(\ga)_{\bu}\lra 
\ale(\ga',\ov\ga)_{\bu}\!\equiv\!\ale(\ga)_{\bu}\big/\!\!\sim,\quad
v_1\sim v_2~~\hbox{if}~~u_{v_1}(\P^1)\!=\!u_{v_2}(\P^1),$$
be the quotient map.
Define
\begin{alignat*}{2}
\vr\!:\ale(\ga)_{\bu}&\lra\Z^+, &\qquad \vr(v)&=\deg h_v, \\
\fd'\!:\ale(\ga',\ov\ga)_{\bu}&\lra H_2(X;\Z), &\qquad 
\fd'(v')&=\big\{u_{v*}'\big\}_*[\P^1]~~\hbox{if}~\ka(v)\!=\!v'\,.
\end{alignat*}
These two maps are well-defined, i.e.~independent of the choices of $h_v$, $u_v'$,
and \hbox{$v\!\in\!\ka^{-1}(v')$}, and satisfy~\eref{degprod_e}.
The condition 
$$\phi\big(u_{v_1}(\P^1)\big)=u_{v_2}\big(\P^1\big) 
\qquad\forall~v_1\!\in\!\ka^{-1}(v'),\,v_2\!\in\!\ka^{-1}\big(\si'(v')),\,
v'\!\in\!\ale(\ga',\ov\ga)_{\bu}$$
determines an involution~$\si'$ on~$\ale(\ga',\ov\ga)_{\bu}$.\\

\noindent
For each $v'\!\in\!\ale(\ga',\ov\ga)_{\bu}$, we pick $v\!\in\!\ka^{-1}(v')$ 
and first set 
$$u_{v'}'\!=\!u_v'\!:\P^1\lra X\,.$$
For each $v'\!\in\!\ale(\ga',\ov\ga)_{\bu}$, the map $\phi\!\circ\!u_{v'}'\!\circ\!\tau_1$
is simple and $J$-holomorphic and has the same image as~$u_{\si'(v')}'$.
If $\si'(v')\!\neq\!v'$, we can thus assume that 
\BE{mapred_e7a}u_{\si'(v')}'\!=\!\phi\!\circ\!u_{v'}'\!\circ\!\tau_1\!:\P^1\lra X\,.\EE
If $\si'(v')\!=\!v'$, the involution~$\phi$ on $u_{v'}'(\P^1)$ determines 
a anti-holomorphic involution~$\si_{v'}'$ on~$\P^1$ such~that
\BE{mapred_e7b}\phi\circ u_{v'}'= u_{v'}'\circ \si_{v'}'\,.\EE
By replacing $u_{v'}'$ with  $u_{v'}'\!\circ\!h$ for some $h\!\in\!\Aut(\P^1)$,
we can assume that $\si_{v'}'\!\in\!\{\tau,\eta\}$.
We set $\si_{v'}'\!=\!\tau_1$ if $\si'(v')\!\neq\!v'$.\\

\noindent
For all $v'\!\in\!\ale(\ga',\ov\ga)_{\bu}$ and $v\!\in\!\ka^{-1}(v')$,
$u_{v'}'(\P^1)\!=\!u_v(\P^1)$.
For every \hbox{$f\!\in\!\ve^{-1}(v)$}, 
there thus exists \hbox{$\ch{z}_f'\!\in\!\P^1$}
such that $u_{v'}'(\ch{z}_f')\!=\!u_v(z_f)$.
Since
$$u_{\si'(v')}'\big(\ch{z}_{\si(f)}'\big) 
=u_{\si(v)}(z_{\si(f)})=\phi\big(u_v(z_f)\big)
=\phi\big(u_{v'}'(\ch{z}_f')\big)
=u_{\si'(v')}'\big(\si_{v'}'(\ch{z}_f')\big),$$ 
these points $\ch{z}_f'$ can be chosen so~that
\BE{mapred_e15}
\ch{z}_{\si(f)}'=\si_{v'}'\big(\ch{z}_f'\big)
\qquad\forall~f\!\in\!\ve^{-1}(v),\,v\!\in\!\ka^{-1}(v'),\,
v'\!\in\!\ale(\ga',\ov\ga)_{\bu}.\EE

\vspace{.2in}

\noindent
For each $v'\!\in\!\ale(\ga',\ov\ga)_{\bu}$, let
$$\ve^{-1}\big(\ka^{-1}(v')\big)=\bigsqcup_{w\in\sV_{0;v'}}\!\!\!\!\sF_w$$
be the decomposition so that 
\begin{alignat*}{2}
\ch{z}_{f_1}'&=\ch{z}_{f_2}' &\quad&\hbox{if}~~f_1,f_2\!\in\!\sF_w,~w\!\in\!\sV_{0;v'},\\
\ch{z}_{f_1}'&\neq\ch{z}_{f_2}' &\quad&\hbox{if}~~f_1\!\in\!\sF_{w_1},~f_2\!\in\!\sF_{w_2},~
w_1,w_2\!\in\!\sV_{0;v'},~w_1\!\neq\!w_2.
\end{alignat*}
By~\eref{mapred_e15}, there is an involution $\si'$ 
$$\si'\!:\sV_0\equiv\bigsqcup_{v'\in\ale(\ga',\ov\ga)_{\bu}}\!\!\!\!\!\!\!\!\!
\big\{w\!\in\!\sV_{0;v'}\!:|\sF_w|\!\ge\!2\big\}\lra \sV_0$$
such that $\si(f)\!\in\!\sF_{\si'(w)}$ for all $f\!\in\!\sF_w$ and
$w\!\in\!\sV_0$.\\

\noindent
Define 
\begin{gather*}
\si'\!:\sF_0\!\equiv\!\bigsqcup_{w\in\sV_0}
\!\!\!\big\{f_w^-,f_w^+\big\}\lra \sF_0, \qquad
\si'\big(f_w^{\pm}\big)=f_{\si'(w)}^{\pm}\,,\\
\ve'\!:\sF_0\!\sqcup\!\ve^{-1}(\ale(\ga)_{\bu})\lra\Ver', \quad
\ve'(f)=\begin{cases}v',&\hbox{if}~w\!\in\!\sV_{0;v'},
f\!=\!f_w^-~\hbox{or}~f\!\in\!\sF_w,\,|\sF_w|\!=\!1;\\
w,&\hbox{if}~w\!\in\!\sV_{0;v'},
f\!=\!f_w^+~\hbox{or}~f\!\in\!\sF_w,\,|\sF_w|\!\ge\!2.
\end{cases}
\end{gather*}
This completes the specification of $\ga'$.
We take 
$$\vp\!=\!\vp_{\ga}\!:S_{l;k}\!\cap\!\ve'^{-1}\big(\ale(\ga',\ov\ga)\big)\!=\!
S_{l;k}\!\cap\!\ve^{-1}\big(\ale(\ga)\big)
\lra \sF_{\ga}^*(\ov\Ver)\!-\!\ov\Fl\!=\!\sF_{\ga'}^*(\ov\Ver)\!-\!\ov\Fl$$
to be the injective map corresponding to the contraction of~$\ga$ to~$\ov\ga$.
This defines \hbox{$(\ga',\vp)\!\in\!\cA(\ga)$}.\\

\noindent
Let $v'\!\in\!\ale(\ga',\ov\ga)_{\bu}$ and $w\!\in\!\sV_{0;v'}$.
If $\sF_w\!=\!\{f\}$ consists of a single element, we take 
$$z_f'=\ch{z}_f'\in \Si_{\ve'(f)}\equiv\Si_{v'}=\P^1\,.$$
If $|\sF_w|\!\ge\!2$ and $f\!\in\!\sF_w$, we take
\begin{gather*}
z_{f_w^-}'=\ch{z}_f'\in \Si_{\ve'(f_w^-)}\equiv \Si_{v'}=\P^1, \qquad
z_{f_w^+}'=1\in \Si_{\ve'(f_w^+)}\equiv \Si_w=\P^1, \\
u_w'\!:\Si_w\lra X,~u_w'(z)=u_{v'}'(\ch{z}_f') \quad\forall\,z\!\in\!\Si_w,\\
\si_w'\!=\!\tau\!:\Si_w\lra\Si_w~~\hbox{if}~\si'(w)\!=\!w, \quad
\si_w'\!=\!\tau_1\!:\Si_w\lra\Si_{\si'(w)}~~\hbox{if}~\si'(w)\!\neq\!w\,.
\end{gather*}
These definitions are independent of the choice of $f\!\in\!\sF_w$.
In this case, we also choose distinct~points 
$$z_f'\in \Si_w\!-\!\{1\}\subset\Si_{\ve'(f)},
\qquad f\!\in\!\sF_w\,.$$ 
They can be chosen so~that
$$\si_w'\big(z_f'\big)=z_{\si(f)}'
\qquad\forall~f\!\in\!\sF_w,\,w\!\in\!\sV_0\,.$$

\vspace{.2in}

\noindent
The maps $u_{v'}\!:\P^1\!\lra\!X$ with $v'\!\in\!\ale(\ga',\ov\ga)_{\bu}$ 
such that $|\ve'^{-1}(v')|\!\le\!2$ can be reparametrized to
achieve~\eref{nodecondC_e} and~\eref{nodecondR_e}
with $z_f,\ga,\ve$ replaced by $z_f',\ga',\ve'$,
while preserving the conditions~\eref{mapred_e7a} and~\eref{mapred_e7b}.
This completes the specification of a map~$\u'$  with dual graph~$\ga'$ that 
satisfies~\eref{nodecondC_e} and~\eref{nodecondR_e}.
Since 
\begin{gather*}
u_{\ve'(f)}'(z_f')=u_{\ve(f)}(z_f)=u_{\ve(\vt(f))}\big(z_{\vt(f)}\big)
=u_{\ve'(\vt'(f))}'\big(z_{\vt'(f)}'\big) \quad\forall\,f\!\in\!\Fl, \\
u_{v'}'\big(z_{f_w^-}'\big)=u_w'\big(z_{f_w^+}'\big)\quad
\forall~w\!\in\!\sV_0\,,
\end{gather*}
this map is continuous at the nodes.
Furthermore,
\BE{mapred_e25} u_{\ve'(f)}'(z_f')=u_{\ve(f)}(z_f) \qquad\forall~f\!\in\!S_{l;k}\,.\EE

\vspace{.2in}

\noindent
Let $\rho'\!\in\!\cA(\ga')$ be given by $\rho'(v')\!=\!\si_{v'}'$  
whenever $v'\!\in\!\ale(\ga',\ov\ga)_{\bu}$, $\si'(v')\!=\!v'$, and 
\hbox{$|\ve'^{-1}(v')|\!\le\!2$}.
By~\eref{mapred_e3} and the choice of~$\vp$, 
\BE{mapred_e27}\begin{split}
q_{\ga',\vp;v}\!=\!q_{\ga,\vp_{\ga};v}\!: \Si_v&\lra \R\wt\cU_{g,l;k}
\quad\forall~v\!\in\!\nV_{\R}(\ov\ga),\\
q_{\ga',\vp;v}^{\bu}\!=\!q_{\ga,\vp_{\ga};v}^{\bu}\!: \Si_v&\lra \R\wt\cU_{g,l;k}
\quad\forall~v\!\in\!\nV_{\C}(\ov\ga).
\end{split}\EE
Thus,
$$\nu_{\ga,\vp';\rho'}\big|_{\Si_v}=\nu_{\ga,\vp;\rho}|_{\Si_v},\quad
\dbar_{J,\fJ}u_v'\big|_z=\dbar_{J,\fJ}u_v\big|_z
=\nu_{\ga,\vp;\rho}\big(z,u_v(z)\big)
=\nu_{\ga',\vp;\rho'}\big(z,u_v'(z)\big)
~~\forall\,z\!\in\!\Si_v$$
for all $v\!\in\!\Ver'$.
By the construction, $u_{v'}'$ is a simple map for every $v'\!\in\!\ale(\ga',\ov\ga)_{\bu}$
and 
$$u_{v_1'}'(\P^1)\neq u_{v_2'}'(\P^1) \qquad
\forall~v_1',v_2'\!\in\!\ale(\ga',\ov\ga)_{\bu},~v_1'\!\neq\!v_2'.$$
Thus, $\u'\!\in\!\wt\fM_{\ga',\vp;\rho'}^*(J,\nu)$.
By~\eref{mapred_e25} and~\eref{mapred_e27},  
$[\u']$ satisfies the condition in~\eref{mapred_e0}.
\end{proof}

\noindent
Let $\bI\!=\![0,1]$.
For $(J,\nu)$ and $(J',\nu')$ in $\cH_{g,l;k}^{\om,\phi}(X)$, 
define
$$\sP(J,\nu;J',\nu') =
\big\{\al\!:[0,1]\!\lra\!\cH_{g,l;k}^{\om,\phi}(X)\!:
\al(0)\!=\!(J,\nu),~\al(1)\!=\!(J',\nu')\big\}$$
to be the space of paths from $(J,\nu)$ and $(J',\nu')$. 
For any such path~$\al$, let 
\begin{equation*}\begin{split}
\ov\fM_{g,l;k}(B;\al)^{\phi}&=\big\{\big(t,[\u]\big)\!:t\!\in\!\bI,\,
[\u]\!\in\!\ov\fM_{g,l;k}\big(X,B;\al(t)\big)^{\phi}\big\},\\
\fM_{g,l;k}^{\star}(B;\al)^{\phi}&=\big\{\big(t,[\u]\big)\!:t\!\in\!\bI,\,
[\u]\!\in\!\fM_{g,l;k}^{\star}\big(X,B;\al(t)\big)^{\phi}\big\}.
\end{split}\end{equation*}
If in addition $\ov\ga\!\in\!\cA_{g,l;k}$ and $(\ga,\vp)\!\in\!\cA(\ov\ga)$, let 
$$\cZ_{\ga,\vp}^*(\al)=\big\{\big(t,[\u]\big)\!:t\!\in\!\bI,\,
[\u]\!\in\!\cZ_{\ga,\vp}^*\big(\al(t)\big)\big\}\,.$$
These spaces are again topologized as in \cite[Section~3]{LT} so that the~maps
\begin{alignat}{1}
\label{alRTreal_e}
\st\!\times\!\ev\!:\fM_{g,l;k}^{\star}(B;\al)^{\phi} &\lra 
\R\ov\cM_{g,l;k}\times \big(X^l\!\times\!(X^{\phi})^k\big),\\
\label{alredstev_e}
\st_{\ga,\vp}\!\times\!\ev\!:
 \cZ_{\ga,\vp}^*(\al) &\lra \R\ov\cM_{g,l;k}\times \big(X^l\!\times\!(X^{\phi})^k\big),
\end{alignat}
induced by~\eref{evst_e} and~\eref{redstev_e} are continuous.\\

\noindent
For \hbox{$g,l,k\!\in\!\Z^{\ge0}$} and $B\!\in\!H_2(X;\Z)$, let
\BE{RTprp_e1}  \dim_{g,l;k}(B)
=\blr{c_1(TX),B}\!+\!(n\!-\!3)(1\!-\!g)\!+\!2l\!+\!k.\EE
For $\ov\ga\!\in\!\cA_{g,l;k}$ and $(\ga,\vp)\!\in\!\cA(\ov\ga)$
with $\ga\!\in\!\cA_{g',l;k}^{\phi}(B')$, let 
\BE{RTprp_e}\begin{split}
\dim_{\ga,\vp}=\dim_{g',l;k}(B')-|\ga|
+n\,\ell\big(\ga,\ale(\ga,\ov\ga)_0\big)-
\big(\!\dim_{\R}\!\R\cM_{\ga;\ale(\ga,\ov\ga)}\!+\!
\big|\sF_{\ga}^{\dag}\!(\ale(\ga,\ov\ga)_0)\big|\big).
\end{split}\EE
Propositions~\ref{RTreg_prp} and~\ref{RTreg2_prp} 
below are  analogous to \cite[Theorem~6.2.6]{MS} and 
\cite[Theorem~3.16]{RT2}; they are established in Section~\ref{trans_sec}.
 
\begin{prp}\label{RTreg_prp}
Let  $(X,\om,\phi)$ be a compact real symplectic $2n$-manifold.
For all \hbox{$g,l,k\!\in\!\Z^{\ge0}$} with \hbox{$2(g\!+\!l)\!+\!k\!\ge\!3$}
and $B\!\in\!H_2(X;\Z)$, there exists a Baire~subset 
\BE{whcHdfn_e}\wh\cH_{g,l;k}^{\om,\phi}(X)\subset\cH_{g,l;k}^{\om,\phi}(X)\EE
of second category such~that for every $(J,\nu)\!\in\!\wh\cH_{g,l;k}^{\om,\phi}(X)$
\begin{enumerate}[label=(\arabic*),leftmargin=*]

\item\label{mainstr_it} 
$\fM_{g,l;k}^{\star}(X,B;J,\nu)^{\phi}$ is a smooth manifold of 
dimension~\eref{RTprp_e1} and the restriction of~\eref{evst_e} to
$\fM_{g,l;k}^{\star}(X,B;J,\nu)^{\phi}$
is a smooth map,

\item\label{cZstr_it}  $\cZ_{\ga,\vp}^*(J,\nu)$ is a smooth manifold of dimension~\eref{RTprp_e} 
and the restriction of~\eref{redstev_e} to $\cZ_{\ga,\vp}^*(J,\nu)$ is a smooth map 
for all \hbox{$(\ga,\vp)\!\in\!\cA(\ov\ga)$}
satisfying \hbox{$\ov\ga\!\in\!\cA_{g,l;k}$}, 
\hbox{$\ga\!\in\!\cA_{g',l;k}^{\phi}(B')$}, $B'\!\in\!H_2(X;\Z)$,
and  $\om(B')\!\le\!\om(B)$.

\end{enumerate}
\end{prp}

\begin{prp}\label{RTreg2_prp}
Let  $(X,\om,\phi)$,  $g,l,k$, $B$, and $\wh\cH_{g,l;k}^{\om,\phi}(X)$ be as
in Proposition~\ref{RTreg_prp}.
For all elements $(J,\nu)$ and $(J',\nu')$  of $\wh\cH_{g,l;k}^{\om,\phi}(X)$,
there exists a Baire subset 
\BE{whsPdfn_e}\wh\sP(J,\nu;J',\nu')\!\subset\!\sP(J,\nu;J',\nu')\EE
of second category such~that for every $\al\!\in\!\wh\sP(J,\nu;J',\nu')$
\begin{enumerate}[label=(\arabic*),leftmargin=*]

\item\label{mainstr2_it} $\fM_{g,l;k}^{\star}(\al)^{\phi}$ is a smooth manifold with boundary 
of dimension $\dim_{g,l;k}(B)\!+\!1$ and~\eref{alRTreal_e} is a smooth map,

\item $\cZ_{\ga,\vp}^*(\al)$ is a smooth manifold with boundary of
dimension $\dim_{\ga,\vp}\!+\!1$ and~\eref{alredstev_e} is a smooth map
for all \hbox{$(\ga,\vp)\!\in\!\cA(\ov\ga)$}
satisfying $\ov\ga\!\in\!\cA_{g,l;k}$, \hbox{$\ga\!\in\!\cA_{g',l;k}^{\phi}(B')$},
$B'\!\in\!H_2(X;\Z)$, and  \hbox{$\om(B')\!\le\!\om(B)$}.

\end{enumerate}
\end{prp}

\begin{crl}\label{RTorient_crl}
Let $n\!\not\in\!2\Z$ and $(X,\om,\phi)$ be a compact  real symplectic $2n$-manifold 
endowed with a real orientation.
For all  $g,l\!\in\!\Z^{\ge0}$ with $g\!+\!l\!\ge\!2$ and $B\!\in\!H_2(X;\Z)$,
there exists a Baire subset \hbox{$\wh\cH_{g,l}^{\om,\phi}(X)\!\subset\!\cH_{g,l;0}^{\om,\phi}(X)$}
of second category such~that  $\fM_{g,l;0}^{\star}(X,B;J,\nu)^{\phi}$ is an oriented manifold of
dimension~\eref{RTreal_e2}.
\end{crl}

\begin{proof}
By Proposition~\ref{RTreg_prp}\ref{mainstr_it}, 
$\fM_{g,l;0}^{\star}(X,B;J,\nu)^{\phi}$ is a smooth manifold
of the expected dimension.
By \cite[Theorem~1.3]{RealGWsI}, a real orientation $(X,\om,\phi)$ determines 
an orientation on this space.
\end{proof}

\begin{crl}\label{RTorient2_crl}
Let $n$, $(X,\om,\phi)$, $g,l$, $B$, and $\wh\cH_{g,l}^{\om,\phi}(X)$ be in
Corollary~\ref{RTorient_crl}.
For all elements $(J,\nu)$ and $(J',\nu')$ of $\wh\cH_{g,l}^{\om,\phi}(X)$,
there exists a Baire subset of second category as in~\eref{whsPdfn_e}
such~that $\fM_{g,l;0}^{\star}(B;\al)^{\phi}$ is an oriented manifold with boundary
\BE{RTorient2_e}\prt\,\fM_{g,l;0}^{\star}(B;\al)^{\phi}
=\fM_{g,l;0}^{\star}(X,B;J',\nu')^{\phi}
-\fM_{g,l;0}^{\star}(X,B;J,\nu)^{\phi}\EE
for every $\al\!\in\!\wh\sP(J,\nu;J',\nu')$.
\end{crl}

\begin{proof}
By Proposition~\ref{RTreg2_prp}\ref{mainstr2_it}, 
$\fM_{g,l;0}^{\star}(B;\al)^{\phi}$ is a smooth manifold with boundary;
its boundary is as specified by~\eref{RTorient2_e}.
By the proof of \cite[Theorem~1.3]{RealGWsI}, a real orientation $(X,\om,\phi)$ determines 
an orientation on $\fM_{g,l;0}^{\star}(B;\al)^{\phi}$.
\end{proof}

\subsection{Proof of Theorem~\ref{RTreal_thm}}
\label{RTrealpf_subs}

\noindent
We deduce Proposition~\ref{RTdim_prp} below from Proposition~\ref{RTreg_prp}
primarily through dimension counting.
Proposition~\ref{RTdim2_prp} is obtained similarly from 
Proposition~\ref{RTreg2_prp}.
We~then combine these propositions with Corollaries~\ref{RTorient_crl} and~\ref{RTorient2_crl}
to establish Theorem~\ref{RTreal_thm}.

\begin{prp}\label{RTdim_prp}
Let $(X,\om,\phi)$ be a compact semi-positive real symplectic manifold.
For all 
\hbox{$g,l,k\!\in\!\Z^{\ge0}$} with \hbox{$2(g\!+\!l)\!+\!k\!\ge\!3$} and $B\!\in\!H_2(X;\Z)$,
there exists a Baire subset of second category as in~\eref{whcHdfn_e} such~that
\BE{RTcrl_e2}
\dim \big\{\st\!\times\!\ev\big\}
\big(\ov\fM_{g,l;k}(X,B;J,\nu)^{\phi}\!-\!\fM_{g,l;k}^{\star}(X,B;J,\nu)^{\phi}\big)
\le  \dim_{g,l;k}(B)\!-\!2\EE
for every $(J,\nu)\!\in\!\wh\cH_{g,l;k}^{\om,\phi}(X)$.
\end{prp}

\begin{prp}\label{RTdim2_prp}
Let $(X,\om,\phi)$, $g,l,k$, $B$, and
$\cH_{g,l;k}^{\om,\phi}(X)$ be as in Proposition~\ref{RTdim_prp}.
For all elements $(J,\nu)$ and $(J',\nu')$ of $\wh\cH_{g,l;k}^{\om,\phi}(X)$,
there exists a Baire subset of second category as in~\eref{whsPdfn_e}
such~that
$$\dim \big\{\st\!\times\!\ev\big\}
\big(\ov\fM_{g,l;k}(B;\al)^{\phi}\!-\!\fM_{g,l;k}^{\star}(B;\al)^{\phi}\big)
\le \dim_{g,l;k}(B)\!-\!1$$
for every $\al\!\in\!\wh\sP(J,\nu;J',\nu')$.
\end{prp}

\noindent
As with the definitions of pseudocycle in \cite{pseudo,MS},
\eref{RTcrl_e2} means~that 
$$\big\{\st\!\times\!\ev\big\}\big(\ov\fM_{g,l;k}(X,B;J,\nu)^{\phi}\!-\!
\fM_{g,l;k}^{\star}(X,B;J,\nu)\big)^{\phi} 
\subset \R\ov\cM_{g,l;k} \!\times\! \big(X^l\!\times\!(X^{\phi})^k\big)$$
is contained in the image of a smooth map from a manifold of
dimension equal to the right-hand side of this inequality.
Thus, the restriction 
$$\st\!\times\!\ev\!:\fM_{g,l;k}^{\star}(X,B;J,\nu)^{\phi} \lra 
\R\ov\cM_{g,l;k} \times \big(X^l\!\times\!(X^{\phi})^k\big)$$
is a pseudocycle whenever $(J,\nu)\!\in\!\wh\cH_{g,l;k}^{\om,\phi}(X)$ {\it and}
the domain of this map is an oriented manifold.\\

\noindent
Let $B\!\in\!H_2(X;\Z)\!-\!\{0\}$. 
For $J\!\in\!\cJ_{\om}^{\phi}$ and $\si\!=\!\tau,\eta$, we denote~by 
$$\fM_0^*(X,B;J)\subset\ov\fM_0(X,B;J)  \qquad\hbox{and}\qquad
\fM_0^*(X,B;J)^{\si,\phi}\subset\ov\fM_0(X,B;J)^{\si,\phi}$$
the moduli spaces of equivalences classes of simple degree~$B$ $J$-holomorphic maps
from~$\P^1$ to~$X$ and of simple real degree~$B$ $J$-holomorphic maps
from~$(\P^1,\si)$ to~$(X,\phi)$, respectively.
Let
$$\fM_0^*(X,B;J)^{\phi} \equiv
\fM_0^*(X,B;J)^{\phi,\tau}\!\sqcup\!\fM_0^*(X,B;J)^{\phi,\eta}
\subset\ov\fM_0(X,B;J)^{\phi}$$
be the space of all simple real degree~$B$ $J$-holomorphic maps from~$\P^1$ to~$(X,\phi)$.
We note~that 
\begin{gather*}
\fM_0^*(X,B;J)=\eset~~~\hbox{if}~~\om(B)\!\le\!0,\\
\fM_0^*(X,B;J)^{\phi}=\eset~~\hbox{if}~B\!\not\in\!H_2^{\R S}(X;\Z)^{\phi},\qquad
\fM_0^*(X,B;J)^{\phi,\tau}=\eset~~\hbox{if}~B\!\not\in\!H_2^{\tau}(X;\Z)^{\phi}.
\end{gather*}
The natural morphism
$$\ov\fM_0(X,B;J)^{\phi} \lra \ov\fM_0(X,B;J) $$
restricts to an embedding 
$$\fM_0^*(X,B;J)^{\phi} \lra \fM_0^*(X,B;J);$$
we view $\fM_0^*(X,B;J)^{\phi}$ as a subspace of $\fM_0^*(X,B;J)$.\\

\noindent
For $J,J'\!\in\!\cJ_{\om}^{\phi}$, define
$$\sP(J;J') =
\big\{\al\!:[0,1]\!\lra\!\cJ_{\om}^{\phi}\!:\al(0)\!=\!J,~\al(1)\!=\!J'\big\}$$
to be the space of paths from~$J$ and~$J'$. 
For any such path~$\al$ and each \hbox{$B\!\in\!H_2(X;\Z)\!-\!\{0\}$}, let 
\begin{equation*}\begin{split}
\fM_0^*(B;\al)&=\big\{\big(t,[\u]\big)\!:t\!\in\!\bI,\,
[\u]\!\in\!\fM_0^*\big(X,B;\al(t)\big)\big\},\\
\fM_0^*(B;\al)^{\phi}&=\big\{\big(t,[\u]\big)\!:t\!\in\!\bI,\,
[\u]\!\in\!\fM_0^*\big(X,B;\al(t)\big)^{\phi}\big\},\\
\fM_0^*(B;\al)^{\phi,\tau}&=\big\{\big(t,[\u]\big)\!:t\!\in\!\bI,\,
[\u]\!\in\!\fM_0^*\big(X,B;\al(t)\big)^{\phi,\tau}\big\}.
\end{split}\end{equation*}

\begin{lmm}\label{g0empty_lmm}
Let $(X,\om,\phi)$ be a compact real symplectic $2n$-manifold.
For every \hbox{$B\!\in\!H_2(X;\Z)$},
there exists a Baire subset \hbox{$\wh\cJ_{\om}^{\phi}\!\subset\!\cJ_{\om}^{\phi}$}
of second category  with the following property.
If \hbox{$J,J'\!\in\!\wh\cJ_{\om}^{\phi}$}, there exists a Baire subset 
\BE{whsPJdfn_e}\wh\sP(J;J')\!\subset\!\sP(J;J')\EE 
of second category such~that for every $\al\!\in\!\wh\sP(J;J')$
\begin{alignat*}{2}
\fM_0^*(B';\al)&=\fM_0^*(B';\al)^{\phi}  &\qquad
&\hbox{if}~~B'\!\in\!H_2(X;\Z),~0\!<\!\om(B')\!\le\!\om(B),~
\lr{c_1(TX),B'}<3\!-\!n,\\
\fM_0^*(B';\al)^{\phi}&=\eset  &\qquad
&\hbox{if}~~B'\!\in\!H_2(X;\Z),~0\!<\!\om(B')\!\le\!\om(B),~\lr{c_1(TX),B'}<2\!-\!n.
\end{alignat*}
\end{lmm}

\begin{proof}
By Section~\ref{trans_sec}, there exists a Baire subset 
\hbox{$\wh\cJ_{\om}^{\phi}\!\subset\!\cJ_{\om}^{\phi}$} 
of second category  with the following property.
If $J,J'\!\in\!\wh\cJ_{\om}^{\phi}$, there exists a Baire subset of second category
as in~\eref{whsPJdfn_e} such~that for all 
$$\al\in\wh\sP(J;J') \qquad\hbox{and}\qquad
B'\!\in\!H_2(X;\Z) ~~\hbox{with}~~0\!<\!\om(B')\!\le\!\om(B)$$
the moduli spaces $\fM_0^*(B';\al)\!-\!\fM_0^*(B';\al)^{\phi}$
and $\fM_0^*(B';\al)^{\phi}$ are smooth manifolds of dimensions
\begin{equation*}\begin{split}
\dim_{\R}\!\big(\fM_0^*(B';\al)\!-\!\fM_0^*(B';\al)^{\phi}\big)
&=2\big(\lr{c_1(TX),B'}\!+\!n\!-\!3\big)+1, \\
\dim_{\R}\fM_0^*(B';\al)^{\phi}&=\lr{c_1(TX),B'}\!+\!n\!-\!3+1.
\end{split}\end{equation*}
This implies the claim.
\end{proof}

\begin{crl}\label{g0empty_crl}
Let $(X,\om,\phi)$ be a compact semi-positive real symplectic $2n$-manifold.
For every \hbox{$B\!\in\!H_2(X;\Z)$},
there exists a Baire subset \hbox{$\wh\cJ_{\om}^{\phi}\!\subset\!\cJ_{\om}^{\phi}$}
of second category  with the following property.
If \hbox{$J,J'\!\in\!\wh\cJ_{\om}^{\phi}$}, there exists a Baire subset of second category 
as in~\eref{whsPJdfn_e} such~that for every $\al\!\in\!\wh\sP(J;J')$
\begin{alignat}{2}
\label{g0emptyC_e}
\lr{c_1(TX),B'}&\ge0,3\!-\!n &\quad 
&\hbox{if}~~B'\!\in\!H_2(X;\Z),~0\!<\!\om(B')\!\le\!\om(B),~
\fM_0^*(B';\al)\!\neq\!\eset,\\
\label{g0emptyR_e}
\lr{c_1(TX),B'}&\ge\de_{n2} &\quad 
&\hbox{if}~~B'\!\in\!H_2(X;\Z),~0\!<\!\om(B')\!\le\!\om(B),~
\fM_0^*(B';\al)^{\phi}\!\neq\!\eset,\\
\label{g0emptyRtau_e}
\lr{c_1(TX),B'}&\ge1 &\quad 
&\hbox{if}~~B'\!\in\!H_2(X;\Z),~0\!<\!\om(B')\!\le\!\om(B),~
\fM_0^*(B';\al)^{\phi,\tau}\!\neq\!\eset.
\end{alignat}
\end{crl}

\begin{proof}
The first inequality in~\eref{g0emptyC_e}, \eref{g0emptyR_e}, and \eref{g0emptyRtau_e} 
follow immediately from Lemma~\ref{g0empty_lmm} and Definition~\ref{SSP_dfn}.
If 
$$\fM_0^*(B';\al)\neq\eset  \qquad\hbox{and}\qquad  \lr{c_1(TX),B'}<3\!-\!n,$$
then the first statement of Lemma~\ref{g0empty_lmm} gives
$$\fM_0^*(B';\al)^{\phi}= \fM_0^*(B';\al)\neq\eset\,.$$
By the second statements of Lemma~\ref{g0empty_lmm} and of the present corollary, 
this implies that 
$$2\!-\!n=\lr{c_1(TX),B'}\ge\de_{n2}.$$
Thus, $n\!=\!1$ and $X\!=\!\P^1$ ($X$ can be assumed to be connected). 
However, $\lr{c_1(TX),B'}$ is even for every homology class~$B'$ on $X\!=\!\P^1$.
\end{proof}

\noindent 
Suppose \hbox{$g,l,k\!\in\!\Z^{\ge0}$} with \hbox{$2(g\!+\!l)\!+\!k\!\ge\!3$}, 
$B\!\in\!H_2(X;\Z)$, and $\ga\!\in\!\cA_{g,l;k}^{\phi}(B)$.
For $(\ga',\vp)\!\in\!\cA'(\ga)$, let
$$\coeff{\ga',\vp}=|\ga'|\!-\!\big|\sE_{\ga'}\big(\ale(\ga',\ov\ga)_0\big)\big|
\!-\!\big|\pi_0\big(\ga',\ale(\ga',\ov\ga)_0\big)\big|\,.$$
Since $\ga'$ is connected,
\BE{g0empty_e23}\coeff{\ga',\vp} \ge \big|\ale(\ga',\ov\ga)_{\bu}\big|\,.\EE
By the second condition in~\eref{redgcond_e},
\BE{g0empty_e19} \coeff{\ga',\vp}+\ell\big(\ga',\ale(\ga',\ov\ga)_0\big)
=g'\!+\!\big|\ale(\ga',\ov\ga)_{\bu}\big|+\big|\ov\Ver\big|\!-\!1
\ge \big|\ale(\ga)_{\bu}\big|.\EE
We denote $\cA'(\ga)\!\subset\!\cA(\ga)$ the subset of pairs $(\ga',\vp)$
such~that 
\BE{RTcrl_e9}|\vr|_{v'}\equiv\sum_{v\in\ka^{-1}(v')}\!\!\!\!\!\!\!\vr(v)-1~\ge1\EE
for some $v'\!\in\!\ale(\ga',\ov\ga)_{\bu}$,
where $\ka$ and $\vr$ are the maps~\eref{redmapcond_e} corresponding to~$(\ga',\vp)$.

\begin{crl}\label{g0empty_crl2}
Let $(X,\om,\phi)$ be a compact semi-positive real symplectic $2n$-manifold.
For all \hbox{$g,l,k\!\in\!\Z^{\ge0}$} with \hbox{$2(g\!+\!l)\!+\!k\!\ge\!3$} 
and $B\!\in\!H_2(X;\Z)$, there exists a Baire subset of second category as in~\eref{whcHdfn_e}
with the following property.
For all elements $(J,\nu)$ and $(J',\nu')$ of this subset,
there exists a Baire subset of second category as in~\eref{whsPdfn_e}
such~that 
$$\blr{c_1(TX),B\!-\!B'}+(n\!-\!3)\big(g'\!-\!g\!-\!\ell\big(\ga',\ale(\ga',\ov\ga)_0\big)\big)\\
+\coeff{\ga',\vp}\ge2$$
for every $\al\!\in\!\wh\sP(J,\nu;J',\nu')$,
$\ga\!\in\!\cA_{g,l;k}^{\phi}(B)$, and 
$(\ga',\vp)\!\in\!\cA'(\ga)$ with $\ga'\!\in\!\cA_{g',l;k}^{\phi}(B')$
and  \hbox{$\cZ_{\ga';\vp}^*(\al)\!\neq\!\eset$}.
\end{crl}

\begin{proof}
Take the Baire subsets to be the preimages of the subsets 
$$\wh\cJ_{\om}^{\phi}\subset\cJ_{\om}^{\phi}
\qquad\hbox{and}\qquad \wh\sP(J;J')\!\subset\!\sP(J;J')$$ 
of Corollary~\ref{g0empty_crl} under the natural projections
\BE{g0empty_e17}\cH_{g,l;k}^{\om,\phi}(X)\lra\cJ_{\om}^{\phi} \qquad\hbox{and}\qquad
\sP(J,\nu;J',\nu')\lra \sP(J;J')\,.\EE
Suppose  $\al\!\in\!\wh\sP(J,\nu;J',\nu')$,
$\ga\!\in\!\cA_{g,l;k}^{\phi}(B)$, and
$(\ga',\vp)\!\in\!\cA'(\ga)$ with~$\ga'$ as in~\eref{gaprdfn_e}
and  \hbox{$\fM_{\ga';\vp}^*(\al)\!\neq\!\eset$}.
Let $\al_{\cJ}$ be the image of~$\al$ under the second projection in~\eref{g0empty_e17} and
$\ka$ and $\vr$ be the maps~\eref{redmapcond_e} corresponding to~$(\ga',\vp)$.\\

\noindent
We first note that
\begin{gather}\label{g0empty_e21a}
\sum_{v'\in\ale(\ga',\ov\ga)_{\bu}}\!\!\!|\vr|_{v'}
\blr{c_1(TX),\fd'(v')}+\coeff{\ga',\vp}\ge2,\\
\label{g0empty_e21b}
\sum_{v'\in\ale(\ga',\ov\ga)_{\bu}}\!
\sum_{v\in\ka^{-1}(v')}\!\!\!\!\!\!\!\big(\vr(v)\!-\!1\big)\blr{c_1(TX),\fd'(v')}
+\coeff{\ga',\vp}\!+\!\ell\big(\ga',\ale(\ga',\ov\ga)_0\big)\ge2.
\end{gather}
By the first inequality in~\eref{g0emptyC_e},
\BE{g0empty_e22}\blr{c_1(TX),\fd'(v')}\ge0 \qquad\forall~v'\!\in\!\ale(\ga',\ov\ga)_{\bu}\,.\EE
If the inequality in~\eref{g0empty_e23} is an equality and 
$|\ale(\ga',\ov\ga)_{\bu}|\!=\!1$, then 
there is a unique flag $f\!\in\!\Fl'$ such~that 
\begin{enumerate}[label=$\bu$,leftmargin=*]

\item $v'\!=\!\ve'(f)$ is the unique element of $\ale(\ga',\ov\ga)_{\bu}$ and

\item the removal of the edge $e\!=\!\{f,\vt'(f)\}$ separates~$v'$
from all vertices \hbox{$\ov\Ver\!\subset\!\Ver'$}.

\end{enumerate}
The unique flag~$f$ is then preserved by the involution~$\si'$ on~$\ga'$.
It thus corresponds to a real point of the irreducible component~$\Si_{v'}$
of the domain of any element of $\fM_{\ga';\vp}^*(\al)$.
It follows that the moduli space $\fM_0^*(\fd'(v');\al_{\cJ})^{\phi,\tau}$ is not empty.
The inequality~\eref{g0empty_e21a} thus follows from~\eref{g0emptyRtau_e},
\eref{g0empty_e23}, and \eref{g0empty_e22} in all cases.  
If $\vr(v)\!\ge\!2$ for some $v\!\in\!\ale(\ga)_{\bu}$,  
\eref{g0empty_e21b} also follows from~\eref{g0emptyRtau_e},
\eref{g0empty_e23}, and \eref{g0empty_e22}.
Otherwise, \hbox{$|\ale(\ga)_{\bu}|\!\ge\!2$} and \eref{g0empty_e21b} follows 
from~\eref{g0empty_e19} and~\eref{g0empty_e22}.\\

\noindent
By the last statement in~\eref{degdiff_e},
\BE{g0empty_e27}\begin{split}
&\blr{c_1(TX),B\!-\!B'}=\sum_{v'\in\ale(\ga',\ov\ga)_{\bu}}\!\!\!\!\!\!\!\!
|\vr|_{v'}\blr{c_1(TX),\fd'(v')}\\
&\quad=\sum_{v'\in\ale(\ga',\ov\ga)_{\bu}}\!
\sum_{v\in\ka^{-1}(v')}\!\!\!\!\!\!\!\big(\vr(v)\!-\!1\big)\blr{c_1(TX),\fd'(v')}
-(n\!-\!3)\big(|\ale(\ga)_{\bu}|\!-\!|\ale(\ga',\ov\ga)_{\bu}|\big)\\
&\qquad
+\sum_{v'\in\ale(\ga',\ov\ga)_{\bu}}\!\!\!\!\!\!\!\!\big(|\ka^{-1}(v')|\!-\!1\big)
\big(\blr{c_1(TX),\fd'(v')}\!+\!n\!-\!3\big).
\end{split}\EE
For $n\!\ge\!3$, the claim follows from the first equality above, 
the first statement in~\eref{degdiff_e}, and~\eref{g0empty_e21a}.
For $n\!<\!3$, the second equality in~\eref{g0empty_e27}, the second statement
in~\eref{redgcond_e}, and the second inequality in~\eref{g0emptyC_e} give 
\begin{equation*}\begin{split}
&\blr{c_1(TX),B\!-\!B'}+(n\!-\!3)\big(g'\!-\!g\!-\!\ell\big(\ga',\ale(\ga',\ov\ga)_0\big)\big)\\
&\hspace{.5in}\ge \sum_{v'\in\ale(\ga',\ov\ga)_{\bu}}\!
\sum_{v\in\ka^{-1}(v')}\!\!\!\!\!\!\!\big(\vr(v)\!-\!1\big)\blr{c_1(TX),\fd'(v')}
+\ell\big(\ga',\ale(\ga',\ov\ga)_0\big).
\end{split}\end{equation*}
The claim now follows from~\eref{g0empty_e21b}.
\end{proof}

\begin{proof}[{\bf{\emph{Proof of Proposition~\ref{RTdim_prp}}}}]
For each $\ga\!\in\!\cA_{g,l;k}^{\phi}(B)$, let
$$\fM_{\ga}^*(J,\nu)\subset\fM_{\ga}(J,\nu)
\quad\hbox{and}\quad
\fM_{\ga}^{mc}(J,\nu)\equiv \fM_{\ga}(J,\nu)\!-\!\fM_{\ga}^*(J,\nu)
\subset \fM_{\ga}(J,\nu)$$
denote the subspace of simple maps in the sense of Definition~\ref{Jnumap_dfn2}
and the subspace of \sf{multiply covered maps}, respectively.
In particular,
\begin{equation*}\begin{split}
\ov\fM_{g,l;k}(X,B;J,\nu)^{\phi}\!-\!\fM_{g,l;k}^{\star}(X,B;J,\nu)^{\phi}
=\bigsqcup_{\begin{subarray}{c}\ga\in\cA_{g,l;k}^{\phi}(B)\\ |\ga|\ge2\end{subarray}}
\hspace{-.22in}\fM_{\ga}^*(J,\nu)~~\sqcup\!\!\!
\bigsqcup_{\begin{subarray}{c}\ga\in\cA_{g,l;k}^{\phi}(B)\\ \ale(\ga)\neq\eset\end{subarray}}
\!\!\!\!\!\!\!\!\!\fM_{\ga}^{mc}(J,\nu)\,.
\end{split}\end{equation*}
Since the map~\eref{evst_e} is continuous and 
$$\big|\big\{\ga\!\in\!\cA_{g,l;k}^{\phi}(B)\!:
\fM_{\ga}(J,\nu)\!\neq\!\eset\big\}\big|<\i\,,$$
it is thus sufficient to show~that 
\begin{alignat}{2}
\label{RTcrl_e5a}
\dim_{\R}\fM_{\ga}^*(J,\nu)&\le\dim_{g,l;k}(B)\!-\!2
&\qquad&\forall~\ga\!\in\!\cA_{g,l;k}^{\phi}(B),~|\ga|\!\ge\!2,\\
\label{RTcrl_e5b}
\dim \big\{\st\!\times\!\ev\big\}
\big(\fM_{\ga}^{mc}(J,\nu)\big)&\le \dim_{g,l;k}(B)\!-\!2
&\qquad&\forall~\ga\!\in\!\cA_{g,l;k}^{\phi}(B),~\ale(\ga)\!\neq\!\eset.
\end{alignat}

\vspace{.2in}

\noindent
Let $\ga\!\in\!\cA_{g,l;k}^{\phi}(B)$, $\ov\ga\!\in\!\cA_{g,l;k}$ be the contraction
of~$\ga$, and $\vp_{\ga}\!\equiv\!\vp$ be as in~\eref{gavpovga_e}.
Then,
$$\fM_{\ga}^*(J,\nu)=\fM_{\ga,\vp_{\ga}}^*(J,\nu)\big/\Aut(\ga)
=\big(\R\cM_{\ga;\ale(\ga,\ov\ga)}\!\times\!\cZ_{\ga,\vp_{\ga}}'^*(J,\nu)\big)\big/\Aut(\ga)\,.$$
For a generic choice of $(J,\nu)$, the images of the irreducible components~$\Si_v$ 
of the domain~$\Si$ of any element~$\u$ of $\cZ_{\ga,\vp_{\ga}}'^*(J,\nu)$ are distinct.
Since $\ale(\ga,\ov\ga)\!=\!\ale(\ga)$ contains no loops, 
the $\Aut(\ga)$-action on $\fM_{\ga}^*(J,\nu)$ is thus free and 
\hbox{$\ell(\ga,\ale(\ga,\ov\ga)_0)\!=\!0$}.
Thus, \eref{RTcrl_e5a} follows from Proposition~\ref{RTreg_prp}\ref{cZstr_it}.\\

\noindent
Suppose $\ga\!\in\!\cA_{g,l;k}^{\phi}(B)$ and $\ale(\ga)\!\neq\!\eset$.
By Lemma~\ref{mapred_lmm},
$$\big\{\st\!\times\!\ev\big\}\big(\fM_{\ga}^{mc}(J,\nu)\big)
\subset \bigsqcup_{(\ga',\vp)\in\cA'(\ga)} \hspace{-.25in}
\big\{\st_{\ga';\vp}\!\times\!\ev\big\}\big(\cZ_{\ga';\vp}^*(J,\nu)\big)\,.$$
For the purposes of establishing~\eref{RTcrl_e5b}, it thus suffices to show~that
\BE{RTcrl_e17}
\dim_{\R}\cZ_{\ga';\vp}^*(J,\nu)\le \dim_{g,l;k}(B)\!-\!2
\qquad\forall\,(\ga',\vp)\!\in\!\cA'(\ga)\,.\EE

\vspace{.2in}

\noindent
Let $(\ga',\vp)\!\in\!\cA'(\ga)$ with $\ga'\!\in\!\cA_{g',l;k}^{\phi}(B')$
and $|\vr|_{v'}$ for each $v'\!\in\!\ale(\ga',\ov\ga)_{\bu}$
be as in~\eref{RTcrl_e9}.
By Proposition~\ref{RTreg_prp}\ref{cZstr_it} and~\eref{RTprp_e}, 
\begin{equation*}\begin{split}
\dim_{\R}\cZ_{\ga';\vp}^*(J,\nu)=
& \dim_{g,l;k}(B) -
\big(\blr{c_1(TX),B\!-\!B'}+(n\!-\!3)(g'\!-\!g)\big)\\
&-|\ga'|+n\,\ell\big(\ga',\ale(\ga',\ov\ga)_0\big)-
\big(\!\dim\R\cM_{\ga';\ale(\ga',\ov\ga)}\!+\!
\big|\sF_{\ga'}^{\dag}\!(\ale(\ga',\ov\ga)_0)\big|\big)\,.
\end{split}\end{equation*}
Along with~\eref{cMgasV_e} with $\sV\!=\!\ale(\ga',\ov\ga)$, this implies that 
\begin{equation*}\begin{split}
\dim_{\R}\cZ_{\ga';\vp}^*(J,\nu)\le
 \dim_{g,l;k}(B) -
\big(\blr{c_1(TX),B\!-\!B'}&+(n\!-\!3)(g'\!-\!g)\big)\\
&+(n\!-\!3)\ell\big(\ga',\ale(\ga',\ov\ga)_0\big)
-\coeff{\ga',\vp} \,.
\end{split}\end{equation*}
The inequality~\eref{RTcrl_e17} now follows from Corollary~\ref{g0empty_crl2}.
\end{proof}

\noindent
Corollary~\ref{RTorient_crl} and Proposition~\ref{RTdim_prp} establish 
Theorem~\ref{RTreal_thm}\ref{RTexist_it}.
Corollary~\ref{RTorient2_crl} and Proposition~\ref{RTdim2_prp} establish 
the first claim of Theorem~\ref{RTreal_thm}\ref{RTindep_it}.\\

\noindent
It remains to establish the second claim of Theorem~\ref{RTreal_thm}\ref{RTindep_it}.
Let 
$$p\!:\wt\cM_{g,l}\lra\ov\cM_{g,l} \qquad\hbox{and}\qquad
p'\!:\wt\cM_{g,l}'\lra\ov\cM_{g,l}$$
be regular covers. Then the cover
$$\wh{p}\!:\wh\cM_{g,l}\equiv\wt\cM_{g,l}\times_{\ov\cM_{g,l}}\!\!\wt\cM_{g,l}'
\lra\ov\cM_{g,l}$$
is also regular and is the composition of $p$ and~$p'$ with covers
$$q\!:\wh\cM_{g,l}\lra\wt\cM_{g,l} \qquad\hbox{and}\qquad 
q'\!:\wh\cM_{g,l}\lra\wt\cM_{g,l}'\,,$$
respectively.
It is thus sufficient to compare the class~\eref{RTclass_e} with its $\wh{p}$-analogue.\\

\noindent
We denote by $\wh\Ga^{0,1}_{g,l;k}(X;J)^{\phi}$ the analogue of the space~\eref{Ga01dfn_e}
for~$\wh{p}$.
The projection $q$ lifts to a cover
$$\wt{q}\!:\R\wh\cU_{g,l;k} \lra \R\wt\cU_{g,l;k}$$
between the real universal curves~\eref{RwtcU_e} determined by~$\wh{p}$ and~$p$.
This lift commutes with the involutions and is biholomorphic on each fiber
of~\eref{RwtcU_e}.
Thus,
$$\wh\nu\equiv \big\{\wt{q}\!\times\!\id_X\big\}^*\nu\in \wh\Ga^{0,1}_{g,l;k}(X;J)^{\phi}
\qquad\forall\,\nu\!\in\!\Ga^{0,1}_{g,l;k}(X;J)^{\phi}\,.$$
The composition of $u_{\cM}$ in~\eref{Jnumap_e} with $\wt{q}$ determines a projection
\BE{wtqfM_e}\wt{q}\!:\fM_{g,l;0}^{\star}(X,B;J,\wh\nu)^{\phi}\lra
\fM_{g,l;0}^{\star}(X,B;J,\nu)^{\phi}\EE
of degree equal to the degree of~$q$.\\

\noindent
If $(J,\nu)\!\in\!\wh\cH_{g,l}^{\om,\phi}(X)$, then 
$$\st\!\times\!\ev\!:\fM_{g,l;0}^{\star}(X,B;J,\wh\nu)^{\phi} \lra 
\R\ov\cM_{g,l}\!\times\!X^l$$
is a pseudocycle representing the class~\eref{RTclass_e} determined by~$\wh{p}$.
It equals to the composition of the pseudocycle 
$$\st\!\times\!\ev\!:\fM_{g,l;0}^{\star}(X,B;J,\nu)^{\phi} \lra 
\R\ov\cM_{g,l}\!\times\!X^l$$
with~\eref{wtqfM_e}.
Thus, 
\begin{equation*}\begin{split}
&\Big[\st\!\times\!\ev\!:\fM_{g,l;0}^{\star}(X,B;J,\wh\nu)^{\phi} \lra 
\R\ov\cM_{g,l}\!\times\!X^l\Big]\\
&\hspace{1in}
=\big(\deg q\big)
\Big[\st\!\times\!\ev\!:\fM_{g,l;0}^{\star}(X,B;J,\nu)^{\phi} \lra 
\R\ov\cM_{g,l}\!\times\!X^l\Big].
\end{split}\end{equation*}
Since the degree of $\wh{p}$ is the product of the degrees of $p$ and~$q$,
this establishes the second claim of Theorem~\ref{RTreal_thm}\ref{RTindep_it}.

\section{Transversality}
\label{trans_sec}

\noindent
Proposition~\ref{RTreg_prp}\ref{cZstr_it}
comes down to the smoothness of the second subspace in~\eref{fMga_e2} 
for a generic pair~$(J,\nu)$.
It   holds for fundamentally the same reasons as
\cite[Theorem~6.2.6]{MS} and \cite[Theorem~3.16]{RT2}.
However, we present these reasons more systematically.
In Section~\ref{DefObs_subs}, we describe a deformation-obstruction setup
for each of the four types of irreducible components~$\Si_v$ of the domain
of maps in this subspace for a typical element \hbox{$(\ga,\vp)\!\in\!\cA(\ov\ga)$}.
By Lemmas~\ref{RTtransC_lmm} and~\ref{RTtransR_lmm}, the deformations of $(J,\nu)$ in~\eref{cHdfn_e} 
supported in an open set~$W$ intersecting the image of~$\Si_v$
cover the obstruction spaces in all four cases.
In Section~\ref{RTregpf_subs}, we show that this lemma implies
the smoothness of the universal moduli space~\eref{UfMdfn_e}.
As usual, the latter in turn implies that 
the second subspace in~\eref{fMga_e2} is smooth   
for a generic element~$(J,\nu)$ of $\cH_{g,l;k}^{\om,\phi}(X)$. 
As explained in the next paragraph, Proposition~\ref{RTreg_prp}\ref{cZstr_it}
implies Proposition~\ref{RTreg_prp}\ref{mainstr_it}.
The proof of Proposition~\ref{RTreg2_prp} is similar.\\

\noindent
Let $\ga_0\!\in\!\cA_{g,l;k}^{\phi}(B)$ denote the unique element with $|\ga_0|\!=\!0$.
For each $\ga\!\in\!\cA_{g,l;k}^{\phi}(B)$, let 
$$\fM_{\ga}^*(J,\nu)\subset\fM_{\ga}(J,\nu)$$
denote the subspace of simple maps in the sense of Definition~\ref{Jnumap_dfn2}.
The subspace in~\eref{fMstar_e} 
is stratified by the subspaces $\fM_{\ga}^*(J,\nu)$ with 
$\ga\!\in\!\cA_{g,l;k}^{\phi}(B)$ such that $|\ga|\!\le\!1$.
For a generic pair~$(J,\nu)$, each of the latter subspaces is cut transversely by
the $(\dbar_J\!-\!\nu)$-operator.
Any standard gluing construction, such as in \cite[Section~3]{LT} or 
\cite[Chapter~10]{MS}, restricts to the real setting and provides a continuous~map
$$\Phi_{\ga}\!: \cN_{\ga}'\lra 
\fM_{\ga}^*(J,\nu)\!\cup\!\fM_{\ga_0}(J,\nu)
\subset \fM_{g,l;k}^{\star}(X,B;J,\nu)^{\phi}\subset\ov\fM_{g,l;k}(X,B;J,\nu)^{\phi}$$
from a neighborhood $\cN_{\ga}'$ of the zero section $\fM_{\ga}^*(J,\nu)$
in the (real) rank~$|\ga|$ bundle of smoothing parameters.
It restricts to the identity on $\fM_{\ga}^*(J,\nu)$ and 
to a diffeomorphism from its complement to
an open subspace of $\fM_{\ga_0}(J,\nu)$.
Thus, the maps~$\Phi_{\ga}$ with $\ga\!\in\!\cA_{g,l;k}^{\phi}(B)$ such that $|\ga|\!=\!1$ extend 
the canonical smooth structure on $\fM_{\ga_0}(J,\nu)$ to
 a smooth structure on the entire subspace in~\eref{fMstar_e}.\\

\noindent
For the remainder of this paper, fix $(X,\om,\phi)$,  $B$, and $g,l,k$  
as in Proposition~\ref{RTreg_prp}, $\ov\ga\!\in\!\cA_{g,l;k}$ as in~\eref{stgadfn_e},  
\hbox{$(\ga,\vp)\!\in\!\cA(\ov\ga)$} with 
\hbox{$\ga\!\in\!\cA_{g',l;k}^{\phi}(B')$} as in~\eref{gadfn_e},
and $\rho\!\in\!\cA(\ga)$.
With $\nV_{\R}(\ga)$ and $\nV_{\C}(\ga)$ as defined in~\eref{nVRCdfn_e}, 
let 
\begin{alignat*}{3}
\ale_{\R}(\ga,\ov\ga)_{\bu}&=
\ale(\ga,\ov\ga)_{\bu}\!\cap\!\nV_{\R}(\ga), &\quad
\ale_{\R}(\ga,\ov\ga)_0&=
\ale(\ga,\ov\ga)_0\!\cap\!\nV_{\R}(\ga), &\quad
\ale_{\R}(\ga,\ov\ga)_0^c&=
\ale(\ga,\ov\ga)_0^c\!\cap\!\nV_{\R}(\ga),\\
\ale_{\C}(\ga,\ov\ga)_{\bu}&=
\ale(\ga,\ov\ga)_{\bu}\!\cap\!\nV_{\C}(\ga),
&\quad \ale_{\C}(\ga,\ov\ga)_0&=
\ale(\ga,\ov\ga)_0\!\cap\!\nV_{\C}(\ga),
&\quad \ale_{\C}(\ga,\ov\ga)_0^c&=
\ale(\ga,\ov\ga)_0^c\!\cap\!\nV_{\C}(\ga).
\end{alignat*}
For each $v\!\in\!\Ver$, let
$$S_{v;\C}(\ga),S_{v;\R}(\ga)\subset\ve^{-1}(v)\subset S_{l;k}\!\sqcup\!\Fl$$
be as in~\eref{Svdfn_e} with $\ov\ga$ replaced by~$\ga$.

\subsection{Spaces of deformations and obstructions}
\label{DefObs_subs}

\noindent
Let $\Si$ be a nodal surface. Its irreducible components~$\Si_v$, the nodes~$z_e$,
and the preimages~$z_f$ of the nodes in the normalization~$\wt\Si$ of~$\Si$ are
indexed by the sets~$\Ver$ of vertices, $\E(\ga)$ of edges, and $\Fl$ of flags, 
respectively, of the dual graph~$\ga$.
We call a continuous map $u\!:\Si\!\lra\!X$ \sf{smooth} if
the restriction~$u_v$ of~$u$ to each irreducible component~$\Si_v$ is smooth.\\

\noindent
If $u_v\!:\Si_v\!\lra\!X$ is a smooth map, let 
$$\Ga(u_v)=\Ga(\Si_v;u_v^*TX).$$
For a complex structure~$\fJ$ on~$\Si_v$
and an almost complex structure~$J$ on~$X$, define
$$\Ga^{0,1}_{J,\fJ}(u_v)=\Ga\big(\Si_v;(T\Si_v,-\fJ)^*\!\otimes_{\C}\!u_v^*(TX,J)\big).$$
For a real map~$u_v$ from~$(\Si_v,\si)$ to~$(X,\phi)$, define
$$\Ga(u_v)^{\phi,\si}
=\big\{\xi\!\in\!\Ga(u_v)\!:\,\tnd\phi\!\circ\!\xi\!=\!\xi\!\circ\!\si\big\}.$$
If in addition $\fJ$ is a complex structure on~$\Si_v$ reversed by~$\si$ and
$J$ is an almost complex structure on~$X$ reversed by~$\phi$, let
$$\Ga^{0,1}_{J,\fJ}(u_v)^{\phi,\si}
=\big\{\eta\!\in\!\Ga^{0,1}_{J,\fJ}(u_v)\!:\,
\tnd\phi\!\circ\!\eta\!=\!\eta\!\circ\!\tnd\si\big\}\,.$$

\vspace{.2in}

\noindent
Denote by $\si_v$ for $v\!\in\!\nV_{\R}(\ga)$ and  $v\!\in\!\nV_{\C}(\ga)$
the involutions on~\eref{pigarhovR_e} and~\eref{pigarhovC_e}, respectively.  
Let 
$$\cT_v\equiv\ker\tnd\pi_{\ga;\rho;v}\lra \R\wt\cU_{\ga;\rho;v}, ~~v\!\in\!\nV_{\R}(\ga),
\qquad
\cT_v^{\bu}\equiv\ker\tnd\pi_{\ga;v}^{\bu}\lra\wt\cU_{\ga;v}^{\bu}, ~~v\!\in\!\nV_{\C}(\ga),$$
be the vertical tangent line bundles with complex structures~$\fJ_v$ and~$\fJ_v^{\bu}$,
respectively.
For $J\!\in\!\cJ_{\om}^{\phi}$, define 
\begin{alignat*}{2}
\Ga^{0,1}_v(X;J)^{\phi}
&=\big\{\nu\!\in\!\Ga\big(\R\wt\cU_v\!\times\!X;
\pi_1^*(\cT_v,-\fJ_v)^*\!\otimes_{\C}\!\pi_2^*(TX,J)\big)\!:
\tnd\phi\!\circ\!\nu\!=\!\nu\!\circ\!\tnd\si_v\big\},
&\quad&v\!\in\!\nV_{\R}(\ga),\\
\Ga^{0,1}_{v;\bu}(X;J)^{\phi}  
&=\big\{\nu\!\in\!\Ga\big(\wt\cU_v^{\bu}\!\times\!X;
\pi_1^*(\cT_v^{\bu},-\fJ_v^{\bu})^*\!\otimes_{\C}\!\pi_2^*(TX,J)\big)\!:
\tnd\phi\!\circ\!\nu\!=\!\nu\!\circ\!\tnd\si_v\big\},
&\quad&v\!\in\!\nV_{\C}(\ga).
\end{alignat*}

\vspace{.2in}

\noindent
Let $v\!\in\!\nV_{\R}(\ga)$.
Denote by $\fB_v$ the space of tuples
\BE{uvRdfn_e}\u_v\equiv \big(u_v\!:\Si_v\!\lra\!X,
(z_f)_{f\in\ve^{-1}(v)},\si,\fJ\big)\EE
so that $(\Si,(z_f)_{f\in\ve^{-1}(v)},\si,\fJ)$ is a fiber of~\eref{pigarhovR_e}
and $u_v$~is
\begin{enumerate}[label=$\bu$,leftmargin=*]

\item   a smooth $(\phi,\si)$-real degree~$\fd(v)$ map if 
$v\!\in\!\nV_{\R}(\ov\ga)\!\cup\!\ale_{\R}(\ga,\ov\ga)_{\bu}$,
\item is  a constant $(\phi,\si)$-real  map if  $v\!\in\!\ale_{\R}(\ga,\ov\ga)_0$.
\end{enumerate}
In the first case, let
\BE{DefObsR_e}\begin{split}
&\Ga(\u_v)=\Ga(u_v)^{\phi,\si}, \qquad 
\Ga_J^{0,1}(\u_v)=\Ga^{0,1}_{J,\fJ}(u_v)^{\phi,\si},\\
&\Ga_0(\u_v)=\big\{\xi\!\in\!\Ga(\u_v)\!:
\xi(z_f)\!=\!0~\forall\,f\!\in\!\ve^{-1}(v)\big\}.
\end{split}\EE
In the second case,
we take $\Ga(\u_v)$ to be the space of constant real sections of $u_v^*TX$
and $\Ga_J^{0,1}(\u_v)$ and $\Ga_0(\u_v)$  to be the zero vector spaces.
For $f\!\in\!\ve^{-1}(v)$, let
$$\ev_f\!:\fB_v\lra\begin{cases}X,&\hbox{if}~f\!\in\!S_{v;\C}(\ga);\\
X^{\phi},&\hbox{if}~f\!\in\!S_{v;\R}(\ga);\end{cases} \qquad
L_f\!:\Ga(\u_v)\lra \begin{cases}T_{\ev_f(\u_v)}X,&\hbox{if}~f\!\in\!S_{v;\C}(\ga);\\
T_{\ev_f(\u_v)}X^{\phi},&\hbox{if}~f\!\in\!S_{v;\R}(\ga);\end{cases}$$
be the evaluation~maps at the marked point~$z_f$ corresponding to~$f$.\\

\noindent
For $J\!\in\!\cJ_{\om}^{\phi}$ and $\nu\!\in\!\Ga^{0,1}_v(X;J)^{\phi}$, 
let $\wt\fM_v(J,\nu)\!\subset\!\fB_v$
be the subspace of tuples~\eref{uvRdfn_e} such~that  
\BE{fMfBcond_e}\dbar_{J,\fJ}u_v\big|_z=\nu\big(z,u_v(z)\big)
\qquad\forall\,z\!\in\!\Si_v\,.\EE
Denote~by
\BE{wtfMfg_e1b}
\fB_v^*\subset\fB_v \qquad\hbox{and}\qquad
\wt\fM_v^*(J,\nu) \subset\wt\fM_v(J,\nu) \EE
\begin{enumerate}[label=$\bu$,leftmargin=*]

\item the subspaces of simple maps if $v\!\in\!\ale_{\R}(\ga,\ov\ga)_{\bu}$,

\item the entire spaces if $v\!\in\!\nV_{\R}(\ov\ga)\!\cup\!\ale_{\R}(\ga,\ov\ga)_0$.

\end{enumerate}
For $\u_v\!\in\!\wt\fM_v(J,\nu)$, let
$$D_{J,\nu;u_v}\!:\Ga(u_v)\lra\Ga^{0,1}_{J,\fJ}(u_v)
\qquad\hbox{and}\qquad
D_{J,\nu;\u_v}^0\!: \Ga_0(\u_v)\lra \Ga_J^{0,1}(\u_v)$$
be the linearization of $\dbar_J\!-\!\nu$ at~$\u_v$
and its restriction.
If \hbox{$v\!\in\!\ale_{\R}(\ga,\ov\ga)$} and $W\!\subset\!X$, define 
$$\wh\Ga_{J,\nu;W}^{0,1}(\u_v)=\Im\,D_{J,\nu;\u_v}^0+
\big\{A\!\circ\!\tnd u_v\!\circ\!\fJ\!: 
A\!\in\!T_J\cJ_{\om}^{\phi},\,\supp(A)\!\subset\!W\big\}.$$
If $v\!\in\!\nV_{\R}(\ov\ga)$, $\io_v\!:\Si_v\!\lra\!\R\wt\cU_{g,l;k}$
is a normalization of a real irreducible component of a fiber of~\eref{RwtcU_e},
and  \hbox{$W\!\subset\!\R\wt\cU_{g,l;k}$}, define
$$\wh\Ga_{J,\nu;W}^{0,1}(\u_v)=\Im\,D_{J,\nu;\u_v}^0+
\Big\{ \big\{\io_v\!\times\!u_v\big\}^*\nu'\!:
\nu'\!\in\!\Ga^{0,1}_{g,l;k}(X;J)^{\phi},\,\supp(\nu')\!\subset\!W\!\times\!X\Big\}.$$

\vspace{.2in}

\noindent
Let $v\!\in\!\nV_{\C}(\ga)$.
Denote by $\fB_v^{\bu}$ the space of tuples
\BE{uvCdfn_e}\u_v^{\bu}\equiv \big(u_v\!\sqcup\!u_{\si(v)}:
\Si_v\!\sqcup\!\Si_{\si(v)}\!\lra\!X,(z_f)_{f\in\ve^{-1}(v)\sqcup\ve^{-1}(\si(v))},\si,\fJ\big)\EE
so that $(\Si_v\!\sqcup\!\Si_{\si(v)},(z_f)_{f\in\ve^{-1}(v)\sqcup\ve^{-1}(\si(v))},\si,\fJ)$
is a fiber of~\eref{pigarhovC_e}, \hbox{$u_{\si(v)}\!=\!\phi\!\circ\!u_v\!\circ\!\si$},
and $u_v$~is 
\begin{enumerate}[label=$\bu$,leftmargin=*]

\item  a smooth degree~$\fd(v)$ map if 
$v\!\in\!\nV_{\C}(\ov\ga)\!\cup\!\ale_{\C}(\ga,\ov\ga)_{\bu}$,

\item a constant  map if  \hbox{$v\!\in\!\ale_{\C}(\ga,\ov\ga)_0$}.

\end{enumerate}
In the first case, let
\BE{DefObsC_e}\begin{split}
\Ga(\u_v^{\bu})&=\big\{(\xi,\xi')\!\in\!\Ga(u_v)\!\oplus\!\Ga(u_{\si(v)})\!:
\xi'\!=\!\tnd\phi\!\circ\!\xi\!\circ\!\si\big\},\\
\Ga_J^{0,1}(\u_v^{\bu})&=\big\{(\eta,\eta')\!\in\!
\Ga^{0,1}_{J,\fJ}(u_v)\!\oplus\!\Ga^{0,1}_{J,\fJ}(u_{\si(v)})\!:
\eta'\!=\!\tnd\phi\!\circ\!\eta\!\circ\!\tnd\si\big\},\\
\Ga_0(\u_v^{\bu})&=\big\{(\xi,\xi')\!\in\!\Ga(\u_v^{\bu})\!:\,
\xi(z_f)\!=\!0~\forall\,f\!\in\!\ve^{-1}(v)\big\}.
\end{split}\EE  
In the second case, we take $\Ga(\u_v^{\bu})$ to be the space of pairs 
$(\xi,\xi')$ as above so that $\xi|_{\Si_v}$ is constant and
$\Ga_J^{0,1}(\u_v^{\bu})$ and $\Ga_0(\u_v^{\bu})$ to be the zero vector spaces.
For \hbox{$f\!\in\!\ve^{-1}(v)\!\cup\!\ve^{-1}(\si(v))$}, define
\begin{alignat*}{2}
\ev_f\!:\fB_v^{\bu}&\lra X,&\quad
\ev_f(\u_v^{\bu})&=
\begin{cases}u_v(z_f) &\hbox{if}~f\!\in\!\ve^{-1}(v);\\
u_{\si(v)}(z_f),&\hbox{if}~f\!\in\!\ve^{-1}(\si(v));\end{cases}\\
L_f\!:\Ga(\u_v^{\bu})&\lra T_{\ev_f(\u_v^{\bu})}X, &\qquad
L_f(\xi,\xi')&=\begin{cases}\xi(z_f),&\hbox{if}~f\!\in\!\ve^{-1}(v);\\
\xi'(z_f),&\hbox{if}~f\!\in\!\ve^{-1}(\si(v)).\end{cases}
\end{alignat*}

\vspace{.2in}

\noindent
For $J\!\in\!\cJ_{\om}^{\phi}$ and $\nu\!\in\!\Ga^{0,1}_{v;\bu}(X;J)^{\phi}$,
let $\wt\fM_v^{\bu}(J,\nu)\!\subset\!\fB_v^{\bu}$
be the subspace of tuples~\eref{uvCdfn_e} satisfying~\eref{fMfBcond_e}.
Denote by
\BE{wtfMfg_e2b}\fB_v^{\bu,*} \subset\fB_v^{\bu} 
\quad\hbox{and}\quad
\wt\fM_v^{\bu,*}(J,\nu)\subset\wt\fM_v^{\bu}(J,\nu) \EE

\begin{enumerate}[label=$\bu$,leftmargin=*]

\item the subspaces of simple maps such that 
$u_v(\Si_v)\!\neq\!u_{\si(v)}(\Si_{\si(v)})$ 
if $v\!\in\!\ale_{\C}(\ga,\ov\ga)_{\bu}$,

\item the entire spaces if $v\!\in\!\nV_{\C}(\ov\ga)\!\cup\!\ale_{\C}(\ga,\ov\ga)_0$.

\end{enumerate}
For $\u_v^{\bu}\!\in\!\wt\fM_v(J,\nu)$, let
$$D_{J,\nu;u_v}^{\bu}\!:\Ga(u_v)\!\oplus\!\Ga\big(u_{\si(v)}\big)
\lra\Ga^{0,1}_{J,\fJ}(u_v)\!\oplus\!\Ga^{0,1}_{J,\fJ}\big(u_{\si(v)}\big)
\quad\hbox{and}\quad
D_{J,\nu;\u_v^{\bu}}^0\!: \Ga_0(\u_v^{\bu})\lra \Ga_J^{0,1}(\u_v^{\bu})$$
be the linearization of $\dbar_J\!-\!\nu$ at~$\u_v^{\bu}$
and its restriction.
If \hbox{$v\!\in\!\ale_{\C}(\ga,\ov\ga)$} and $W\!\subset\!X$, define 
$$\wh\Ga_{J,\nu;W}^{0,1}(\u_v^{\bu})=\Im\,D_{J,\nu;\u_v^{\bu}}^0+
\big\{\big(A\!\circ\!\tnd u_v\!\circ\!\fJ,A\!\circ\!\tnd u_{\si(v)}\!\circ\!\fJ\big)\!: 
A\!\in\!T_J\cJ_{\om}^{\phi},\,\supp(A)\!\subset\!W\big\}.$$
If $v\!\in\!\nV_{\C}(\ov\ga)$, $\io_v^{\bu}\!:\Si_v\!\cup\!\Si_{\si(v)}\lra\!\R\wt\cU_{g,l;k}$
is a normalization of a conjugate pair of irreducible components of a fiber of~\eref{RwtcU_e},
and  \hbox{$W\!\subset\!\R\wt\cU_{g,l;k}$}, let
$$\wh\Ga_{J,\nu;W}^{0,1}(\u_v^{\bu})=\Im\,D_{J,\nu;\u_v^{\bu}}^0+
\Big\{ \big\{\io_v^{\bu}\!\times\!(u_v\!\sqcup\!u_{\si(v)})\big\}^*\nu'\!:
\nu'\!\in\!\Ga^{0,1}_{g,l;k}(X;J)^{\phi},\,\supp(\nu')\!\subset\!W\!\times\!X\Big\}.$$

\begin{lmm}\label{RTtransC_lmm}
Suppose $v\!\in\!\ale_{\C}(\ga,\ov\ga)_{\bu}\!\cup\!\nV_{\C}(\ov\ga)$,
$J\!\in\cJ_{\om}^{\phi}$, \hbox{$\nu\!\in\!\Ga^{0,1}_{v;\bu}(X;J)^{\phi}$},
and \hbox{$\u_v^{\bu}\!\in\!\wt\fM_v^{\bu,*}(J,\nu)$} is as in~\eref{uvCdfn_e}.
If $v\!\in\!\ale_{\C}(\ga,\ov\ga)_{\bu}$, 
let \hbox{$W\!\subset\!X$} be a $\phi$-invariant open subset
intersecting~$u_v(\Si_v)$.
If $v\!\in\!\nV_{\C}(\ov\ga)$, let 
$$\io_v^{\bu}\!:\Si_v\!\cup\!\Si_{\si(v)}\lra\R\wt\cU_{g,l;k}
\qquad\hbox{and}\qquad W\subset\R\wt\cU_{g,l;k}$$ 
be a normalization of a conjugate pair of irreducible components 
of a fiber of~\eref{RwtcU_e} and a $\wt\si_{\R}$-invariant open subset intersecting
$\io_v^{\bu}(\Si_v)$, respectively.
Then $\wh\Ga_{J,\nu;W}^{0,1}(\u_v^{\bu})\!=\!\Ga^{0,1}(\u_v^{\bu})$.
\end{lmm}

\begin{proof}
Denote~by 
$$D_{J,\nu;u_v}\!:\Ga(u_v)\lra \Ga^{0,1}_{J,\fJ}(u_v), \quad
D_{J,\nu;\u_v}^0\!:\Ga_0(\u_v)\!\equiv\!\big\{\xi\!\in\!\Ga(u_v)\!:\,
\xi(z_f)\!=\!0~\forall\,f\!\in\!\ve^{-1}(v)\big\}\lra \Ga^{0,1}_{J,\fJ}(u_v)$$
the restrictions of~$D_{J,\nu;\u_v}^{\bu}$ and~$D_{J,\nu;\u_v^{\bu}}^0$, respectively.
If $v\!\in\!\ale_{\C}(\ga,\ov\ga)$, define
$$\wh\Ga_{J,\nu;W}^{0,1}(\u_v)=\Im\,D_{J,\nu;\u_v}^0+
\big\{A\!\circ\!\tnd u_v\!\circ\!\fJ\!: A\!\in\!T_J\cJ_{\om}^{\phi},\,
\supp(A)\!\subset\!W\big\}.$$
If $v\!\in\!\nV_{\C}(\ov\ga)$, define
$$\wh\Ga_{J,\nu;W}^{0,1}(\u_v)=\Im\,D_{J,\nu;\u_v}^0+
\Big\{\big\{\io_v^{\bu}|_{\Si_v}\!\times\!u_v\big\}^*\nu'\!:
\nu'\!\in\!\Ga^{0,1}_{g,l;k}(X;J)^{\phi},\,\supp(\nu')\!\subset\!W\Big\}.$$

\vspace{.2in}

\noindent
The projections
$$\Ga_0(\u_v^{\bu})\lra\Ga_0(\u_v) \qquad\hbox{and}\qquad 
\Ga^{0,1}_{J,\fJ}(\u_v^{\bu})\lra \Ga^{0,1}_{J,\fJ}(u_v)$$
are isomorphisms intertwining $D_{J,\nu;\u_v^{\bu}}^0$ and $D_{J,\nu;\u_v}^0$.
The second projection maps $\wh\Ga_{J,\nu;W}^{0,1}(\u_v^{\bu})$ onto $\wh\Ga_{J,\nu;W}^{0,1}(\u_v)$.
Thus, the claim of the lemma is equivalent to 
\hbox{$\wh\Ga_{J,\nu;W}^{0,1}(\u_v)\!=\!\Ga^{0,1}(\u_v)$}.\\

\noindent
Denote~by
$$\Ga^{1,0}(\u_v)\equiv\Ga_J^{1,0}(\u_v)\subset 
\Ga\big(\Si_v\!-\!\ve^{-1}(v);(T\Si_v,\fJ)^*\!\otimes_{\C}\!u_v^*(TX,J)^*\big)$$
the subspace of smooth sections~$\eta$ with at most simple poles at the points 
of~$\ve^{-1}(v)$.
In other words, for every $z_0\!\in\!\ve^{-1}(v)$ there exists a holomorphic coordinate~$w$
on a neighborhood~$U$ of~$z_0$ in~$\Si_v$ such that 
$$w(z_0)=0 \qquad\hbox{and}\qquad 
w\cdot\eta|_U\in \Ga\big(U;(T\Si_v,\fJ)^*\!\otimes_{\C}\!u^*(TX,J)^*\big).$$
By \cite[Lemma~2.4.1]{Sh} and \cite[Lemma~2.3.2]{IvSh}, the cokernel of~$D_{J,\nu;\u_v}^0$ 
is isomorphic to the kernel of the formal adjoint $D_{\u_v}^*$ of~$D_{J,\nu;u_v}$
on $\Ga^{1,0}(\u_v)$ via the standard pairing of $(0,1)$- and $(1,0)$-forms
on~$\Si_v$; see also \cite[Sections~2.1,2.2]{FanoGV}.\\

\noindent
Let $\eta\!\in\!\ker D_{\u_v}^*\!-\!\{0\}$.
The only property of~$D_{\u_v}^*$ relevant for our purposes~is  
\begin{enumerate}[label=(P\arabic*),leftmargin=*]

\item\label{vaneta_it} $\eta$ does not vanish
on any non-empty open subset of~$\Si_v$.

\end{enumerate}
As with \cite[Proposition~3.2.1]{MS} and \cite[Proposition~3.2]{RT2},
the proof comes down to showing that~$\eta$ pairs non-trivially with some 
element of $T_J\cJ_{\om}^{\phi}$ if $v\!\in\!\ale_{\C}(\ga,\ov\ga)_{\bu}$
and of $\Ga^{0,1}_{g,l;k}(X;J)^{\phi}$ if $v\!\in\!\nV_{\C}(\ov\ga)$.\\

\noindent
Suppose $v\!\in\!\ale_{\C}(\ga,\ov\ga)_{\bu}$.
Since $u_v$ is simple and \hbox{$\phi(u_v(\Si_v))\!\neq\!u_v(\Si_v)$},
we can assume that there exist non-empty open subsets
$U\!\subset\!\Si_v$  and $W'\!\subset\!X$
such~that $u_v|_U$ is an embedding, 
\BE{RTtrans_e5} W=W'\sqcup \phi(W'), \qquad
u_v(\Si_v)\cap W'=u_v(U), \qquad u_v(\Si_v)\cap\phi(W')=\eset\,.\EE
The proof of \cite[Proposition~3.2.1]{MS} provides $A'\!\in\!T_J\cJ_{\om}$ 
such~that 
\BE{RTtrans_e7}\supp(A')\subset W' \qquad\hbox{and}\qquad
\int_{\Si_v}\!\!\big(A'\!\circ\!\tnd u_v\!\circ\!\fJ\big)\!\w\!\eta\neq0\,.\EE
We define $A\!\in\!T_J\cJ_{\om}^{\phi}$ by 
\BE{RTtrans_e9}A|_{W'}=A', \qquad A|_{\phi(W')}=-\phi^*A', \qquad A|_{X-W}=0\,.\EE
By the last assumption in~\eref{RTtrans_e5}, $A$ pairs non-trivially with~$\eta$.\\

\noindent
Let $v\!\in\!\ov\nV_{\C}(\ov\ga)$.
Since $\io_v^{\bu}$ is injective outside of finitely many points, 
we can assume that there exist
non-empty open subsets
$U\!\subset\!\Si_v$ and $W'\!\subset\!\R\wt\cU_{g,l;k}$
such~that 
\BE{RTtrans_e25}
W=W'\sqcup \wt\si_{\R}(W'), \quad 
\io_v^{\bu}(\Si_v)\cap W'= \io_v^{\bu}(U), \quad
\io_v^{\bu}(\Si_v)\cap\wt\si_{\R}(W')=\eset.\EE
The reasoning in the proof of \cite[Proposition~3.2.1]{MS} 
applied in the setting of \cite[Proposition~3.2]{RT2} provides
$\nu''\!\in\!\Ga^{0,1}_{g,l;k}(X;J)$
such~that 
\BE{RTtrans_e27}\supp(\nu'')\subset W'\!\times\!X \qquad\hbox{and}\qquad
\int_{\Si_v}\!\!\big(\{\io_v^{\bu}|_{\Si_v}\!\times\!u_v\}^*\nu''\big)\!\w\!\eta\neq0\,.\EE
We define $\nu'\!\in\!\Ga^{0,1}_{g,l;k}(X;J)^{\phi}$ by 
\BE{RTtrans_e29}\nu'|_{W'\times X}=\nu'', \quad 
\nu'|_{\wt\si_{\R}(W')\times X}=\big\{\wt\si_{\R}\!\times\!\phi\big\}^*\nu'', 
\quad \nu'|_{(\R\wt\cU_{g,l;k}-W)\times X}=0\,.\EE
By the last assumption in~\eref{RTtrans_e25}, $\nu'$ pairs non-trivially with~$\eta$.
\end{proof}

\begin{lmm}\label{RTtransR_lmm}
Suppose $v\!\in\!\ale_{\R}(\ga,\ov\ga)_{\bu}\!\cup\!\nV_{\R}(\ov\ga)$,
$J\!\in\cJ_{\om}^{\phi}$, \hbox{$\nu\!\in\!\Ga^{0,1}_v(X;J)^{\phi}$},
and \hbox{$\u_v\!\in\!\wt\fM_v^*(J,\nu)$} is as in~\eref{uvRdfn_e}.
If $v\!\in\!\ale_{\R}(\ga,\ov\ga)_{\bu}$, 
let \hbox{$W\!\subset\!X$} be a $\phi$-invariant open subset
intersecting~$u_v(\Si_v)$.
If $v\!\in\!\nV_{\R}(\ov\ga)$, let 
$$\io_v\!:\Si_v\lra\R\wt\cU_{g,l;k} \qquad\hbox{and}\qquad W\subset\R\wt\cU_{g,l;k}$$ 
be a normalization of a real irreducible component
of a fiber of~\eref{RwtcU_e} and a $\wt\si_{\R}$-invariant open subset intersecting
$\io_v(\Si_v)$, respectively.
Then $\wh\Ga_{J,\nu;W}^{0,1}(\u_v)\!=\!\Ga^{0,1}(\u_v)$.
\end{lmm}

\begin{proof}
With $\Ga_J^{1,0}(\u_v)$ as in the proof of Lemma~\ref{RTtransC_lmm}, define
$$\Ga^{1,0}(\u_v)=
\{\eta\!\in\!\Ga_J^{1,0}(\u_v)\!:\,\tnd\phi\!\circ\!\eta\!=-\eta\!\circ\!\tnd\si\}.$$
Since the involution~$\si$ on~$\Si_v$ is orientation-reversing,
the pairing of $(0,1)$- and $(1,0)$-forms on~$\Si_v$ of Lemma~\ref{RTtransC_lmm}
restricts to a pairing between the invariant subspaces $\Ga_J^{0,1}(\u_v)$ and
$\Ga^{1,0}(\u_v)$ and induces an isomorphism between 
the cokernel of~$D_{J,\nu;\u_v}^0$ and the kernel of $D_{\u_v}^*$ 
on $\Ga^{1,0}(\u_v)$ as before.\\

\noindent
Let $\eta\!\in\!\ker D_{\u_v}^*\!-\!\{0\}$. 
The only properties of~$D_{\u_v}^*$ relevant for our purposes~are \ref{vaneta_it} and
\begin{enumerate}[label=(P\arabic*),leftmargin=*]

\setcounter{enumi}{1}

\item\label{realeta_it} $\tnd\phi\!\circ\!\eta\!=\!-\eta\!\circ\!\tnd\si$.

\end{enumerate}
We show that~$\eta$ pairs non-trivially with some 
element of $T_J\cJ_{\om}^{\phi}$ if $v\!\in\!\ale_{\R}(\ga,\ov\ga)_{\bu}$
and of $\Ga^{0,1}_{g,l;k}(X;J)^{\phi}$ if $v\!\in\!\nV_{\R}(\ov\ga)$.\\

\noindent
Suppose $v\!\in\!\ale_{\R}(\ga,\ov\ga)_{\bu}$;
thus, $\phi\!\circ\!u_v\!=\!u_v\!\circ\!\si$.
Since $u_v$ is simple, we can assume there exist non-empty open subsets \hbox{$U\!\subset\!\Si_v$} 
and $W'\!\subset\!X$ satisfying the  condition before~\eref{RTtrans_e5} and
the first two conditions in~\eref{RTtrans_e5}.
Let $A'\!\in\!T_J\cJ_{\om}$ be as in~\eref{RTtrans_e7}
and define $A\!\in\!T_J\cJ_{\om}^{\phi}$ by~\eref{RTtrans_e9}.
By~\ref{realeta_it},
\BE{RTtrans_e15}\big( (A\!\circ\!\tnd u_v\!\circ\!\fJ)\!\w\!\eta\big)\big|_{\si(U)}
=-\si^*\Big(\big( (A\!\circ\!\tnd u_v\!\circ\!\fJ)\!\w\!\eta\big)\big|_U\Big)\,.\EE
We conclude that 
\BE{RTtrans_e17}\int_{\Si_v}\!\!\big(A\!\circ\!\tnd u_v\!\circ\!\fJ\big)\!\w\!\eta
=\int_U\!\!\big(A\!\circ\!\tnd u_v\!\circ\!\fJ\big)\!\w\!\eta
+\int_{\si(U)}\!\!\big(A\!\circ\!\tnd u_v\!\circ\!\fJ\big)\!\w\!\eta
=2\int_U\!\!\big(A'\!\circ\!\tnd u_v\!\circ\!\fJ\big)\!\w\!\eta\neq0,\EE
i.e.~$A$ pairs non-trivially with~$\eta$.\\

\noindent
Let $v\!\in\!\nV_{\R}(\ov\ga)$.
We can assume that there exist non-empty open subsets
\hbox{$U\!\subset\!\Si_v$} and \hbox{$W'\!\subset\!\R\wt\cU_{g,l;k}$}
satisfying the first two conditions in~\eref{RTtrans_e25} with~$\io_v^{\bu}$
replaced by~$\io_v$. 
Let \hbox{$\nu''\!\in\!\Ga^{0,1}_{g,l;k}(X;J)$} be as in~\eref{RTtrans_e27}
and define \hbox{$\nu'\!\in\!\Ga^{0,1}_{g,l;k}(X;J)^{\phi}$} by~\eref{RTtrans_e29}.
By~\ref{realeta_it}, \eref{RTtrans_e15} and~\eref{RTtrans_e17} 
hold with $A\!\circ\!\tnd u_v\!\circ\!\fJ$ replaced by 
$\{\io_v\!\times\!u_v\}^*\nu'$.
\end{proof}

\subsection{Universal moduli spaces}
\label{RTregpf_subs}

\noindent
Let $\E_{\C}(\ga),\E_{\R\C}(\ga),\E_{\R\R}(\ga)\!\subset\!\E(\ga)$
be as defined in Section~\ref{RTstr_subs}.
Choose \hbox{$\E_+(\ga)\!\subset\!\E_{\C}(\ga)$}
so~that
\BE{Eplus_e}\E_{\C}(\ga)=\E_+(\ga)\sqcup\si\big(\E_+(\ga)\big).\EE
For each $e\!\in\!\E_{\R\C}(\ga)$, choose a flag $f_e\!\in\!e$.\\

\noindent
For $e\!\in\!\E_+(\ga)$ and $e\!\in\!\E_{\R}(\ga)$, let
$$\De_{\ga;e}\subset X_{\ga;e}\!\equiv\!\prod_{f\in e}\!X \qquad\hbox{and}\qquad
\De_{\ga;e}^{\phi}\subset X_{\ga;e}^{\phi}\!\equiv\!\prod_{f\in e}\!X^{\phi}$$
be the respective diagonals. 
Define
\BE{DeXdfn_e}
\De_{\ga}\equiv \prod_{e\in\E_+(\ga)}\!\!\!\!\!\!\De_{\ga;e}
\times\prod_{e\in\E_{\R\C}(\ga)}\!\!\!\!\!\!X^{\phi}\times
\prod_{e\in\E_{\R\R}(\ga)}\!\!\!\!\!\!\!\De_{\ga;e}^{\phi}
\subset X_{\ga}\equiv \prod_{e\in\E_+(\ga)}\!\!\!\!\!\!X_{\ga;e}\times
\prod_{e\in\E_{\R\C}(\ga)}\!\!\!\!\!\!\!X\times
\prod_{e\in\E_{\R\R}(\ga)}\!\!\!\!\!\!\!X_{\ga;e}^{\phi}\,.\EE
The evaluation maps~$\ev_f$ induce a map
\BE{evgadfn_e} 
\ev_{\ga}\!\equiv\!\prod_{e\in\E_+(\ga)}\prod_{f\in e}\!\ev_f\times
\prod_{e\in\E_{\R\C}(\ga)}\!\!\!\!\!\!\!\ev_{f_e}\times
\prod_{e\in\E_{\R\R}(\ga)}\prod_{f\in e}\!\ev_f \!:
\prod_{v\in\nV_{\R}(\ga)}\!\!\!\!\!\!\fB_v^* \times\!\!
\prod_{\{v,\si(v)\}\subset\nV_{\C}(\ga)}\!\!\!\!\!\!\!\!\!\!\!\!\!\!\fB_v^{\bu,*}
\lra X_{\ga}\EE
Define
\BE{fBstrdfn_e}\begin{split}
\fB_{\ga,\vp;\rho}^{\star}&=\big\{\u\in
\prod_{v\in\nV_{\R}(\ga)}\!\!\!\!\!\!\fB_v^* \times\!\!
\prod_{\{v,\si(v)\}\subset\nV_{\C}(\ga)}\!\!\!\!\!\!\!\!\!\!\!\!\!\!\fB_v^{\bu,*}\!:
\ev_{\ga}(\u)\!\in\!\De_{\ga}\big\}, \\
\fB_{\ga,\vp;\rho}^*&=\big\{\u\!\in\!\fB_{\ga,\vp;\rho}^{\star}\!:
u_{v_1}(\Si_{v_1})\!\neq\!u_{v_2}(\Si_{v_2})~\forall\,v_1,v_2\!\in\!\ale(\ga,\ov\ga)_{\bu},
v_1\!\neq\!v_2\big\}\,.
\end{split}\EE

\vspace{.2in}

\noindent
An element of $\fB_{\ga,\vp;\rho}^{\star}$ corresponds to a tuple 
\BE{bfudfn_e}\u\equiv \big(u\!:\Si\!\lra\!X,(z_f)_{f\in S_{l;k}},\si,\fJ\big)
\equiv \big((\u_v)_{v\in\nV_{\R}(\ga)},
(\u_v^{\bu})_{\{v,\si(v)\}\subset\nV_{\C}(\ga)}\big)
\equiv \big(\u_v\big)_{v\in\Ver},\EE
where $(\Si,(z_f)_{f\in S_{l;k}},\si,\fJ)$
is a fiber of~\eref{pigarho_e} and $u$ is a smooth $(\phi,\si)$-real map such~that  
$$ u_*[\Si_v]=\fd(v)\in H_2(X;\Z) \quad\forall\,v\!\in\!\Ver\,,
\qquad u|_{\Si_v}=\tn{const}\quad\forall\,v\!\in\!\ale(\ga,\ov\ga)_0.$$
For a tuple~$\u$ as above, let
\BE{fFfibdfn_e1}
\Ga_0(\u)=\bigoplus_{v\in\nV_{\R}(\ga)}\!\!\!\!\!\Ga_0(\u_v)\!\oplus\!\!
\bigoplus_{\{v,\si(v)\}\subset\nV_{\C}(\ga)}\hspace{-.32in}\Ga_0(\u_v^{\bu}).
\EE
If in addition $J\!\in\!\cJ_{\om}^{\phi}$, let
\BE{fFfibdfn_e2}
\Ga_J^{0,1}(\u)=\bigoplus_{v\in\nV_{\R}(\ga)}\!\!\!\!\!\Ga_J^{0,1}(\u_v)\!\oplus\!\!
\bigoplus_{\{v,\si(v)\}\subset\nV_{\C}(\ga)}\hspace{-.32in}\Ga_J^{0,1}(\u_v^{\bu}).\EE

\vspace{.2in}

\noindent
Denote by 
\BE{fFbndldfn_e}\fF\lra\cH_{g,l;k}^{\om,\phi}(X)\!\times\!\fB_{\ga,\vp;\rho}^*\EE
the bundle with fibers $\fF_{(J,\nu;\u)}\!=\!\Ga^{0,1}_J(\u)$.
We define a section of this bundle
\BE{dbardfn_e}\dbar\!: \cH_{g,l;k}^{\om,\phi}(X)\!\times\!\fB_{\ga,\vp;\rho}^* \lra\fF
\qquad\hbox{by}\quad
\dbar(J,\nu;\u)\big|_z=\dbar_{J,\fJ}u\big|_z-\nu_{\ga,\vp;\rho}\big(z,u(z)\big)
~~\forall\,z\!\in\!\Si\,.\EE
The zero set of this section 
\BE{UfMdfn_e} \fU\wt\fM_{\ga,\vp;\rho}^*(J,\nu)
\equiv\big\{(J,\nu;\u)\!\in\!\cH_{g,l;k}^{\om,\phi}(X)\!\times\!\fB_{\ga,\vp;\rho}^*\!:
\dbar(J,\nu;\u)\!=\!0\big\}\EE
is the \sf{universal moduli space}.
The preimage of a pair $(J,\nu)$ in $\cH_{g,l;k}^{\om,\phi}(X)$ under the projection
\BE{pifUfMdfn_e} \pi\!: 
\fU\wt\fM_{\ga,\vp;\rho}^*(J,\nu)\lra\cH_{g,l;k}^{\om,\phi}(X)\EE
is the second subspace in~\eref{fMga_e2}.\\

\noindent
For the purposes of applying the Sard-Smale Theorem \cite[(1.3)]{SardSmale},
we complete 
\begin{enumerate}[label=$\bu$,leftmargin=*]

\item the map components of the spaces $\fB_v$ and $\fB_v^{\bu}$
and the spaces~\eref{fFfibdfn_e1} in the $L^p_2$ Sobolev norm for
some \hbox{$p\!>\!2$} (fixed),

\item the parameter space~\eref{cHdfn_e} in the $C^m$ H\"older norm for some $m\!\ge\!2$
(to be chosen later), and

\item the spaces~\eref{fFfibdfn_e2} in the $L^p_1$ Sobolev norm.

\end{enumerate}
We denote the completed spaces and the induced spaces 
in~\eref{wtfMfg_e1b},  \eref{wtfMfg_e2b},
\eref{fBstrdfn_e}-\eref{fFbndldfn_e}, and~\eref{UfMdfn_e} as before.\\

\noindent
By the assumption $p\!>\!2$ and the Sobolev Embedding Theorem \cite[Theorem~B.1.12]{MS}, 
the map component of any element of the completed spaces~$\fB_v$ 
and~$\fB_v^{\bu}$ 
is a $C^1$-map. By the reasoning at the top of \cite[p47]{MS}, these completions are 
smooth separable Banach manifolds.
By the reasoning at the bottom of \cite[p49]{MS}, the (completed) parameter space
$\cH_{g,l;k}^{\om,\phi}(X)$ is a smooth separable Banach manifold.
By the reasoning at the bottom of \cite[p50]{MS}, 
\eref{dbardfn_e} is the restriction of a $C^{m-2}$ section of $C^{m-2}$ 
Banach bundle~\eref{fFbndldfn_e} over the product of the spaces~$\fB_v$
and~$\fB_v^{\bu}$.

\begin{prp}\label{UfM_prp}
Let $p$ and $m$ be as above. Then 
\BE{UfMprp_e1}\fU\wt\fM_{\ga,\vp;\rho}^*(J,\nu) \subset 
\cH_{g,l;k}^{\om,\phi}(X)\times 
\prod_{v\in\nV_{\R}(\ga)}\!\!\!\!\!\!\fB_v \times\!\!
\prod_{\{v,\si(v)\}\subset\nV_{\C}(\ga)}\!\!\!\!\!\!\!\!\!\!\!\!\!\!\fB_v^{\bu}\EE
is a separable $C^{m-2}$  Banach submanifold and the projection~\eref{pifUfMdfn_e}
is a $C^{m-2}$ Fredholm~map of~index
\BE{UfMprp_e2}\ind_{\R}\pi=\dim_{g',l;k}(B')-|\ga|
+n\,\ell\big(\ga,\ale(\ga,\ov\ga)_0\big)+\dim_{\R}\!G_{\ga;\rho}\,.\EE
\end{prp}

\begin{lmm}\label{UfB_lmm}
Let $p$ and $m$ be as in Proposition~\ref{UfM_prp}. Then 
$$\fB_{\ga,\vp;\rho}^* \subset 
\prod_{v\in\nV_{\R}(\ga)}\!\!\!\!\!\!\fB_v \times\!\!
\prod_{\{v,\si(v)\}\subset\nV_{\C}(\ga)}\!\!\!\!\!\!\!\!\!\!\!\!\!\!\fB_v^{\bu}$$
is a separable  Banach submanifold of codimension
\BE{UfBlmm_e2}\codim_{\R}\fB_{\ga,\vp;\rho}^*=n\Big(|\ga|\!-\!
\ell\big(\ga,\ale(\ga,\ov\ga)_0\big)\Big)\,. \EE
\end{lmm}

\begin{proof} 
With $\sE_{\ga}(\sV)$ as defined in~\eref{sVsEsF_e}, let
\begin{gather*}
\E_+(\ga,\ov\ga)=\E_+(\ga)\!-\!\sE_{\ga}\big(\ale(\ga,\ov\ga)_0\big),\quad
\E_{\R\C}(\ga,\ov\ga)=\E_{\R\C}(\ga)\!-\!\sE_{\ga}\big(\ale(\ga,\ov\ga)_0\big),\\
\E_{\R}(\ga,\ov\ga)=\E_{\R}(\ga)\!-\!\sE_{\ga}\big(\ale(\ga,\ov\ga)_0\big).
\end{gather*}
Define
\begin{gather*}
\De_{\ga,\ov\ga}\!\equiv\! \prod_{e\in\E_+(\ga,\ov\ga)}\!\!\!\!\!\!\!\!\De_{\ga;e}
\times\prod_{e\in\E_{\R\C}(\ga,\ov\ga)}\!\!\!\!\!\!\!\!X^{\phi}\times
\prod_{e\in\E_{\R\R}(\ga,\ov\ga)}\!\!\!\!\!\!\!\!\!\De_{\ga;e}^{\phi}
\subset X_{\ga,\ov\ga}\!\equiv\!
 \prod_{e\in\E_+(\ga,\ov\ga)}\!\!\!\!\!\!\!\!X_{\ga;e}\times
\prod_{e\in\E_{\R\C}(\ga,\ov\ga)}\!\!\!\!\!\!\!\!\!X\times
\prod_{e\in\E_{\R\R}(\ga,\ov\ga)}\!\!\!\!\!\!\!\!\!X_{\ga;e}^{\phi}\,,\\
\ev_{\ga,\ov\ga}\!\equiv\!\prod_{e\in\E_+(\ga,\ov\ga)}\prod_{f\in e}\!\ev_f\times
\prod_{e\in\E_{\R\C}(\ga,\ov\ga)}\!\!\!\!\!\!\!\!\ev_{f_e}\times
\prod_{e\in\E_{\R\R}(\ga,\ov\ga)}\prod_{f\in e}\!\ev_f \!:
\prod_{v\in\nV_{\R}(\ga)}\!\!\!\!\!\!\fB_v^* \times\!\!
\prod_{\{v,\si(v)\}\subset\nV_{\C}(\ga)}\!\!\!\!\!\!\!\!\!\!\!\!\!\!\fB_v^{\bu,*}
\lra X_{\ga,\ov\ga}\,.
\end{gather*}
For each element $\u\!\in\!\fB_{\ga,\vp;\rho}^{\star}$, let
\begin{equation*}\begin{split}
&L_{\ga,\ov\ga;\u}\equiv\!
\bigoplus_{e\in\E_+(\ga,\ov\ga)}\bigoplus_{f\in e}\!L_f\oplus
\bigoplus_{e\in\E_{\R\C}(\ga,\ov\ga)}\!\!\!\!\!\!\!\!L_{f_e}\oplus
\bigoplus_{e\in\E_{\R\R}(\ga,\ov\ga)}\bigoplus_{f\in e}\!L_f \!:\\
&\hspace{.2in}
\bigoplus_{v\in\ale_{\R}(\ga,\ov\ga)_0}\!\!\!\!\!\!\!\!\Ga_0(\u_v)\!\oplus\!\!
\bigoplus_{v\in\ale_{\R}(\ga,\ov\ga)_0^c}\!\!\!\!\!\!\!\!\Ga(\u_v)\!\oplus\!\!
\bigoplus_{\{v,\si(v)\}\subset\ale_{\C}(\ga,\ov\ga)_0}\hspace{-.4in}\Ga_0(\u_v^{\bu})
\!\oplus\!\!
\bigoplus_{\{v,\si(v)\}\subset\ale_{\C}(\ga,\ov\ga)_0^c}\hspace{-.4in}\Ga(\u_v^{\bu})
\lra T_{\ev_{\ga,\ov\ga}(\u)}X_{\ga,\ov\ga}\,.
\end{split}\end{equation*}

\vspace{.2in}

\noindent
For $v\!\in\!\ale_{\R}(\ga,\ov\ga)_0$, the map component~$u_v$ of every $\u_v\!\in\!\fB_v^*$ 
is constant.
For $v\!\in\!\ale_{\C}(\ga,\ov\ga)_0$, the map components $u_v$ and $u_{\si(v)}$
of every $\u_v^{\bu}\!\in\!\fB_v^{\bu,*}$ are constant.
Thus,
\begin{equation*}\begin{split}
\fB_{\ga,\vp}^0&\equiv\Big\{\u\in
\prod_{v\in\ale_{\R}(\ga,\ov\ga)_0}\!\!\!\!\!\!\!\!\!\fB_v^* \times\!\!
\prod_{\{v,\si(v)\}\subset\ale_{\C}(\ga,\ov\ga)_0}\hspace{-.42in}\fB_v^{\bu,*}\!:
\ev_f(u_{\ve(f)})\!=\!\ev_{f'}(u_{\ve(f')})~\forall\,
\{f,f'\}\!\in\!\sE_{\ga}\big(\ale(\ga,\ov\ga)_0\big)\Big\}\\
&\subset 
\prod_{v\in\ale_{\R}(\ga,\ov\ga)_0}\!\!\!\!\!\!\!\!\!\fB_v^* \times\!\!
\prod_{\{v,\si(v)\}\subset\ale_{\C}(\ga,\ov\ga)_0}\hspace{-.42in}\fB_v^{\bu,*}
\end{split}\end{equation*}
is a smooth submanifold of dimension $n|\pi_0(\ga,\ale(\ga,\ov\ga)_0)|$
and thus of codimension 
\BE{UfBlmm_e15} \codim_{\R}\fB_{\ga,\vp}^0
=n\Big(\big|\sE_{\ga}\big(\ale(\ga,\ov\ga)_0\big)\big|\!-\!
\ell\big(\ga,\ale(\ga,\ov\ga)_0\big)\Big)\,.\EE
By definition,
\BE{UfBlmm_e17}
\fB_{\ga,\vp;\rho}^{\star}=\Big\{\u\in\fB_{\ga,\vp}^0\!\times\!
\prod_{v\in\ale_{\R}(\ga,\ov\ga)_0^c}\hspace{-.23in}\fB_v^* 
\times
\prod_{\{v,\si(v)\}\subset\ale_{\R}(\ga,\ov\ga)_0^c}
\hspace{-.41in}\fB_v^{\bu,*}\!:\ev_{\ga,\ov\ga}(\u)\!\in\!\De_{\ga,\ov\ga}\Big\}\,.\EE

\vspace{.2in}

\noindent
For each $v\!\in\!\nV_{\R}(\ga)$, $\fB_v^*$ is an open subset of~$\fB_v$. 
For each $v\!\in\!\nV_{\C}(\ga)$, $\fB_v^{\bu,*}$ is an open subset of~$\fB_v^{\bu}$. 
Since $\fB_{\ga,\vp;\rho}^*$ is an open subset of $\fB_{\ga,\vp;\rho}^{\star}$,
it is sufficient to show~that 
$$\fB_{\ga,\vp;\rho}^{\star} \subset 
\prod_{v\in\nV_{\R}(\ga)}\!\!\!\!\!\!\fB_v^* \times\!\!
\prod_{\{v,\si(v)\}\subset\nV_{\C}(\ga)}\!\!\!\!\!\!\!\!\!\!\!\!\!\!\fB_v^{\bu,*}$$
is a Banach submanifold of codimension~\eref{UfBlmm_e2}.
We show below that 
\BE{UfB_e5} T_{\ev_{\ga,\ov\ga}(\u)}X_{\ga,\ov\ga}=
\Im\,L_{\ga,\ov\ga;\u} +T_{\ev_{\ga,\ov\ga}(\u)}\De_{\ga,\ov\ga}\EE
for every $\u\!\in\!\fB_{\ga,\vp;\rho}^{\star}$.
Thus, the smooth map
$$\ev_{\ga,\ov\ga}\!:\fB_{\ga,\vp}^0\!\times\!
\prod_{v\in\ale_{\R}(\ga,\ov\ga)_0^c}\hspace{-.23in}\fB_v^* 
\times
\prod_{\{v,\si(v)\}\subset\ale_{\R}(\ga,\ov\ga)_0^c}
\hspace{-.41in}\fB_v^{\bu,*}\lra X_{\ga,\ov\ga}$$ 
is transverse to the submanifold $\De_{\ga,\ov\ga}$.
Along with~\eref{UfBlmm_e17}, the Implicit Function for Banach manifolds, 
and~\eref{UfBlmm_e15}, this implies  the lemma.\\

\noindent
For each edge~$e$ in $\E_{\R\C}(\ga,\ov\ga)$, $f_e\!\not\in\!\ale(\ga,\ov\ga)_0$.
For each edge~$e$ in 
\hbox{$\E_{\C}(\ga,\ov\ga)\!\cup\!\E_{\R\R}(\ga,\ov\ga)$},
there exists a flag $f_e\!\in\!e$ such that $f_e\!\not\in\!\ale(\ga,\ov\ga)_0$.
For each $v\!\in\!\Ver$, let
\begin{gather*}
\Fl_v^*=\ve^{-1}(v)\!\cap\!\big\{f_e\!:e\!\in\!\E_+(\ga,\ov\ga)\!\cup\!
\E_{\R\C}(\ga,\ov\ga)\!\cup\!\E_{\R\R}(\ga,\ov\ga)\big\}, \\
\Fl_{v;\R}^*=\Fl_v^*\!\cap\!S_{v;\R}(\ga),\qquad
\Fl_{v;\C}^*=\Fl_v^*\!\cap\!S_{v;\C}(\ga).
\end{gather*}
If $v\!\in\!\nV_{\C}(\ga)$, $\Fl_{v;\R}^*\!=\!\eset$ and $\Fl_{v;\C}^*\!=\!\Fl_v^*$.
By~\eref{Eplus_e},
\BE{UfBlmm_e23}\si(\Fl_v^*)\cap\Fl_{\si(v)}^*=\Fl_{v;\R}^*,\quad
\big\{f_e\!:e\!\in\!\E_+(\ga,\ov\ga)\!\cup\!
\E_{\R\C}(\ga,\ov\ga)\!\cup\!\E_{\R\R}(\ga,\ov\ga)\big\}=
\bigcup_{v\in\ale(\ga,\ov\ga)_0^c}\!\!\!\!\!\!\!\Fl^*_v\,.\EE

\vspace{.1in}

\noindent
Let $\u\!\in\!\fB_{\ga,\vp;\rho}^{\star}$ and
$$\pi_{\ga,\ov\ga}\!: T_{\ev_{\ga,\ov\ga}(\u)}X_{\ga,\ov\ga}\lra
\bigoplus_{e\in\E_+(\ga,\ov\ga)}\!\!\!\!\!\!T_{\ev_{f_e}(u_{\ve(f_e)})}X
\oplus
\bigoplus_{e\in\E_{\R\C}(\ga,\ov\ga)}\!\!\!\!\!\!\!\!T_{\ev_{f_e}(u_{\ve(f_e)})}X
\oplus
\bigoplus_{e\in\E_{\R\R}(\ga,\ov\ga)}\!\!\!\!\!\!\!T_{\ev_{f_e}(u_{\ve(f_e)})}X^{\phi}$$
be the projection on the components indexed by the flags of the form~$f_e$.
By the first statement in~\eref{UfBlmm_e23}, the homomorphisms
$$\bigoplus_{f\in\Fl_{v;\C}^*}\!\!\!\!L_f\oplus
\bigoplus_{f\in\Fl_{v;\R}^*}\!\!\!\!L_f\!:  \Ga(\u_v)\!\equiv\!\Ga(u_v)^{\phi,\si} 
\lra  \bigoplus_{f\in\Fl_{v;\C}^*}\!\!\!\!T_{u_v(z_f)}X
~\oplus\bigoplus_{f\in\Fl_{v;\R}^*}\!\!\!\!T_{u_v(z_f)}X^{\phi}$$
with $v\!\in\!\ale_{\R}(\ga,\ov\ga)_0^c$ and the homomorphisms
$$\bigoplus_{f\in\Fl_v^*}\!\!\!L_f\!: 
\Ga(\u_v)\!\equiv\!\Ga\big(u_v\!\sqcup\!u_{\si(v)}\big)^{\phi,\si}
\lra \bigoplus_{f\in\Fl_v^*}\!\!\!T_{u_v(z_f)}X $$
with $v\!\in\!\ale_{\C}(\ga,\ov\ga)_0^c$ are surjective.
Thus, the restriction of~$\pi_{\ga,\ov\ga}$ to the first subspace 
on the right-hand side of~\eref{UfB_e5} is surjective as well.
This in turn implies~\eref{UfB_e5}.
\end{proof}

\begin{proof}[{\bf{\emph{Proof of Proposition~\ref{UfM_prp}}}}]
Let
\BE{UfM_e3}\fF_v\lra\cH_{g,l;k}^{\om,\phi}(X)\!\times\!\fB_v,~~v\!\in\!\ale_{\R}(\ga,\ov\ga)_0^c,
\quad\hbox{and}\quad
\fF_v^{\bu}\lra\cH_{g,l;k}^{\om,\phi}(X)\!\times\!\fB_v^{\bu},~~v\!\in\!\ale_{\C}(\ga,\ov\ga)_0^c,\EE
be the bundles with fibers $\Ga^{0,1}_J(\u_v)$ and $\Ga^{0,1}_J(\u_v^{\bu})$, respectively.
Denote~by
\begin{alignat*}{2}
\pi_v\!:\cH_{g,l;k}^{\om,\phi}(X)\!\times\!
\prod_{v\in\nV_{\R}(\ga)}\!\!\!\!\!\!\fB_v \times\!\!
\prod_{\{v,\si(v)\}\subset\nV_{\C}(\ga)}\!\!\!\!\!\!\!\!\!\!\!\!\!\!\fB_v^{\bu}
&\lra \cH_{g,l;k}^{\om,\phi}(X)\!\times\!\fB_v, &\quad &v\!\in\!\ale_{\R}(\ga,\ov\ga)_0^c,\\
\pi_v^{\bu}\!:\cH_{g,l;k}^{\om,\phi}(X)\!\times\!
\prod_{v\in\nV_{\R}(\ga)}\!\!\!\!\!\!\fB_v \times\!\!
\prod_{\{v,\si(v)\}\subset\nV_{\C}(\ga)}\!\!\!\!\!\!\!\!\!\!\!\!\!\!\fB_v^{\bu}
&\lra \cH_{g,l;k}^{\om,\phi}(X)\!\times\!\fB_v^{\bu}, &\quad &v\!\in\!\ale_{\C}(\ga,\ov\ga)_0^c,
\end{alignat*}
the projection maps.
Define sections~$\dbar_v$ and~$\dbar_v^{\bu}$ of~\eref{UfM_e3} 
by~\eref{dbardfn_e} with $(\u,\Si)$ replaced
by $(\u_v,\Si_v)$ and $(\u_v\!\sqcup\!\u_{\si(v)},\Si_v\!\sqcup\!\Si_{\si(v)})$, 
respectively.\\

\noindent
The section~$\dbar$ in~\eref{dbardfn_e} is the restriction of the section
\BE{UfM_e5}\begin{split}
&\dbar\!\equiv\!\bigoplus_{v\in\ale_{\R}(\ga,\ov\ga)_0^c}
\hspace{-.19in}\pi_v^*\dbar_v
\oplus\bigoplus_{\{v,\si(v)\}\subset\ale_{\C}(\ga,\ov\ga)_0^c}
\hspace{-.39in}\pi_v^{\bu\,*}\dbar_v^{\bu}\!:\\
&\hspace{.5in}
\cH_{g,l;k}^{\om,\phi}(X)\!\times\!
\prod_{v\in\nV_{\R}(\ga)}\!\!\!\!\!\!\fB_v \times\!\!
\prod_{\{v,\si(v)\}\subset\nV_{\C}(\ga)}\!\!\!\!\!\!\!\!\!\!\!\!\!\!\fB_v^{\bu}
\lra \bigoplus_{v\in\ale_{\R}(\ga,\ov\ga)_0^c}\hspace{-.19in}\pi_v^*\fF_v
\oplus
\bigoplus_{\{v,\si(v)\}\subset\ale_{\C}(\ga,\ov\ga)_0^c}\hspace{-.39in}\pi_v^{\bu\,*}\fF_v^{\bu}
\end{split}\EE
to the base of the bundle~\eref{fFbndldfn_e}.
By Lemma~\ref{UfB_lmm}, the latter is a separable Banach submanifold of
the base of~\eref{UfM_e5}.
We show that the bundle section~\eref{dbardfn_e} is transverse to the zero set.
Fix an element $(J,\nu;\u)$ of its zero set~\eref{UfMdfn_e} with~$\u$ 
as in~\eref{bfudfn_e}.\\

\noindent
For each $v\!\in\!\ale_{\R}(\ga,\ov\ga)_0^c$, let
\begin{gather*}
D_{J,\nu;\u_v}^0\dbar\!:
T_{(J,\nu)}\cH_{g,l;k}^{\om,\phi}(X)\!\oplus\!\Ga_0(\u_v)\lra\Ga^{0,1}_J(\u_v),\\
D_{J,\nu;\u_v}^0\dbar(A,\nu';\xi)=D_{J,\nu_{\ga,\vp;\rho};\u_v}^0\xi
+\frac12A\!\circ\!\tnd u_v\!\circ\!\fJ
-\{q_{\ga,\vp;v}\!\times\!u_v\}^*\nu',
\end{gather*}
with the last term above defined to be~0 if $v\!\in\!\ale_{\R}(\ga,\ov\ga)_{\bu}$.
We note~that 
\BE{UfM_e15R}\begin{aligned}
A\!\circ\!\tnd u_v\!\circ\!\fJ&=0 &\quad
&\hbox{if}\quad u_v(\Si_v)\!\cap\!\supp(A)=\eset,\\
\{q_{\ga,\vp;v}\!\times\!u_v\}^*\nu'&=0 &\quad
&\hbox{if}\quad
\big(q_{\ga,\vp;v}(\Si_v)\!\times\!X\big)\!\cap\!\supp(\nu')=\eset.
\end{aligned}\EE
For each $v\!\in\!\ale_{\C}(\ga,\ov\ga)_0^c$, let
\begin{gather*}
D_{J,\nu;\u_v^{\bu}}^0\dbar\!:
T_{(J,\nu)}\cH_{g,l;k}^{\om,\phi}(X)\!\oplus\!\Ga_0(\u_v^{\bu})\lra\Ga^{0,1}_J(\u_v^{\bu}),\\
D_{J,\nu;\u_v^{\bu}}^0\dbar(A,\nu';\xi)=D_{J,\nu;\u_v^{\bu}}^0\xi
+\frac12\big(A\!\circ\!\tnd u_v\!\circ\!\fJ,A\!\circ\!\tnd u_{\si(v)}\!\circ\!\fJ\big)
-\big\{q_{\ga,\vp_{\ga,\vp;\rho};v}^{\bu}\!\times\!(u_v\!\sqcup\!u_{\si(v)})\big\}^*\nu',
\end{gather*}
with the last term above defined to be~0 if $v\!\in\!\ale_{\C}(\ga,\ov\ga)_{\bu}$.
We note~that 
\BE{UfM_e15C}\begin{aligned}
\big(A\!\circ\!\tnd u_v\!\circ\!\fJ,A\!\circ\!\tnd u_{\si(v)}\!\circ\!\fJ\big)&=0 
&\quad&\hbox{if}\quad u_v(\Si_v)\!\cap\!\supp(A)=\eset,\\
\{q_{\ga,\vp;v}^{\bu}\!\times\!u_v\}^*\nu'&=0 &\quad
&\hbox{if}\quad
\big(q_{\ga,\vp;v}^{\bu}(\Si_v)\!\times\!X\big)\!\cap\!\supp(\nu')=\eset.
\end{aligned}\EE

\vspace{.2in}

\noindent
The homomorphisms $D_{J,\nu;\u_v}^0\dbar$ and $D_{J,\nu;\u_v^{\bu}}^0\dbar$
are the restrictions of the linearizations of~$\dbar_v$ at~$(J,\nu;\u_v)$
and of~$\dbar_v^{\bu}$ at~$(J,\nu;\u_v^{\bu})$ to
\begin{equation*}\begin{split}
T_{_{(J,\nu)}}\cH_{g,l;k}^{\om,\phi}(X)\!\oplus\!\Ga_0(\u_v)
&\subset T_{(J,\nu;\u_v)}\big(\cH_{g,l;k}^{\om,\phi}(X)\!\times\!\fB_v) 
\qquad\hbox{and}\\
T_{(J,\nu)}\cH_{g,l;k}^{\om,\phi}(X)\!\oplus\!\Ga_0(\u_v^{\bu})
&\subset T_{(J,\nu;\u_v^{\bu})}\big(\cH_{g,l;k}^{\om,\phi}(X)\!\times\!\fB_v^{\bu}),
\end{split}\end{equation*}
respectively.
Since $\Ga_0(\u)\!\subset\!T_{\u}\fB_{\ga,\vp;\rho}^*$,
it is sufficient to show that the homomorphism
\BE{UfM_e17}\begin{split}
&\bigoplus_{v\in\ale_{\R}(\ga,\ov\ga)_0^c}\hspace{-.19in}\pi_v^*D_{J,\nu;\u_v}^0\dbar
\oplus
\bigoplus_{\{v,\si(v)\}\subset\ale_{\C}(\ga,\ov\ga)_0^c}\hspace{-.39in}
\pi_v^{\bu\,*}D_{J,\nu;\u_v^{\bu}}^0\dbar\!:\\
&\hspace{.7in}
T_{(J,\nu)}\cH_{g,l;k}^{\om,\phi}(X)\!\oplus\!\Ga_0(\u)
\lra \bigoplus_{v\in\ale_{\R}(\ga,\ov\ga)_0^c}\hspace{-.19in}\Ga^{0,1}_J(\u_v)
\oplus
\bigoplus_{\{v,\si(v)\}\subset\ale_{\C}(\ga,\ov\ga)_0^c}\hspace{-.39in}
\Ga^{0,1}_J(\u_v^{\bu})
\end{split}\EE
is surjective.\\

\noindent
Since $\u\!\in\!\fB_{\ga,\vp;\rho}^*$,
the subsets $u_v(\Si_v)\!\subset\!X$ with 
$v\!\in\!\ale(\ga,\ov\ga)_{\bu}$ are distinct. 
There thus exist $\phi$-invariant open subsets $W_v\!\subset\!X$
with \hbox{$v\!\in\!\ale(\ga,\ov\ga)_{\bu}$} such~that 
\begin{gather}
\label{UfM_e19a1}
u_v(\Si_v)\cap W_v\neq\eset\qquad\forall\,v\!\in\!\ale(\ga,\ov\ga)_{\bu}, \\
\label{UfM_e19a2}
u_{v_1}(\Si_{v_1})\cap W_{v_2}=\eset,~~
W_{v_1}\cap W_{v_2}=\eset \qquad\forall\,v_1,v_2\!\in\!\ale(\ga,\ov\ga)_{\bu},
\,v_1\!\neq\!v_2,\si(v_2).
\end{gather}
For $v\!\in\!\nV_{\C}(\ov\ga)$, let $q_{\ga,\vp;v}\!=\!q_{\ga,\vp;v}^{\bu}|_{\Si_v}$.
The subsets $q_{\ga,\vp;v}(\Si_v)\!\subset\!\R\wt\cU_{g,l;k}$ with $v\!\in\!\nV(\ov\ga)$
are also distinct.
There thus exist $\wt\si_{\R}$-invariant open subsets \hbox{$W_v\!\subset\!\R\wt\cU_{g,l;k}$}
with $v\!\in\!\nV(\ov\ga)$ such~that
\begin{gather}
\label{UfM_e19b1}
q_{\ga,\vp;v}(\Si_v)\cap W_v\neq\eset \qquad\forall\,v\!\in\!\nV(\ov\ga),\\
\label{UfM_e19b2}
q_{\ga,\vp;v_1}(\Si_{v_1})\cap W_{v_2}=\eset,~~W_{v_1}\cap W_{v_2}=\eset
\qquad\forall\,v_1,v_2\!\in\!\nV(\ov\ga),
\,v_1\!\neq\!v_2,\si(v_2).
\end{gather}
Define
\begin{alignat*}{2}
T_v\cH&=\big\{A\!\in\!T_J\cJ_{\om}^{\phi}\!:\supp(A)\!\subset\!W_v\big\}
&\qquad&\forall\,v\!\in\!\ale(\ga,\ov\ga)_{\bu}, \\
T_v\cH&=\big\{\nu'\!\in\!\Ga^{0,1}_{g,l;k}(X;J)^{\phi}\!:
\supp(\nu')\!\subset\!W_v\!\times\!X\big\}
&\qquad&\forall\,v\!\in\!\nV(\ov\ga)\,.
\end{alignat*}

\vspace{.2in}

\noindent
By Lemmas~\ref{RTtransC_lmm} and~\ref{RTtransR_lmm}, \eref{UfM_e19a1}, and~\eref{UfM_e19b1}, 
\BE{UfM_e25}\begin{aligned}
D_{J,\nu;\u_v}^0\dbar\big(T_v\cH\!\oplus\!\Ga_0(\u_v)\big)&=\Ga^{0,1}_J(\u_v)
&~~&\forall~v\!\in\!\ale_{\R}(\ga,\ov\ga)_0^c,\\
D_{J,\nu;\u_v^{\bu}}^0\dbar\big(T_v\cH\!\oplus\!\Ga_0(\u_v^{\bu})\big)&=\Ga^{0,1}_J(\u_v^{\bu})
&~~&\forall~v\!\in\!\ale_{\C}(\ga,\ov\ga)_0^c\,.
\end{aligned}\EE
By definition,
\BE{UfM_e27}\begin{split}
&D_{J,\nu;\u_{v_1}}^0\!\dbar\big(\Ga_0(\u_{v_1'})\big),
D_{J,\nu;\u_{v_1}}^0\!\dbar\big(\Ga_0(\u_{v_2}^{\bu})\big),
D_{J,\nu;\u_{v_2}^{\bu}}^0\!\dbar\big(\Ga_0(\u_{v_1})\big),
D_{J,\nu;\u_{v_2}^{\bu}}^0\!\dbar\big(\Ga_0(\u_{v_2'}^{\bu})\big)=\{0\},\\
&\hspace{1in}
\forall~v_1,v_1'\!\in\!\ale_{\R}(\ga,\ov\ga)_0^c,~v_2,v_2'\!\in\!\ale_{\C}(\ga,\ov\ga)_0^c,~
v_1\!\neq\!v_1',~v_2\!\neq\!v_2',\si(v_2').
\end{split}\EE
By~\eref{UfM_e15R}, \eref{UfM_e15C}, and 
the first statements in~\eref{UfM_e19a2} and~\eref{UfM_e19b2},
\BE{UfM_e29}\begin{split}
&\hspace{1.5in}
D_{J,\nu;\u_{v_1}}^0\dbar\big(T_v\cH\big),D_{J,\nu;\u_{v_2}^{\bu}}^0\dbar\big(T_v\cH\big)
=\{0\}\\
&\forall~v_1\!\in\!\ale_{\R}(\ga,\ov\ga)_0^c,~
v_2\!\in\!\ale_{\C}(\ga,\ov\ga)_0^c,~v\!\in\!\ale(\ga,\ov\ga)_0^c,~v\!\neq\!v_1,v_2,\si(v_2),~
(v_i,v)\not\in\nV(\ov\ga)\!\times\!\ale(\ga,\ov\ga)_{\bu}.
\end{split}\EE
By the last statements in~\eref{UfM_e19a2} and~\eref{UfM_e19b2},
\BE{UfM_e31}
\bigoplus_{v\in\ale_{\R}(\ga,\ov\ga)_0^c}\hspace{-.19in}\big(T_v\cH\!\oplus\!\Ga_0(\u_v)\big)
\oplus
\bigoplus_{\{v,\si(v)\}\subset\ale_{\C}(\ga,\ov\ga)_0^c}\hspace{-.39in}
\big(T_v\cH\!\oplus\!\Ga_0(\u_v^{\bu})\big)
\subset T_{(J,\nu)}\cH_{g,l;k}^{\om,\phi}(X)\!\oplus\!\Ga_0(\u).\EE
By~\eref{UfM_e25}-\eref{UfM_e31}, the homomorphism~\eref{UfM_e17} is surjective.
This establishes the first claim of the proposition.\\

\noindent
Let $G_v$ for $v\!\in\!\nV_{\C}(\ga)$ with $|\ve^{-1}(v)|\!\le\!2$ and 
$G_{\rho;v}$ for $v\!\in\!\nV_{\R}(\ga)$ with $|\ve^{-1}(v)|\!\le\!2$
be as in Section~\ref{RTstr_subs}.
Denote by $G_v$ for $v\!\in\!\nV_{\C}(\ga)$ with $|\ve^{-1}(v)|\!\ge\!3$ and 
$G_{\rho;v}$ for $v\!\in\!\nV_{\R}(\ga)$ with $|\ve^{-1}(v)|\!\ge\!3$
the trivial groups.
Thus,
\begin{alignat*}{2}
\dim_{\R}\!\R\wt\cM_{\ga;v}&=3\big(\fg(v)\!-\!1\big)+\big|\ve^{-1}(v)\big|+
\dim_{\R}\!G_{\rho;v} &\qquad&\forall\,v\!\in\!\nV_{\R}(\ga),\\
\dim_{\R}\!\wt\cM_{\ga;v}^{\bu}&=6\big(\fg(v)\!-\!1\big)+2\big|\ve^{-1}(v)\big|+
\dim_{\R}\!G_v &\qquad&\forall\,v\!\in\!\nV_{\C}(\ga).
\end{alignat*}
Since $\fg(\si(v))\!=\!\fg(v)$ and $|\ve^{-1}(\si(v))|\!=\!|\ve^{-1}(v)|$, it follows that 
\BE{UfM_e35}\begin{split}
\sum_{v\in\ale_{\R}(\ga,\ov\ga)_0^c}\hspace{-.22in}\dim_{\R}\!\R\wt\cM_{\ga;v}
+\sum_{\{v,\si(v)\}\subset\ale_{\C}(\ga,\ov\ga)_0^c}\hspace{-.42in}
\dim_{\R}\!\wt\cM_{\ga;v}^{\bu}
&=\sum_{v\in\Ver}\!\!\Big(3\big(\fg(v)\!-\!1\big)\!+\!\big|\ve^{-1}(v)\big|\Big)
+\dim_{\R}\!G_{\ga;\rho}\\
&=3(g'\!-\!1)\!+\!2l\!+\!k\!-\!|\ga|\!+\!\dim_{\R}\!G_{\ga;\rho};
\end{split}\EE
the second equality above follows from~\eref{stgacnd_e} with $\ov\ga$ replaced 
by $\ga\!\in\!\cA_{g',l;k}^{\phi}(B')$.\\

\noindent
Let $v\!\in\!\ale_{\R}(\ga,\ov\ga)_0^c$. By the proof of \cite[Proposition~3.6]{XCapsSetup},
the restriction
$$D_{J,\nu;\u_v}\!:\Ga(\u_v)\lra \Ga^{0,1}_{J,\fJ}(\u_v)$$
of $D_{J,\nu;u_v}$ is a Fredholm operator of index 
$$\ind_{\R}\!D_{J,\nu;\u_v}=\blr{c_1(TX),\fd(v)}+n\big(1\!-\!\fg(v)\big).$$
Thus, the restriction
$$\wt{D}_{J,\nu;\u_v}\!:T_{\u_v}\fB_v\lra\Ga^{0,1}_{J,\fJ}(\u_v)$$
of the linearization  of~$\dbar_v$
at $(J,\nu_{\ga,\vp;\rho};\u_v)$ is   a Fredholm operator of index 
\BE{UfM_e35R}\ind_{\R}\!\wt{D}_{J,\nu;\u_v}=
\blr{c_1(TX),\fd(v)}+n\big(1\!-\!\fg(v)\big)+\dim_{\R}\!\R\wt\cM_{\ga;v}\,.\EE
This formula is also valid for $v\!\in\!\ale_{\R}(\ga,\ov\ga)_0$ if $\dbar_v$
is replaced by the section of the  rank~0 bundle.\\

\noindent
Let $v\!\in\!\ale_{\C}(\ga,\ov\ga)_0^c$. By the reasoning at the beginning of the proof of
Lemma~\ref{RTtransC_lmm} and \cite[Theorem~C.1.10]{MS}, the restriction
$$D_{J,\nu;\u_v}^{\bu}\!:\Ga(\u_v^{\bu})\lra \Ga^{0,1}_{J,\fJ}(\u_v^{\bu})$$
of $D_{J,\nu;u_v}^{\bu}$ is a Fredholm operator of index 
$$\ind_{\R}\!D_{J,\nu;\u_v}^{\bu}=2\blr{c_1(TX),\fd(v)}+2n\big(1\!-\!\fg(v)\big).$$
Thus, the restriction
$$\wt{D}_{J,\nu;\u_v^{\bu}}\!:T_{\u_v^{\bu}}\fB_v^{\bu}\lra\Ga^{0,1}_{J,\fJ}(\u_v^{\bu})$$
of the linearization  of $\dbar_v^{\bu}$
at $(J,\nu_{\ga,\vp;\rho};\u_v)$ is   a Fredholm operator of index 
\BE{UfM_e35C}\ind_{\R}\!\wt{D}_{J,\nu;\u_v^{\bu}}=
2\blr{c_1(TX),\fd(v)}+2n\big(1\!-\!\fg(v)\big)+\dim_{\R}\!\wt\cM_{\ga;v}^{\bu}\,.\EE
This formula is also valid for $v\!\in\!\ale_{\C}(\ga,\ov\ga)_0$ if $\dbar_v^{\bu}$
is replaced by the section of the  rank~0 bundle.\\

\noindent
By~\eref{UfM_e35}-\eref{UfM_e35C},  
\eref{stgacnd_e} with $\ov\ga$ replaced $\ga\!\in\!\cA_{g',l;k}^{\phi}(B')$,
and \eref{gafdcond_e} with $B$ replaced by~$B'$, 
the restriction
$$\wt{D}_{J,\nu;\u}\!:
\bigoplus_{v\in\ale_{\R}(\ga,\ov\ga)_0^c}\hspace{-.19in}T_{\u_v}\fB_v
\oplus
\bigoplus_{\{v,\si(v)\}\subset\ale_{\C}(\ga,\ov\ga)_0^c}\hspace{-.39in}
T_{\u_v^{\bu}}\fB_v^{\bu}\lra\Ga^{0,1}_{J,\fJ}(\u)$$
of the linearization  of~$\dbar$
at $(J,\nu;\u)$ is   a Fredholm operator of index 
\begin{equation*}\begin{split}
\ind_{\R}\!\wt{D}_{J,\nu;\u}&=
\blr{c_1(TX),B'}\!+\!(n\!-\!3)(1\!-\!g')\!+\!2l\!+\!k\!
+(n\!-\!1)|\ga|\!+\!\dim_{\R}\!G_{\ga;\rho}\\
&=\dim_{g',l;k}(B')+(n\!-\!1)|\ga|+\dim_{\R}\!G_{\ga;\rho}\,.
\end{split}\end{equation*}
Along with Lemma~\ref{UfB_lmm}, this implies that the restriction
\BE{UfM_e39}\wt{D}_{J,\nu;\u}'\!:T_{\u}\fB_{\ga,\vp;\rho}^*\lra \Ga^{0,1}_{J,\fJ}(\u)\EE
of the linearization  of~\eref{dbardfn_e}
at $(J,\nu;\u)$ is  a Fredholm operator of index~\eref{UfMprp_e2}.
The last claim of the proposition follows from this conclusion by
the reasoning at the beginning of the proof of \cite[Theorem~3.1.6(ii)]{MS}.
\end{proof}

\noindent
By the reasoning in the proof of \cite[Theorem~3.1.6(ii)]{MS},
$(J,\nu)$ is a regular value of~$\pi$ if and only if 
the operator~\eref{UfM_e39} is surjective for every element~$\u$
of the preimage
\BE{wtfMsm_e}\wt\fM_{\ga,\vp;\rho}^*(J,\nu)\equiv\pi^{-1}(J,\nu)\subset
\fU\wt\fM_{\ga,\vp;\rho}^*(J,\nu)\EE
of $(J,\nu)$.
Suppose $m\!-\!2>\!\ind_{\R}\pi$.
By the Sard-Smale Theorem, the set~$\wh\cH^m$ of regular values of~$\pi$ 
is then a Baire subset of second category in (the $C^m$-completion~of) 
$\cH_{g,l;k}^{\om,\phi}(X)$ in the $C^m$-topology.
Along with Taubes' argument in the proof of \cite[Theorem~3.1.6(ii)]{MS},
this implies that the subset 
$$\wh\cH_{g,l;k}^{\om,\phi}(X)\subset\cH_{g,l;k}^{\om,\phi}(X)$$
of smooth pairs $(J,\nu)$ so that the operator~\eref{UfM_e39} is surjective 
for every element~$\u$ of $\wt\fM_{\ga,\vp;\rho}^*(J,\nu)$ is Baire of second category
in the space of all smooth pairs~$(J,\nu)$ in the $C^{\i}$-topology.
For every pair~$(J,\nu)$ in this subset, the left-hand side in~\eref{wtfMsm_e}
is a smooth submanifold of the right-hand side of~\eref{UfMprp_e1}  
of dimension~\eref{UfMprp_e2}.
The group~$G_{\ga;\rho}^0$ acts smoothly and freely on this submanifold.
The smooth structure on the quotient  descends to a smooth structure 
on~$\cZ_{\ga,\vp}^*(J,\nu)$.\\

\vspace{.2in}

\noindent
{\it Institut de Math\'ematiques de Jussieu - Paris Rive Gauche,
Universit\'e Pierre et Marie Curie, 
4~Place Jussieu,
75252 Paris Cedex 5,
France\\
penka.georgieva@imj-prg.fr}\\

\noindent
{\it Department of Mathematics, Stony Brook University, Stony Brook, NY 11794\\
azinger@math.stonybrook.edu}


\end{document}